\documentclass[11pt, oneside]{book}
\usepackage[verbose,tmargin=2.5cm,bmargin=2.5cm,lmargin=2.5cm,rmargin=2.5cm]{geometry}
\usepackage{amsfonts,amssymb,amsmath,amsthm,amstext}
\usepackage{enumitem}
\usepackage{graphicx}
\usepackage{caption}
\usepackage{keytheorems}
\usepackage[title]{appendix}

\usepackage{xcolor}
\usepackage{comment}
\usepackage{mathtools}
\usepackage{subcaption}
\usepackage{float}
\usepackage{hyperref}
\usepackage{cleveref}
\usepackage{tikz}
\DeclareUnicodeCharacter{2061}{}

\newtheorem{thm}{Theorem}[section]
\newtheorem{claim}[thm]{Claim}
\newtheorem{proposition}[thm]{Proposition}
\newtheorem{lemma}[thm]{Lemma}
\newtheorem{corollary}[thm]{Corollary}

\theoremstyle{definition}
\newtheorem{example}[thm]{Example}
\newtheorem{obs}[thm]{Observation}
\newtheorem{question}[thm]{Question}
\newtheorem{Oquestion}[thm]{Open Question}
\newtheorem{defn}[thm]{Definition}

\newtheorem{remark}[thm]{Remark}

\usepackage{bbm}
\usepackage{comment} 
\usepackage[normalem]{ulem} 

\binoppenalty=\maxdimen
\relpenalty=\maxdimen

\date{\today}

\begin{document}
\begin{titlepage}
		\begin{center}
			\hspace{1cm}
			\vspace{6.5cm}
					
			\Huge
			\textbf{Percolation on Finite Graphs}
			\LARGE 
			\textbf{(0366-4118)}
			
			\vspace{1cm}
			\Large
			Prof. Michael Krivelevich 
			
			\vspace{0.5cm}
			Tel Aviv University, Spring 2026
			
		\end{center}
	\end{titlepage}

	\newpage
	
	\noindent \hrulefill \newline\newline\newline
	
	\noindent The following notes were written by the participants of the course, compiled by Itay Markbreit and are based on the lectures of a course on percolation on finite graphs, 
	given by Prof. Michael Krivelevich at the School of Mathematical Sciences, Tel Aviv University, Spring 2026. 
	Despite our best efforts, there are probably still some typos and mistakes left. Hence, corrections and other feedback would be greatly appreciated and can be
	sent to krivelev@tauex.tau.ac.il. \\ \\
    \textbf{Acknowledgements:} We would like to thank Michael Krivelevich for his corrections and invaluable suggestions for improvement throughout the preparation of these notes.

\tableofcontents
\pagebreak


\newcommand{\heading}[6]{
	\renewcommand{\thepage}{\arabic{page}}
	\noindent
	\begin{center}
		\framebox{
			\vbox{
				\hbox to 6.43in { \textbf{#2} \hfill #3 }
				\vspace{4mm}
				\hbox to 5.98in { {\Large \hfill #6  \hfill} }
				\vspace{2mm}
				\hbox to 6.43in { \textit{Lecturer: #4 \hfill #5} }
			}
		}
	\end{center}
	\vspace*{4mm}
}

\newcommand{\lecture}[4]{\heading{#1}{Percolation on Finite Graphs, Tel Aviv Univ., Spring 2026}{#2}{Prof. Michael Krivelevich}{Scribe: #4}{#3}}

\newcommand{\lecturenum}{1} 
\newcommand{\lecturedate}{April 12, 2026} 
\newcommand{\lecturetitle}{Lecture 1} 
\newcommand{\scribename}{Itay Markbreit} 
\phantomsection
\addcontentsline{toc}{chapter}{Lecture 1}
\lecture{\lecturenum}{\lecturedate}{\lecturetitle}{\scribename}
\setcounter{chapter}{1}
\section{Standard inequalities}

\begin{enumerate}
    \item For every $x\in \mathbb{R}$, $1+x\le e^x$.
    \item For every small enough $x>0$, $1+x\le e^{x-\frac{x^2}{3}}$.
    \item Recall that: For every $0\le k\le n$, $\binom{n}{k}=\frac{n!}{k!(n-k)!}$. For every $1\le k\le n$:
    \begin{align*}
        \left(\frac{n}{k}\right)^k\le \binom{n}{k}\le \sum_{i=0}^k\binom{n}{i}\le \left(\frac{en}{k}\right)^k.
    \end{align*}
    \item Stirling's formula:
    \begin{align*}
        \lim_{n\xrightarrow{}\infty}\frac{n!}{\sqrt{2\pi n}\cdot \left(\frac{n}{e}\right)^n}=1.
    \end{align*}
    \item Markov's inequality: Let $X$ be a non-negative random variable such that $\mathbb{E}[X]$ exists. Then, for every $t>0$,
    \begin{align*}
        \mathbb{P}(X\ge t)\le \frac{\mathbb{E}[X]}{t}. 
    \end{align*}
    \item Chebyshev's inequality: Let $X$ be a random variable such that $\mathbb{E}[X]$ and $\mathbb{E}[X^2]$ exist. Then, for every $t>0$,
    \begin{align*}
        \mathbb{P}(|X-\mathbb{E}[X]|\ge t)\le \frac{\text{Var}(X)}{t^2}. 
    \end{align*}
    \item Chernoff's inequality: Recall the notion of the binomial distribution. Let $\{X_i\}_{i=1}^n$ be i.i.d. Bernoulli$(p)$ random variables, that is, for every $1\le i\le n$, $\mathbb{P}(X_i=1)=p$ and $\mathbb{P}(X_i=0)=1-p$. Set $X=\sum_{i=1}^nX_i$. $X$ is a binomially distributed random variable with parameters $n$ and $p$, and we denote $X\sim Bin(n,p)$. Notice that $\mathbb{E}[X]=np$ and $\text{Var}(X)=npq$ where $q=1-p$. Chernoff's inequalities say that
    \begin{align*}
        \mathbb{P}(X\le np-a)\le e^{-\frac{a^2}{2np}},\quad \mathbb{P}(X\ge np+a)\le e^{-\frac{a^2}{2np}+\frac{a^3}{2(np)^2}},
    \end{align*}
    for every $a>0$.
\end{enumerate}
\underline{Basic notations and assumption in the course}
\begin{enumerate}
    \item $n\xrightarrow{}\infty$ (but finite).
    \item Let $\bar{\Omega}=\left(\Omega_n\right)_{n=1}^\infty$ be a sequence of probability spaces and let $\bar{A}=\left(A_n\right)_{n=1}^\infty$ such that for every $n\in\mathbb{N}$, $A_n\subseteq \Omega_n$ is an event. We say that $\bar{A}$ occurs \textit{with high probability} (abbv. whp) if 
    \begin{align*}
        \lim_{n\xrightarrow{}\infty}\mathbb{P}_{\Omega_n}(A_n)=1.
    \end{align*}
\end{enumerate}
\begin{example}
    When we say that $G\left(n,\frac{1}{2}\right)$ is Hamiltonian (=contains a Hamilton cycle), formally, we mean the following: $\Omega_n=G\left(n,\frac{1}{2}\right)$. $A_n$ is the event where $G\sim G\left(n,\frac{1}{2}\right)$ is Hamiltonian and 
    \begin{align*}
        \lim_{n\xrightarrow{}\infty}\mathbb{P}\left(G\sim G\left(n,\frac{1}{2}\right)\text{ is Hamiltonian}\right)=1.
    \end{align*}
\end{example}

\section{The Gilbert Model \texorpdfstring{$G(n,p)$}{G(n,p)}}

In the $G(n, p)$ model, we consider a set of $n$ labeled vertices, denoted as $[n] = \{1, \dots, n\}$. The total number of possible edges in a complete graph $K_n$ is given by:
\begin{equation*}
N \coloneqq \binom{n}{2} = |E(K_n)|.
\end{equation*}
\begin{defn}
    Let $0 \le p\coloneqq p(n) \le 1$. We say that $G \sim G(n, p)$ if $V(G)=[n]$ consists of $n$ labeled vertices, and each possible edge $(i, j)$ exists independently with probability $p$.
\end{defn}
\begin{remark}
    The fact that the vertices are labeled implies for example that
    \begin{figure}[H]
    \centering
    \includegraphics[width=0.5\textwidth]{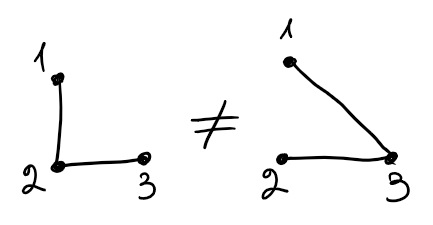}
    \caption{The order matters}
\end{figure}
\end{remark}
\noindent Equivalently,
\begin{defn}[Equivalent Definition]
    Let $\Omega_n = \{G = (V, E) \mid V = [n]\}$. For every $G\in \Omega_n$:
\begin{equation*}
\mathbb{P}(G) = p^{|E(G)|} (1-p)^{N - |E(G)|}.
\end{equation*}
\end{defn}
\begin{example}
    According to the following picture, $n=4$, $p=1/3$, and $|E(G)|=4$, and hence $\mathbb{P}(G) = \left(\frac{1}{3}\right)^4 \cdot \left(\frac{2}{3}\right)^2$.
    \begin{figure}[H]
        \centering
        \includegraphics[width=0.5\textwidth]{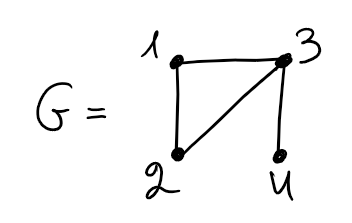}
        \caption{}
    \end{figure}
\end{example}

\begin{remark}
    Notice that $G(n,p)$ is a product space.
\end{remark}
\begin{remark}
    Assume that $G=([n],E)$ satisfies $|E(G)|=m$, then 
    \begin{align*}
        \mathbb{P}(G)=p^m(1-p)^{N-m},
    \end{align*}
    and hence the distribution $G(n,p)\big||E(G)|=m$ is a uniform distribution on graphs with vertex set $[n]$ and exactly $m$ edges.
\end{remark}
\begin{proposition}
    Let $G\sim G(n,p)$ and set $X = |E(G)|$. Then $X\sim Bin(N,p)$.
\end{proposition}
\begin{remark}
    $G(n,p)$ was introduced by Gilbert in 1959 and is often called a \textit{binomial random graph}.
\end{remark}
\begin{example}
    \underline{Special case - $p=\frac{1}{2}$}: In this case, $\Omega_n = \{G = (V, E) \mid V = [n]\}$, and for every $G\in \Omega_n$:
    \begin{equation*}
    \mathbb{P}(G) = \left(\frac{1}{2}\right)^N.
    \end{equation*}
    We get a uniform distribution on the space of all graphs with vertex set $[n]$. Hence, we use the following convention: when saying, for example: ``Almost every graph on $n$ vertices is connected", we mean that 
    \begin{align*}
        \lim_{n\xrightarrow{}\infty}\mathbb{P}\left(G\sim G\left(n,\frac{1}{2}\right)\text{ is connected}\right)=1.
    \end{align*}
\end{example}

\section{The Erd\H{o}s R\'enyi Model \texorpdfstring{$G(n,m)$}{G(n,m)}}

In this model we consider $\Omega = \{G = (V, E) \mid V = [n],|E(G)|=m\}$ and the distribution is uniform: for every $G\in \Omega$,
\begin{align*}
    \mathbb{P}(G)=\frac{1}{|\Omega|}=\frac{1}{\binom{N}{m}}.
\end{align*}
\begin{remark}
    $G(n,m)$ was introduced and studied by Erd\H{o}s and R\'enyi in 1960.
\end{remark}

\section{Random Graph Process}

Conceptually, we start with the empty graph $\bar{K_n}$ and finish with the complete graph $K_n$ in a random process.
Formally, 
\begin{defn}
    Given a permutation $\sigma\in S_N$ on the edges of $K_n$, we define the following \textit{graph process} $\tilde{G}\coloneqq\tilde{G}(\sigma)$: $\tilde{G}= (G_i)_{i=0}^N$, where $G_i$ is a graph with $V(G_i)=[n]$ and $E(G_i)=\{e_{\sigma(1)},e_{\sigma(2)},\ldots ,e_{\sigma(i)}\}$ (the first $i$ edges according to $\sigma$).
\end{defn}
\begin{example}
    Let $n=4$. $\tilde{G}=(G_i)_{i=1}^6$. According to the following figure,\\
    $G_4=\big ([4],\big\{(1,2),(1,3),(3,4),(2,3)\big\}\big )$.
    \begin{figure}[H]
        \centering
        \includegraphics[width=0.5\textwidth]{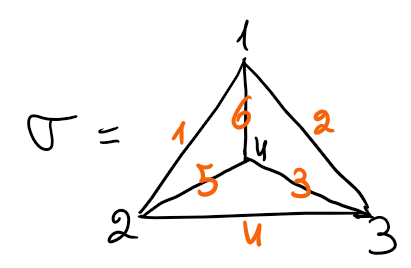}
        \caption{}
    \end{figure}
\end{example}

\noindent \textbf{Properties:}
\begin{enumerate}
    \item For every $0\le i\le N$, $|E(G_i)|=i$.
    \item $\bar{K_n}=G_0\subsetneq G_1\subsetneq \ldots\subsetneq G_N=K_n$ --- a nested sequence.
\end{enumerate}

\begin{defn}
    When $\sigma\in S_n$ is chosen at random, we say that $\tilde{G}=\tilde{G}(\sigma)$ is a \textit{random graph process}.
\end{defn}
\begin{defn}[Equivalent Definition]
    Start with $G_0=\bar{K_n}$. For every $1\le i\le N$, set $G_i=G_{i-1}\cup \{e\}$ where $e$ is a uniform chosen edge from the edges of $K_n$ that are not in $G_{i-1}$.
\end{defn}
\begin{remark}
    This model was introduced and studied by Erd\H{o}s and R\'enyi in the 1960s.
\end{remark}
\begin{proposition}
    Let $\tilde{G}$ be a random graph process. Then, for every $0\le m\le N$, the graph $G_m$ (which is called a \textit{snapshot} at time $m$) is distributed as $G(n,m)$.
\end{proposition}
\begin{proof}
    First, recall that $|E(G_m)|=m$, hence $G_m$ is distributed over graphs with vertex set $[n]$ and exactly $m$ edges. Moreover, for every $G=([n],E)$ with $|E|=m$,
    \begin{align*}
        \mathbb{P}(G_m=G)=\frac{m!(N-m)!}{N!}=\frac{1}{\binom{N}{m}},
    \end{align*}
    where the $m!$ term is for the arrangement of the edges of $G$ at the beginning of the permutation and the $(N-m)!$ term is for the arrangement of the edges that are not in $G$.
    Therefore, $G_m$ is uniformly distributed over graphs with vertex set $[n]$ and exactly $m$ edges, as desired.
\end{proof}
\begin{remark}
    The above proposition demonstrates that random graph process contains all the spaces $G(n, m)$ for $0 \le m \le N$. Therefore, it is more difficult (compared to $G(n, m)$ or $G(n, p)$) to study the random graph process, but results regarding the process can yield consequences on $G(n, m)$ or $G(n, p)$.
\end{remark}

\section{Multiple Exposure}

\begin{proposition}
    Suppose that $0 \le p, p_1, p_2, \dots, p_k \le 1$ satisfy $1-p = \prod_{i=1}^k (1-p_i)$. \\
Let $G \sim G(n, p)$ and suppose $G_i \sim G(n, p_i)$, $i=1, \dots, k$, are independent. Define: $G' = \bigcup_{i=1}^k G_i$. Then $G$ and $G'$ have the same distribution.
\end{proposition}

\begin{proof}
    First, notice that $G, G'$ are product spaces: the probability of a specific edge appearing in $G$ (as well as in $G'$) is independent of other edges. \\
    Therefore, it is sufficient to show that for any edge $e \in E(K_n)$, $\mathbb{P}(e \in G) = \mathbb{P}(e \in G')$.
    \begin{itemize}
        \item In $G$: $\mathbb{P}(e \in G) = p$.
        \item In $G' = G_1 \cup \dots \cup G_k$:
    \end{itemize}
    \[
    \mathbb{P}(e \in G') = 1 - \mathbb{P}(e \notin G') = 1 - \mathbb{P}\left(\bigwedge_{i=1}^k e \notin G_i\right) = 1 - \prod_{i=1}^k \mathbb{P}(e \notin G_i) = 1 - \prod_{i=1}^k (1-p_i) = p,
    \]
    where the last equality follows from the assumption. \\
    Hence, $\mathbb{P}(e \in G) = \mathbb{P}(e \in G')$, and thus $G$ and $G'$ have the same distribution.
\end{proof}

\begin{remark}
    It is customary (with some abuse of notation) to use $G(n,p)$ both as a distribution and as a graph drown from this distribution.
\end{remark}

\begin{remark}[Typical Scenario --- Sprinkling]
    $0 \le p_1, p_2 \le 1$ and $p_1 \gg p_2$. Suppose that $1-p = (1-p_1)(1-p_2)$. Let: $G \sim G(n, p)$, $G_1 \sim G(n, p_1)$, $G_2 \sim G(n, p_2)$. Then from the previous proposition: $G \sim G_1 \cup G_2$. Since $p_1 \gg p_2$, most of the edges are typically from $G_1$, and $G_2$ ``sprinkles'' edges onto $G_1$.
\end{remark}

\section{Monotone Properties and Threshold Function}

\begin{defn}
    Suppose that $\mathcal{A}$ is a property of graphs on vertex set $[n]$, namely $\mathcal{A} \subseteq \Omega_n =\allowbreak \{G = ([n], E)\}$. $\mathcal{A}$ is called \textit{monotone increasing} (or \textit{monotone}) if for every $G \in \mathcal{A}$ and for every $H \in \Omega_n$ such that $G \subseteq H$, it holds that $H \in \mathcal{A}$.
\end{defn}

\begin{defn}
    A property $\mathcal{A}$ is \textit{monotone decreasing} if for every $G \in \mathcal{A}$ and $H \subseteq G$, it holds that $H \in \mathcal{A}$.
\end{defn}

\begin{example}
    Connectivity is a monotone increasing property, whereas planarity is a monotone decreasing property.
\end{example}

\begin{defn}
    A property $\mathcal{A}$ is called non-trivial if:
\begin{enumerate}
    \item $\bar{K}_n \notin \mathcal{A}$,
    \item $K_n \in \mathcal{A}$.
\end{enumerate}
\end{defn}
\begin{defn}
    Let $\mathcal{A} \subseteq \Omega_n$ be a non-trivial monotone property and let $0 \le p = p(n) \le 1$. We define:
    \[
    f_{\mathcal{A}}(p) = \mathbb{P}(G \sim G(n, p) \in \mathcal{A}) = \sum_{G \in \mathcal{A}} \mathbb{P}(G).
    \]
\end{defn}

\noindent \textbf{Properties:}
\begin{enumerate}
    \item $f_{\mathcal{A}}(0) = 0, f_{\mathcal{A}}(1) = 1$.
    \item $f_{\mathcal{A}}(p)$ is a polynomial in $p$ of degree $\le \binom{n}{2}$, therefore it is a continuous function.
    \item $f_{\mathcal{A}}(p)$ is a strictly increasing function. (Exercise!)
\end{enumerate}

\begin{question}
    How fast does $f_{\mathcal{A}}(p)$ jump from nearly 0 to nearly 1?
\end{question}

\begin{proposition}\phantomsection\label{1: jump}
    Let $\epsilon, \delta > 0$ be real numbers. Let $\mathcal{A} \subseteq \Omega_n$ be an increasing monotone property of graphs on a set of vertices $[n]$. Assume there exists an integer $C > 0$ such that $(1-\epsilon)^C \le \delta$. Let $G \sim G(n, p_0)$ satisfy: $\mathbb{P}(G \in \mathcal{A}) \ge \epsilon$. Define $p_1 = C \cdot p_0$. \\
Then: $\mathbb{P}(G \sim G(n, p_1) \in \mathcal{A}) \ge 1 - \delta$.
\end{proposition}

\begin{figure}[H]
    \centering
    \includegraphics[width=0.5\textwidth]{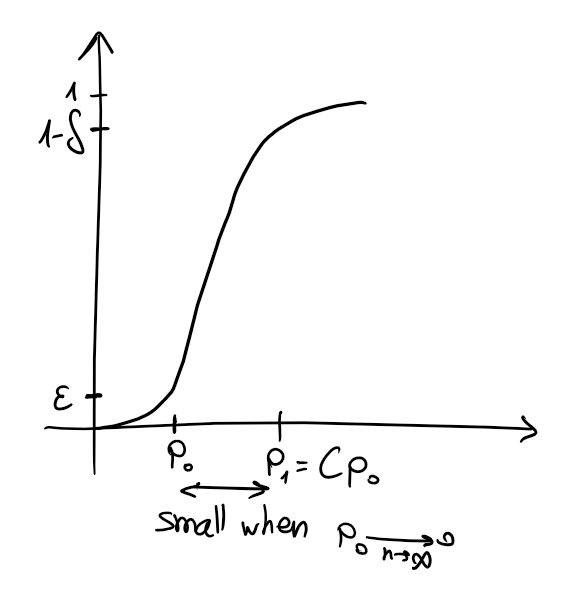}
    \caption{Quick jump to high probability}
\end{figure}

\begin{proof}
    Suppose that the graphs $G_1, \dots, G_C$ are independent and satisfy $G_i \sim G(n, p_0)$. Define $G' = \bigcup_{i=1}^C G_i$. Then $G' \sim G(n, p')$ where $1-p' = \prod_{i=1}^C (1-p_0) = (1-p_0)^C$. Therefore $p' \le C p_0 = p_1$. Since $\mathcal{A}$ is a monotone property, it suffices to prove that: $\mathbb{P}(G \sim G(n, p') \in \mathcal{A}) \ge 1 - \delta$. Indeed, let us note:
    \[
    \mathbb{P}(G \sim G(n, p') \notin \mathcal{A}) \le \mathbb{P}(G_i \notin \mathcal{A}, 1 \le i \le C) = \prod_{i=1}^C \mathbb{P}(G_i \notin \mathcal{A}) \le (1-\epsilon)^C \le \delta.
    \]
    We obtain the required result.
\end{proof}

\noindent \textbf{\underline{Conclusion:}} The jump of $\mathbb{P}(G \sim G(n, p) \in \mathcal{A})$ from nearly 0 to nearly 1 occurs within a range of $\theta(f_{\mathcal{A}}^{-1}(\epsilon))$. Thus if $f_{\mathcal{A}}^{-1}(\epsilon)=o(1)$, then the jump occurs within an interval of negligible width.  

\vspace{1.5em}

\begin{defn}
    Let $\mathcal{A}$ be an increasing monotone property of graphs on vertex set $[n]$. \\
    The probability $p^* = p^*(n)$ is called a threshold function for the property $\mathcal{A}$ if it satisfies:
    \[
    \lim_{n \to \infty} \mathbb{P}(G \sim G(n, p) \in \mathcal{A}) = 
    \begin{cases} 
    0 & p \ll p^*, \\ 
    1 & p \gg p^*. \end{cases}
    \]
\end{defn}

\begin{remark}
    According to the definition, $p^*$ is not a unique function, and it is defined up to the order of magnitude.
\end{remark}

\begin{example}
    $A = \{ G \text{ contains a triangle} \}$. \\
    Let $X$ be a random variable counting the number of triangles in the graph $G(n, p)$. We have:
    \[
    \mathbb{E}X = \binom{n}{3} \cdot p^3.
    \]
    
    \noindent If $\mathbb{E}X \xrightarrow[n \to \infty]{} 0$, then according to Markov's Inequality, $\mathbb{P}(X > 0) = o(1)$. \\
    If $\mathbb{E}X \xrightarrow[n \to \infty]{} \infty$, it can be shown using the Second Moment Method that $\mathbb{P}(X > 0) = 1 - o(1)$.
    
    \vspace{1em}
    
    \noindent When $p = o\left(\frac{1}{n}\right)$, then $\binom{n}{3} \cdot p^3 \xrightarrow[n \to \infty]{} 0$. \\
    When $p = \omega\left(\frac{1}{n}\right)$, then $\binom{n}{3} \cdot p^3 \xrightarrow[n \to \infty]{} \infty$.
    
    \vspace{1em}
    
    \noindent It can be proven that in this case $p^* = \frac{1}{n}$ is a threshold function.
\end{example}

\begin{thm}[Bollobás, Thomason \cite{zbMATH04057560}]
    For every non-trivial monotone property $\mathcal{A}$ of graphs on a set of vertices $[n]$, there exists a threshold function.
\end{thm}

\newpage
\phantomsection
\addcontentsline{toc}{chapter}{Lecture 2}
\lecture{2}{April 19, 2026}{Lecture 2}{Inbal Benbenishty}
\setcounter{chapter}{2}
\setcounter{section}{0}


\section{Monotone Properties}

\begin{thm}[Bollobás, Thomason \cite{zbMATH04057560}]
    For every non-trivial monotone property $\mathcal{A}$ of graphs on a set of vertices $[n]$, there exists a threshold function.
\end{thm}
\begin{proof}
    We saw that $f_A(p) = \mathbb{P}(G \sim G(n, p) \in A)$ is a continuous, strictly increasing function, and $f_A(0) = 0, f_A(1) = 1$. Therefore, there exists a unique $p^* \in (0, 1)$ such that $f_A(p^*) = \frac{1}{2}$. \\
We will use Proposition \ref{1: jump} from Lecture 1 in both directions:

\begin{enumerate}
    \item Let $p_0 = p^*, p_1 = C \cdot p_0, \varepsilon = \frac{1}{2}$. Then by Proposition \ref{1: jump} from Lecture 1:
    \[ \mathbb{P}(G \sim G(n, p_1) \in A) \geq 1 - (1-\varepsilon)^C. \]
    Since $\lim_{C \to \infty} (\frac{1}{2})^C = 0$, it follows that for any $p \gg p^*$:
    \[ \mathbb{P}(G \sim G(n, p) \in A) = 1 - o(1). \]
    
    \item Take $p_1 = p^*, p_0 = \frac{p^*}{C}$. Then by Proposition \ref{1: jump} from Lecture 1:
    \[ \frac{1}{2} = \mathbb{P}(G \sim G(n, p_1) \notin A) \leq (1 - \mathbb{P}(G \sim G(n, p_0) \in A))^C. \]
    For $C \to \infty$ we obtain:
    \[ \mathbb{P}(G \sim G(n, p_0) \in A) = o(1). \]
\end{enumerate}
\end{proof}
\begin{remark}
    In the original proof of the theorem, Bollobás and Thomason used the Kruskal-Katona theorem and obtained a better quantitative estimate.
\end{remark}

\section{Comparison between \texorpdfstring{$G(n,p)$}{G(n,p)} and \texorpdfstring{$G(n,m)$}{G(n,m)}}

One can expect that if $p = \frac{m}{N}$, then $G(n, p)$ and $G(n, m)$ would be ``similar.'' \\
Indeed, let $X = |E(G)|$.

\begin{itemize}
    \item In $G(n, p)$: $X \sim \text{Bin}(N, p)$, and $\mathbb{E}[X] = Np = m$.
    \item In $G(n, m)$: $X = m$ holds with probability $1$.
\end{itemize}
However, this similarity should have some limits. \\
If $m \to \infty$ and $N - m \to \infty$, then $\mathbb{P}(X = m) = \mathbb{P}(\text{Bin}(N, p) = m) = o_n(1)$ since $\text{Var}(X) =\allowbreak Np(1-p) \to \infty$. However, in $G(n, m)$, the number of edges is always $m$.

\subsection*{Quantitative comparison between \texorpdfstring{$G(n,p)$}{G(n,p)} and \texorpdfstring{$G(n,m)$}{G(n,m)}}

\begin{proposition}\label{proposition: comparison gnp_gnm}
    Let $\mathcal{A}$ be a property of graphs on the set of vertices $[n]$. If $4 \leq m \leq N$, $m=Np$, then:
\[
\mathbb{P}(G \sim G(n, p) \in \mathcal{A}) \geq \frac{1}{6\sqrt{m}} \mathbb{P}(G \sim G(n, m) \in \mathcal{A}).
\]
\end{proposition}
\begin{proof}
    Let $X$ be the random variable counting the number of edges in $G \sim G(n, p)$. Notice that $X \sim \text{Bin}(N, p)$, $\mathbb{E}[X] = Np = m$, and $\text{Var}(X) = Np(1-p) < Np = m$. \\
Therefore, by Chebyshev's inequality, for every $a > 0$, the following holds:
\[
\mathbb{P}(|X - m| > a \cdot \sigma) \leq \frac{1}{a^2}.
\]
Furthermore, for every $0 \leq m' \leq N$, we have $\mathbb{P}(X = m') \leq \mathbb{P}(X = m)$ (check it!), and thus:
\[
\mathbb{P}(X = m) \geq \frac{\mathbb{P}(|X - m| \leq a \cdot \sigma)}{2a\sigma + 1} \geq \frac{1 - \frac{1}{a^2}}{2a\sigma + 1},
\]
since in the interval $[m - a\sigma, m + a\sigma]$ there are at most $2a\sigma + 1$ distinct integer values. By setting $a = \sqrt{3}$, we obtain $\mathbb{P}(X = m) \geq \frac{1}{6\sqrt{m}}$ (we used $m\ge 4$ here).\\
Recall that the distribution of $G(n, p)$ conditioned on $|E(G)| = m$ is $G(n, m)$. Therefore, according to the Law of Total Probability:
\begin{align*}
\mathbb{P}(G \sim G(n, p) \in \mathcal{A}) &= \sum_{m'=0}^{N} \mathbb{P}(X = m') \cdot \mathbb{P}(G \sim G(n, p) \in \mathcal{A} \mid X = m') \\
&\geq \mathbb{P}(X = m) \cdot \mathbb{P}(G \sim G(n, p) \in \mathcal{A} \mid X = m) \\
&= \mathbb{P}(X = m) \cdot \mathbb{P}(G \sim G(n, m) \in \mathcal{A}) \\
&\geq \frac{1}{6\sqrt{m}} \cdot \mathbb{P}(G \sim G(n, m) \in \mathcal{A}).
\end{align*}
\end{proof} 
\noindent We can prove a better estimate for monotone properties. 
\begin{proposition}\label{proposition: comparison gnp_gnm - monotone} 
    Let $\mathcal{A}$ be a monotone property of graphs on vertex set $[n]$. Assume $N, p, m$ satisfy $m = N \cdot p$ and $m \to \infty$. Then:
\[
\mathbb{P}(G \sim G(n, p) \in \mathcal{A}) \geq \left(\frac{1}{2} - o(1)\right) \cdot \mathbb{P}(G \sim G(n, m) \in \mathcal{A}).
\]
\end{proposition}
\begin{proof}
    Note that since $\mathcal{A}$ is a monotone property, for every $0 \leq m \leq m' \leq N$:
\[ \mathbb{P}(G \sim G(n, m) \in \mathcal{A}) \leq \mathbb{P}(G \sim G(n, m') \in \mathcal{A}). \] 
Let $X = |E(G \sim G(n, p))|$. By the Law of Total Probability:
\begin{align*}
\mathbb{P}(G \sim G(n, p) \in \mathcal{A}) &= \sum_{m'=0}^{N} \mathbb{P}(X = m') \cdot \mathbb{P}(G \sim G(n, p) \in \mathcal{A} \mid X = m') \\
&= \sum_{m'=0}^{N} \mathbb{P}(X = m') \cdot \mathbb{P}(G \sim G(n, m') \in \mathcal{A}) \\
&\geq \sum_{m'=m}^{N} \mathbb{P}(X = m') \cdot \mathbb{P}(G \sim G(n, m') \in \mathcal{A}) \\
&\geq \sum_{m'=m}^{N} \mathbb{P}(X = m') \cdot \underbrace{\mathbb{P}(G \sim G(n, m) \in \mathcal{A})}_{\text{monotonicity}} \\
&= \mathbb{P}(G \sim G(n, m) \in \mathcal{A}) \cdot \sum_{m'=m}^{N} \mathbb{P}(X = m') \\
&= \mathbb{P}(G \sim G(n, m) \in \mathcal{A}) \cdot \mathbb{P}(\text{Bin}(N, p) \geq m).
\end{align*}
From the Central Limit Theorem, since $m \to \infty$, $\mathbb{P}(\text{Bin}(N, p) \geq m) \geq \frac{1}{2} + o(1)$. Thus:
\[ \mathbb{P}(G \sim G(n, p) \in \mathcal{A}) \geq \left(\frac{1}{2} + o(1)\right) \mathbb{P}(G \sim G(n, m) \in \mathcal{A}). \]
\end{proof}

\section{Connected Components in \texorpdfstring{$G(n,p)$}{G(n,p)} and Phase Transition}

Let $G \sim G(n, p)$. Let $L_1, L_2, ...$ be the connected components of $G$ ordered by cardinality \allowbreak $|L_1| \geq |L_2| \geq \dots$.

\begin{thm}[Erd\H os, Rényi \cite{zbMATH03168330}]\phantomsection
\label{main_ercp}
    Let $p = \frac{c}{n}$ for $c > 0$.
\begin{enumerate}
    \item If $0 < c < 1$ (the sub-critical regime), then whp $|L_1| = O_c(\log n)$.
    \item If $c > 1$ (the super-critical regime), then whp $|L_1| = \Theta(n)$ and all other components satisfy $|L_i| = O_c(\log n)$ for $i \geq 2$.
\end{enumerate}
\end{thm}

\begin{figure}[H]
    \centering
    \includegraphics[width=0.8\textwidth]{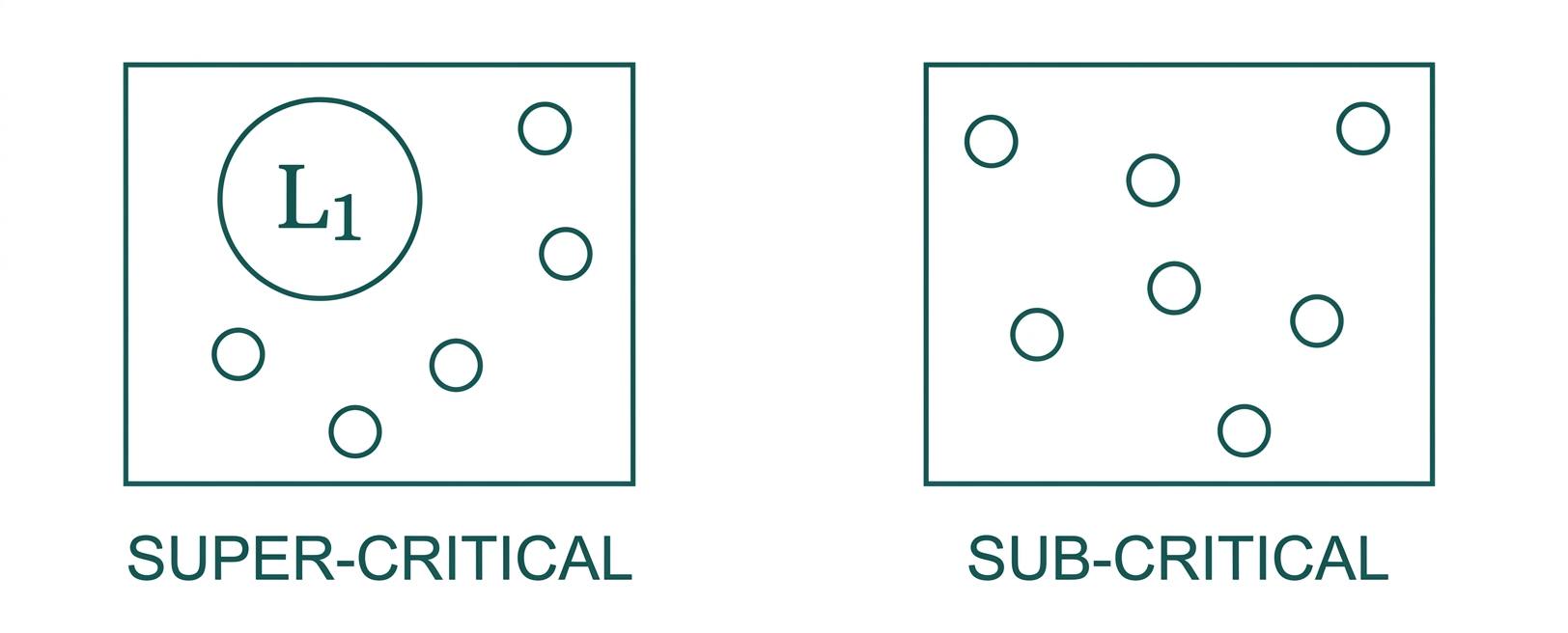} 
    \caption{Super-critical and Sub-critical states}
    \label{fig:critical_states}
\end{figure}

For the proof of Theorem \ref{main_ercp} we utilize several auxiliary lemmas that are given in Lecture 2 and in Lecture 3.

We begin with the case where $0<c<1$.
\begin{lemma}\phantomsection\label{ercp_lemma1}
    In $G \sim G(n, \frac{c}{n})$, whp there are no connected components with 2 or more cycles.
\end{lemma}
\begin{proof}
If a connected graph $H$ has at least 2 cycles, then $H$ contains one of the following 3 configurations:
\begin{enumerate}
    \item Two disjoint cycles $C_1, C_2$ connected by a path.
    \item Two cycles $C_1, C_2$ sharing one vertex.
    \item Two vertices connected by 3 edge-disjoint paths.
\end{enumerate}
One can easily verify that in any of these cases, H contains a path $P$ and 2 additional edges, each from one end of $P$ inwards (``end edges"). The expected number of such structures in $G(n,p)$ is:
\begin{align*}
\mathbb{E}[\#\text{components with }\geq 2 \text{ cycles}]&\leq \sum_{k=4}^{n} \underbrace{\binom{n}{k}}_{\text{\shortstack{choosing the \\ vertices}}} \cdot \underbrace{k!}_{\text{\shortstack{ordering the \\ vertices}}} \cdot \underbrace{k^2}_{\text{\shortstack{placing the \\ "end edges"}}} \cdot \underbrace{p^{k+1}}_{\text{\shortstack{prob. $p$ \\ for each edge}}} \\
&\leq \sum_{k=4}^{n} \frac{n^k}{k!} \cdot k! \cdot k^2 \left(\frac{c}{n}\right)^{k+1}= \frac{c}{n} \sum_{k=4}^{n} k^2 \cdot c^k \xrightarrow[n \to \infty]{} 0.
\end{align*}
Since the expectation goes to 0, by Markov's inequality, whp no such component exists.
\end{proof}

\begin{lemma}\phantomsection\label{ercp_lemma2}
    Whp, the total number of vertices in components containing a (unique) cycle is at most $\omega(n)$ for any $\omega(n)$ such that $\omega(n) \to \infty$ as $n \to \infty$.
\end{lemma}

\begin{proof}
For $3 \leq k \leq n$, let $Z_k$ be the number of vertices in components of order $k$ containing exactly one cycle. Our task is to show that whp $Z_3+Z_4+\ldots +Z_n\le \omega(n)$. We first estimate the expectation:
\[
\mathbb{E}\left[\sum_{k=3}^n Z_k\right] = \sum_{k=3}^n \mathbb{E}[Z_k].
\]
The bound for $\mathbb{E}[Z_k]$ is given by:
\[
\mathbb{E}[Z_k] \leq \underbrace{\binom{n}{k}}_{\text{\shortstack{choosing k \\ vertices}}} \cdot \underbrace{k^{k-2}}_{\text{\shortstack{choosing a spanning \\ tree (Cayley)}}} \cdot \underbrace{\binom{k}{2}}_{\text{forming a cycle}} \cdot \underbrace{p^k}_{\text{\shortstack{prob. $p$ \\for each edge}}}  \cdot \underbrace{(1-p)^{k(n-k) + \binom{k}{2}-k}}_{\text{\shortstack{disconnecting from \\ other components\\ and forbidding\\other edges inside}}} \cdot \underbrace{k}_{\text{\shortstack{order of the \\ component}}}.
\]
Note that we have 
\begin{align*}
\binom{n}{k} &= \frac{n^k}{k!} \cdot \frac{n(n-1)\cdots(n-k+1)}{n^k} = \frac{n^k}{k!} \left( \frac{n}{n} \cdot \frac{n-1}{n} \cdots \frac{n-k+1}{n} \right) \\
&\stackrel{1-x \leq e^{-x}}{\leq} \frac{n^k}{k!} \cdot e^{-\left(\frac{1}{n} + \frac{2}{n} + \dots + \frac{k-1}{n}\right)} = \frac{n^k}{k!} \cdot e^{-\frac{k(k-1)}{2n}}.
\end{align*}
Hence,
\begin{align*}
    \mathbb{E}[Z_k]&\leq \frac{n^k}{k!} \cdot e^{-\frac{k(k-1)}{2n}} \cdot k^{k+1} \left( \frac{c}{n} \right)^k \cdot e^{-\frac{c}{n} (kn - \frac{k^2}{2} - \frac{3k}{2})} \\
    &\leq \frac{n^k}{(\frac{k}{e})^k} e^{-\frac{k(k-1)}{2n}} \cdot k^{k+1} e^{-ck + \frac{k^2}{2n} + \frac{3k}{2n}} \cdot \left( \frac{c}{n} \right)^k \\
    &= k \cdot (ce^{1-c})^k e^{\frac{2k}{n}}\le e^2 k \cdot (ce^{1-c})^k.
\end{align*}
Notice that for every $c\ne 1$, we have $ce^{1-c}<1$, hence plugging this into $\mathbb{E}[Z_k]$ we obtain:
\[
\mathbb{E}\left[\sum_{k=3}^n Z_k\right] \leq \sum_{k=3}^n e^2 \cdot k(ce^{1-c})^k = O(1).
\]
Therefore, by Markov's inequality, for any $\omega(n) \to \infty$, we have that whp $\sum_{k=3}^n Z_k \leq \omega(n)$.
\end{proof}

\newpage
\phantomsection
\addcontentsline{toc}{chapter}{Lecture 3}
\lecture{3}{May 3, 2026}{Lecture 3}{Gali Maman}
\setcounter{chapter}{3}
\setcounter{section}{0}

\section{Connected Components in \texorpdfstring{$G(n,p)$}{G(n,p)} and Phase Transition - \allowbreak Continued}

\begin{lemma}\phantomsection\label{ercp_lemma3}
For $c > 0$, $c\ne 1$, let $p = \frac{c}{n}$. For $1 \le k \le n$, let $X_k$ denote the number of connected components that are trees of order $k$. The following hold:
\begin{enumerate}
    \item[\textbf{(a)}] There exists a constant $\beta_1 = \beta_1(c) > 0$ such that whp, for all $k \ge \beta_1 \log n$, it holds that $X_k = 0$. \textit{(Meaning: whp, there is no connected component of order $\ge \beta_1\log n$ in $G$ which is a tree).}
    
    \item[\textbf{(b)}] There exists a constant $\beta_0 = \beta_0(c) > 0$ such that whp, for all $1 \le k \le \beta_0 \log n$, it holds that:
    \[ X_k = (1+o_n(1)) \frac{n}{c} \cdot \frac{k^{k-2}}{k!} (ce^{-c})^k. \]
\end{enumerate}
\end{lemma}

\begin{proof}
\textbf{Part (a):} We first calculate the expectation $\mathbb{E}[X_k]$:

\[ \mathbb{E}[X_k] = \underbrace{\binom{n}{k}}_{\text{choosing vertices}} \cdot \underbrace{k^{k-2}}_{\text{choosing the tree}} \cdot \underbrace{p^{k-1}}_{\text{tree edges}} \cdot \underbrace{(1-p)^{k(n-k)}}_{\text{no edges to rest of graph}} \cdot \underbrace{(1-p)^{\binom{k}{2}-k+1}}_{\text{no extra edges in tree}}. \]
Notice that for all $k$, $\binom{n}{k} \le \frac{n^k}{k!}$, and specifically for $k = O(\log n)$, we have $\binom{n}{k} = (1+o(1)) \frac{n^k}{k!}$. Therefore, we can bound the expectation:
\[ \mathbb{E}[X_k] \le \frac{n^k}{k!} k^{k-2} \left(\frac{c}{n}\right)^{k-1} e^{-ck + O(\frac{k^2}{n})}, \]
and for $k = O(\log n)$, it holds that:
\[ \mathbb{E}[X_k] = (1+o(1)) \frac{n}{c} \cdot \frac{k^{k-2}}{k!} (ce^{-c})^k. \]
Let us denote $\mu_k = \mathbb{E}[X_k] = \binom{n}{k} k^{k-2} p^{k-1} (1-p)^{k(n-k) + \binom{k}{2} - k + 1}$. We can further bound $\mu_k$ as follows (assuming $k=O(\log n)$):
\begin{align*}
    \mu_k &= (1+o(1)) \frac{n}{c} \cdot \frac{k^{k-2}}{k!} (ce^{-c})^k \le (1+o(1)) \frac{n}{c} e^k (ce^{-c})^k \le (1+o(1)) \frac{n}{c} (ce^{1-c})^k.
\end{align*}

Notice that for every $1 \ne c \in \mathbb{R}$ it holds that $ce^{1-c} < 1$. Therefore, if we choose a constant $0 < \beta_1 = \beta_1(c)$ large enough, for $k = \lceil \beta_1 \log n \rceil$ it will hold that $\mu_k = o(n^{-2})$.

Furthermore, analyzing the ratio between consecutive terms yields:
\begin{align*}
    \frac{\mu_{k+1}}{\mu_k} &= \frac{\binom{n}{k+1} (k+1)^{k-1} p^k (1-p)^{(k+1)(n-k-1) + \binom{k+1}{2} - k}}{\binom{n}{k} k^{k-2} p^{k-1} (1-p)^{k(n-k) + \binom{k}{2} - k + 1}}= \frac{n-k}{k+1} \cdot \frac{(k+1)^{k-1}}{k^{k-2}} p (1-p)^{n-k-2}\\
    &= \frac{n-k}{n} \cdot c \cdot \left(\frac{k+1}{k}\right)^{k-2} (1-p)^{n-k-2}\le \frac{n-k}{n} \cdot c \cdot e \cdot e^{-\frac{c(n-k)}{n}} (1-p)^{-2} < (1-p)^{-2}.
\end{align*}

Hence, summing over the tail gives:
\begin{align*}
    \sum_{k=\lceil\beta_1 \log n\rceil}^n \mu_k &\le \mu_{\lceil\beta_1 \log n\rceil} \sum_{i=0}^{n} \underbrace{(1-p)^{-2i}}_{= O(1)} = O(n) \mu_{\lceil\beta_1 \log n\rceil} = O(n) \cdot o(n^{-2}) = o(n^{-1}).
\end{align*}

Therefore, according to Markov's inequality, whp, for all $k \ge \beta_1 \log n$ it holds that $X_k = 0$.

\vspace{1em}
\noindent \textbf{Part (b) (sketch):} We will use the Second Moment Method (Chebyshev's inequality) for \break $k \le \beta_0 \log n$, where $\beta_0 = \beta_0(c) > 0$ is a small enough constant. We already know that:
\[ \mathbb{E}[X_k] = (1+o(1)) \frac{n}{c} \cdot \frac{k^{k-2}}{k!} (ce^{-c})^k. \]
We will calculate the variance of $X_k$. Consider the number of ordered pairs of trees $T_1, T_2$ of order $k$ that are disjoint in vertices:

\begin{figure}[H]
    \centering
    \includegraphics[width=0.8\textwidth]{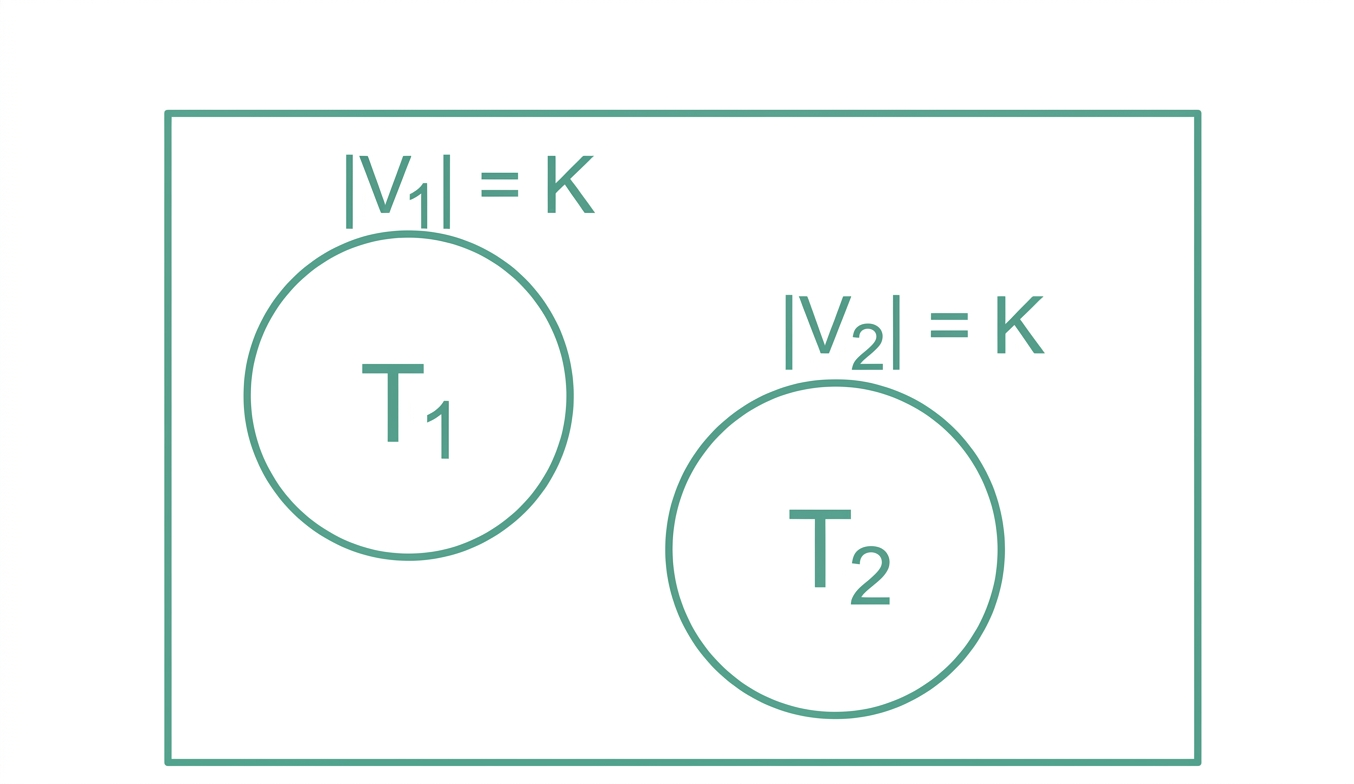} 
    \caption{T1 and T2}
    \label{fig:T1 and T2}
\end{figure}

\[
\mathbb{E}\left[\begin{matrix} \text{Number of ordered pairs of} \\ \text{tree components of order } k \end{matrix}\right] =
\underbrace{\binom{n}{k}}_{\substack{\text{choosing} \\ V(T_1)}} \cdot \underbrace{\binom{n-k}{k}}_{\substack{\text{choosing} \\ V(T_2)}} \cdot \underbrace{(k^{k-2} p^{k-1})^2}_{\substack{\text{trees of order } k \\ \text{and edges}}} \cdot \underbrace{(1-p)^{2k(n-2k) + \binom{2k}{2} - 2k + 2}}_{\substack{\text{no edges going out of} \\ T_1 \cup T_2 \text{ and between them}}}.
\]

We can bound the variance, using the above estimate, as follows:
\begin{align*}
    Var(X_k) &\le \mathbb{E}[X_k] + \mathbb{E}^2[X_k] \underbrace{((1-p)^{-k^2} - 1)}_{\approx p k^2} \le \mathbb{E}[X_k] + \frac{2c k^2}{n} \mathbb{E}^2[X_k].
\end{align*}

Applying Chebyshev's inequality gives us:
\begin{align*}
    \mathbb{P}(|X_k - \mathbb{E}[X_k]| \ge \varepsilon \mathbb{E}[X_k]) &\le \frac{1}{\varepsilon^2 \mathbb{E}[X_k]} + \frac{2c k^2}{\varepsilon^2 n}.
\end{align*}

Recall that:
\begin{align*}
    \mathbb{E}[X_k] &= (1+o(1)) \frac{n}{c} \cdot \frac{k^{k-2}}{k!} (ce^{-c})^k \le (1+o(1)) \frac{n}{c} \underbrace{e^k (ce^{-c})^k}_{(ce^{1-c})^k}.
\end{align*}
For $k \le \beta_0 \log n$, we choose $\beta_0 = \beta_0(c) > 0$ small enough such that: $\mathbb{E}[X_k] \ge n^{\frac{1}{2}}$. Hence, by the union bound:
\begin{align*}
    \mathbb{P}[\exists 1 \le k \le \beta_0 \log n : |X_k - \mathbb{E}[X_k]| > \varepsilon \mathbb{E}[X_k]] &= O(\log n) \cdot O\left(\frac{1}{\sqrt{n}}\right) = o(1),
\end{align*}
which completes the proof.
\end{proof}

Let us conclude for the subcritical regime:
\begin{thm}[Subcritical]
Let $0 < c < 1$ be a constant, and let $G \sim G(n, \frac{c}{n})$. Then whp, all connected components of $G$ are of order $O_c(\log n)$.
\end{thm}

\begin{proof}
According to Lemmas \ref{ercp_lemma1} and \ref{ercp_lemma2} from Lecture 2 and due to Lemma \ref{ercp_lemma3}, we get whp that:
\begin{enumerate}
    \item there are no components with more than one cycle (Lemma \ref{ercp_lemma1});
    \item the number of vertices in components with exactly one cycle is at most $\omega(n)$ (where $1\ll \omega(n)\allowbreak \ll \log n$) (Lemma \ref{ercp_lemma2});
    \item there are no connected components that are trees of order $> \beta_1 \log n$ (Lemma \ref{ercp_lemma3}).
\end{enumerate} 

Combining these facts, it follows that whp, all connected components in $G$ are of order at most $\beta_1 \log n$, as required.
\end{proof}

We now turn to the supercritical regime. We will need the following definition later. Let us define $x = x(c)$ by:
\[
x \coloneqq x(c) = 
\begin{cases} 
  c & c < 1, \\
  \text{the unique solution of } x e^{-x} = c e^{-c} & c > 1 \text{ (in the interval } (0,1)).
\end{cases}
\]

[Note that the function $f(x) = x e^{-x}$ attains its maximum value at $x=1$, is monotone increasing for $0<x<1$, and is monotone decreasing for $x>1$. Therefore, for every $c > 1$, the solution $x \in (0,1)$ is well-defined and unique.]

\begin{figure}[H]
    \centering
    \includegraphics[width=0.8\textwidth]{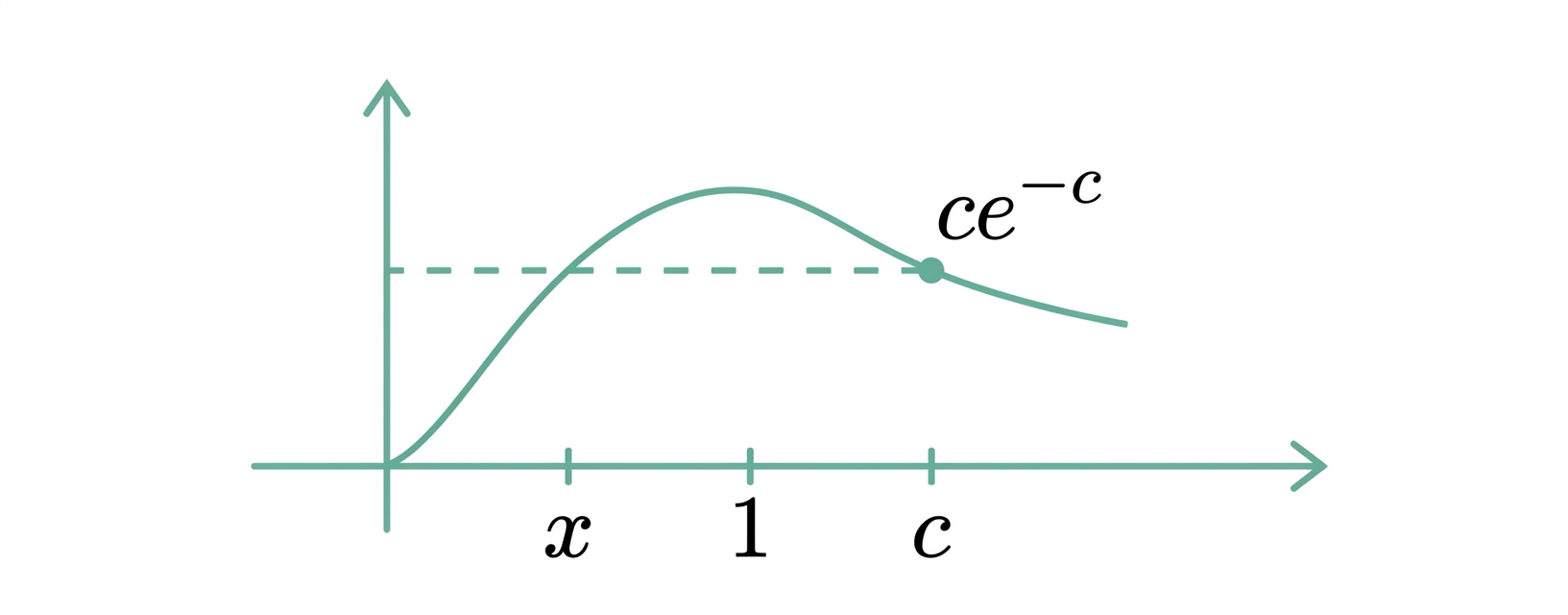} 
    \caption{x}
    \label{fig:x}
\end{figure}

\vspace{1em}
\begin{lemma}\phantomsection\label{ercp_lemma4}
Let $c > 0$, $c \ne 1$. Then:
\[ x = \sum_{k=1}^{\infty} \frac{k^{k-1}}{k!} (ce^{-c})^k. \]
\end{lemma}

\begin{proof}
We will prove this via the $G(n,p)$ model. Let us define a random variable $Z$ to be the number of vertices in components with a cycle. Additionally, for all $k \ge 1$, we define $X_k$ to be the number of connected components of $G$ that are trees with $k$ vertices.

Summing the vertices of all the components gives the total number of vertices:
\[ n = |V(G)| = Z + \sum_{k=1}^n k \cdot X_k. \]
Taking the expectation of both sides yields:
\begin{align}\label{main_eq_3}
    n = \mathbb{E}[Z] + \sum_{k=1}^n k \cdot \mathbb{E}[X_k].
\end{align}

\textbf{Case 1: $0 < c < 1$.} From Lemmas \ref{ercp_lemma1} and \ref{ercp_lemma2} from Lecture 2 it follows that $\mathbb{E}[Z] = o(n)$. From Lemma \ref{ercp_lemma3} we know that for $k \le \beta_0 \log n$, $\mathbb{E}[X_k] = (1+o(1)) \frac{n}{c} \cdot \frac{k^{k-2}}{k!} (ce^{-c})^k$. We have also bounded the tail: $\sum_{k=\beta_1 \log n}^n \mathbb{E}[X_k] = o\left(\frac{1}{n}\right)$.
This implies that the contribution of the tail is negligible:
\[ \sum_{k=\beta_1 \log n}^n k \cdot \mathbb{E}[X_k] \le n \sum_{k=\beta_1 \log n}^n \mathbb{E}[X_k] = n \cdot o\left(\frac{1}{n}\right) = o(1). \]

Therefore, substituting these into equation \eqref{main_eq_3}, we get:
\begin{align*}
    n &= \underbrace{o(n)}_{\mathbb{E}[Z]} + \underbrace{o(1)}_{\text{tail}} + \sum_{k=1}^{\beta_0 \log n} \frac{n}{c} \cdot \frac{k^{k-1}}{k!} (ce^{-c})^k =  o(n) + (1+o(1)) \frac{n}{c} \sum_{k=1}^{\infty} \frac{k^{k-1}}{k!} (ce^{-c})^k.
\end{align*}
Dividing both sides by $n$ and taking the limit as $n \to \infty$, we get:
\[ 1 = \frac{1}{c} \sum_{k=1}^{\infty} \frac{k^{k-1}}{k!} (ce^{-c})^k \quad \Longrightarrow \quad c = \sum_{k=1}^{\infty} \frac{k^{k-1}}{k!} (ce^{-c})^k. \]
Since $x=c$ in this domain, the lemma holds.

\textbf{Case 2: $c > 1$.} In this regime, $xe^{-x} = ce^{-c}$, where $0 < x < 1$. Using the result from the first case applied to $x$, we immediately get:
\[ \sum_{k=1}^{\infty} \frac{k^{k-1}}{k!} (ce^{-c})^k = \sum_{k=1}^{\infty} \frac{k^{k-1}}{k!} (xe^{-x})^k = x.\qedhere \]
\end{proof}

We are now ready to establish the phenomenon in the supercritical regime.
\begin{thm}[Supercritical]
Let $c > 1$ be a constant and let $G \sim G(n, \frac{c}{n})$. Then whp:
\begin{enumerate}
    \item[\textbf{1)}] In $G$ there exists a unique component $L_1$ of order $|L_1| = (1+o(1))\left(1-\frac{x}{c}\right)n$, where $x=x(c)$ is defined as in Lemma \ref{ercp_lemma4}. This is the \textbf{giant component} (of order $\Theta(n)$).
    \item[\textbf{2)}] For every $i \ge 2$ it holds that $|L_i| = O_c(\log n)$.
\end{enumerate}
\end{thm}

\begin{proof}[Proof Idea] We will prove that whp, the following four facts hold:
\begin{enumerate}[label=\alph*)]
    \item There are no components with orders in the range $[\beta_3 \log n, \beta_2 n]$ for some constants $\beta_2, \beta_3 > 0$;
    \item The total order of components of order at most $\beta_3 \log n$ that are not trees is $o(n)$;
    \item The total order of components that are trees of order $\le \beta_3 \log n$ is $(1+o(1))\frac{nx}{c}$;
    \item In conclusion, the rest of the graph has volume $(1+o(1))\left(n - \frac{nx}{c}\right)$, forming a single connected component, which is $L_1$.
\end{enumerate}\renewcommand{\qedsymbol}{}
\end{proof}
\begin{proof} 
For $1 \le k \le n$, let $Y_k$ denote the number of connected components of order $k$ in $G$. We can bound the expectation $\mathbb{E}[Y_k]$ as follows:
\begin{align*}
    \mathbb{E}[Y_k] &\le \binom{n}{k} k^{k-2} p^{k-1} (1-p)^{k(n-k)} \le \left(\frac{en}{k}\right)^k k^{k-2} \left(\frac{c}{n}\right)^{k-1} e^{-ck + \frac{ck^2}{n}} = \frac{n}{ck^2} \left(c e^{1-c + \frac{ck}{n}}\right)^k.
\end{align*}

We will choose $\beta_2 = \beta_2(c) > 0$ small enough such that $c e^{1-c+c\beta_2} < 1$ (which is always possible as $ce^{1-c}<1$). Then, we can choose $\beta_3 = \beta_3(c) > 0$ large enough such that $\left(c e^{1-c+c\beta_2}\right)^{\beta_3 \log n} \ll \frac{1}{n^2}$.
Under these choices, $\mathbb{E}[Y_k] = o\left(\frac{1}{n}\right)$ for all $\beta_3 \log n \le k \le \beta_2 n$. Therefore, according to Markov's inequality, whp, for all $k \in [\beta_3 \log n, \beta_2 n]$ it holds that $Y_k = 0$.

Now, we estimate the total size of the connected components of $G$ that are small trees ($\le \beta_3 \log n$). According to Lemma \ref{ercp_lemma3}, whp we are guaranteed that for components of order $1 \le k \le \beta_0 \log n$, the number of connected components that are trees of order $k$ is $X_k = (1+o(1)) \frac{n}{c} \cdot \frac{k^{k-2}}{k!} (ce^{-c})^k$.

Therefore, whp, the total volume of trees of order $k$ (for $1 \le k \le \beta_0 \log n$) is \break $k \cdot X_k = (1+o(1)) \frac{n}{c} \cdot \frac{k^{k-1}}{k!} (ce^{-c})^k$. Summing this up yields:
\begin{align*}
    \sum_{k=1}^{\beta_0 \log n} k \cdot X_k &= (1+o(1)) \frac{n}{c} \sum_{k=1}^{\beta_0 \log n} \frac{k^{k-1}}{k!} (ce^{-c})^k\\
    &= (1+o(1)) \frac{n}{c} \sum_{k=1}^{\infty} \frac{k^{k-1}}{k!} (ce^{-c})^k = (1+o(1)) \cdot \frac{nx}{c},
\end{align*}
where $x$ is from Lemma \ref{ercp_lemma4}.

Furthermore, for all $\beta_0 \log n \le k \le \beta_3 \log n$, it holds that $\mathbb{E}[X_k] = O\left(\frac{n}{k^2} (ce^{1-c})^k\right)$. If we substitute $k \ge \beta_0 \log n$, this expression becomes $O\left(\frac{n^{1-\delta}}{\log n}\right)$ for some $\delta > 0$. 
Hence, the expected number of components in this intermediate range is $\mathbb{E}\left[\sum_{\beta_0 \log n}^{\beta_3 \log n} X_k\right] = O(n^{1-\delta})$. According to Markov's inequality, we get whp:
\[ \sum_{k=1}^{\beta_3 \log n} k \cdot X_k = (1+o(1)) \frac{nx}{c}. \]

Next, we estimate the typical total volume of components of order $\le \beta_3 \log n$ with cycles. Let $Z_k$ denote the number of vertices in components of order $k$ with a cycle. We can bound this expectation:
\begin{align*}
\mathbb{E}\left[\sum_{k=1}^{\beta_3 \log n} Z_k\right] &\le \sum_{k=1}^{\beta_3 \log n} \underbrace{\binom{n}{k}}_{\substack{\text{choosing} \\ \text{component}}} \cdot k^{k-2} \cdot \binom{k}{2} \cdot k \cdot \underbrace{p^k}_{\substack{\text{payment} \\ \text{for edges} \\ \text{that close} \\ \text{a cycle}}} \cdot \underbrace{(1-p)^{k(n-k)}}_{\substack{\text{no edges} \\ \text{between the} \\ \text{component} \\ \text{and the rest}}} \le \\
    &\le \sum_{k=1}^{\beta_3 \log n} \left(\frac{en}{k}\right)^k k^{k+1} \left(\frac{c}{n}\right)^k e^{-ck + \frac{ck^2}{n}} = O(1).
\end{align*}

Thus, Markov's inequality implies that whp, $\sum_{k=1}^{\beta_3 \log n} Z_k = o(n)$.

\paragraph{Intermediate Summary:}
At this point, we know that whp:
\textbf{(1)} There are no components of order in $[\beta_3 \log n, \beta_2 n]$.
\textbf{(2)} The total volume of components that are trees of order up to $\beta_3 \log n$ is $(1+o(1))\frac{nx}{c}$.
\textbf{(3)} The total volume of components $\le \beta_3 \log n$ with a cycle is negligible ($o(n)$).

Therefore, whp, the remaining volume $(1+o(1))(1-\frac{x}{c})n$ must be taken up entirely by connected components of order at least $\beta_2 n$.

\vspace{1em}
\noindent\textbf{Final Step (Sprinkling):} It remains to show that whp, this entire volume of about $(1-\frac{x}{c})n$ is taken up by a \textbf{single component}. We will prove this using the \textbf{sprinkling} method.

We have $p = \frac{c}{n}$ for $c > 1$. Let us define $c_1 = c - \frac{\log n}{n}$ and its corresponding $p_1 = \frac{c_1}{n}$; thus $p_1$ is only slightly smaller than $p$. We define the "sprinkled" probability $p_2$ by the relation $1-p = (1-p_1)(1-p_2)$.

Let us sample $G_1 \sim G(n, p_1)$ and $G_2 \sim G(n, p_2)$ independently, so that $G = G_1 \cup G_2$. We define $x_1 = x(c_1)$ by $x_1 \in (0,1)$ and $x_1 e^{-x_1} = c_1 e^{-c_1}$. 
Notice that $x_1 = x(1-o(1))$. Applying the previous considerations to $G_1$, we get that whp the volume $(1+o(1))\left(1-\frac{x_1}{c_1}\right)n = (1+o(1))\left(1-\frac{x}{c}\right)n$ is taken up by components of order $\ge \beta_2 n$.

The total number of these large components is bounded by a constant: $\ell \le \frac{(1-\frac{x}{c})n(1+o(1))}{\beta_2 n} = O(1)$. Suppose that $C_1, \dots, C_\ell$ are these components in $G_1$, where $|C_i| \ge \beta_2 n$. 

When we "sprinkle" the edges of $G_2$, the probability that two such huge components fail to connect is very small:
\begin{align*}
    \mathbb{P}\left[\begin{matrix} \text{no edge between } C_i \\ \text{and } C_j \text{ in } G(n, p_2) \end{matrix}\right] &= (1-p_2)^{|C_i|\cdot|C_j|} \le e^{-p_2 \Theta(n^2)} = n^{-\Theta(1)} = o(1),
\end{align*}
where we used here the following fact: since $1-p = (1-p_1)(1-p_2)$ implies $p_2 \ge p - p_1 = \frac{\log n}{n^2}$.

Therefore, after sprinkling, whp, all the large components $C_1, \dots, C_\ell$ merge into a single giant component $L_1$. Its total volume is exactly the volume that completes the rest of the graph:
\[ |L_1| = (1+o(1))\left(1-\frac{x}{c}\right)n. \]
This concludes the proof.
\end{proof}
\newpage
\phantomsection
\addcontentsline{toc}{chapter}{Lecture 4}
\lecture{4}{May 10, 2026}{Lecture 4}{Roy Maulbogat}
\setcounter{chapter}{4}
\setcounter{section}{0}

\section{Giant Component in \texorpdfstring{$G \sim G\!\left(n,\frac{c}{n}\right)$}{G \sim G\!\left(n,\frac{c}{n}\right)} --- Alternative Approach}

In the previous lecture we looked at $G \sim G\!\left(n, \frac{c}{n}\right)$ with $c > 1$. We defined $x \in (0, 1)$ to be the unique solution of $xe^{-x} = ce^{-c}$ in that interval, and proved that whp\ $|L_1| = (1+o(1))\!\left(1 - \frac{x}{c}\right)n$, meaning the largest component is a giant component of order linear in $n$. We also proved that all other components are small: $|L_i| = O_c(\log n)$ for all $i \ge 2$. The proof was a long, tedious calculation. We now want to see an alternative explanation for the order of $|L_1|$ that is more intuitive, and can also be formalized into a rigorous proof.

Let $c > 1$, $G \sim G\!\left(n, \frac{c}{n}\right)$, and $v \in [n]$ be a vertex. Denote by $C_v$ the connected component of $v$. We want to evaluate the size of $C_v$. In fact, we would like to evaluate $y \coloneqq \Pr{|C_v| \text{ is ``big''}}$. After that we can argue that most largish components actually merge into one giant component, and by linearity of expectation we expect $|L_1| = (1 + o(1))y \cdot n$.

How do we evaluate $|C_v|$? The idea is to expose the neighbors of $v$, then expose the neighbors of the neighbors of $v$, and so on. This algorithm is called \textit{Breadth-First Search (BFS)}, which will be presented formally later.

The number of neighbors of $v$ is distributed as $\mathrm{Bin}(n - 1, p)$. More generally, suppose we have exposed $C_v$ up to a certain stage. The number of neighbors of $u\in C_v$ (outside of $C_v$) is distributed as $\mathrm{Bin}(n - |C_v|, p)$. As long as $|C_v|=o(n)$, we can estimate $\mathrm{Bin}(n - |C_v|, p) \approx \mathrm{Bin}(n, p)$, meaning the expected value is asymptotically the same.

Now, $\mathrm{Bin}\!\left(n, \frac{c}{n}\right)$ can be approximated by $\mathrm{Po}(c)$, where $\mathrm{Po}(c)$ is the Poisson distribution given by:
\[
  \mathbb{P}(X = k) = e^{-c} \cdot \frac{c^k}{k!}, \quad k \ge 0.
\]
We model the exposure of $|C_v|$ by a \textit{Galton--Watson branching process} with offspring distribution $\mathrm{Po}(c)$ (for further details, see, for example, \cite{zbMATH06976511}):

\begin{defn}
    We define a \textit{Galton--Watson branching process} with offspring distribution $\mathrm{Po}(c)$ as follows: denote by $X_i$ the size of generation $i$. Set $X_0 = 1$, and
    \[
      X_{n+1} = \sum_{j=1}^{X_n} \xi_j^{(n)},
    \]
    where $\xi_j^{(n)}$ are independent $\mathrm{Po}(c)$ random variables. We grow a tree: in generation~$0$ we start from a single root and assign it a number of children drawn from $\mathrm{Po}(c)$. Each child then independently produces its own children according to $\mathrm{Po}(c)$, and so on. 
\end{defn}

\begin{figure}[H]
  \centering
  \includegraphics[width=0.75\textwidth]{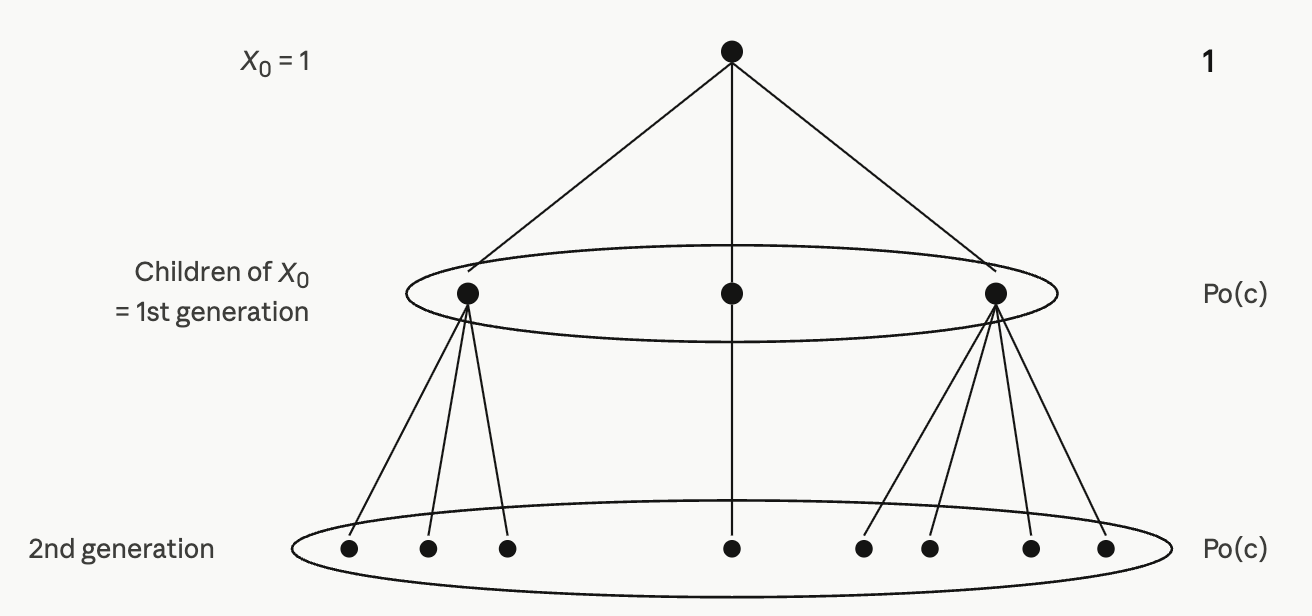}
  \caption{Galton--Watson tree.}
  \label{fig:gw_tree}
\end{figure}

Let $y$ be the probability that this process is \emph{infinite}, i.e.\ the survival probability of the dynasty (and $1-y$ is the probability of extinction). As long as $C_v$ is not too large, this branching
process is a good approximation for growing $C_v$, so $y$ equals the probability that $C_v$ is big.

For the tree rooted at $v$ to \emph{not} survive, every child of the root must also
fail to survive. Since the children are independent, we have:
\begin{align*}
  1 - y &= \sum_{k=0}^{\infty} \mathbb{P}(X = k)\,(1-y)^k
         = \sum_{k=0}^{\infty} e^{-c} \cdot \frac{c^k}{k!}(1-y)^k
         = e^{-c} \cdot e^{c(1-y)}
         = e^{-cy}.
\end{align*}
Hence $1 - y = e^{-cy}$, $y \in (0,1)$. One can show analytically that this
equation has a unique solution for every $c > 1$, from which we can conclude (by further work) that
whp $|L_1| = (1+o(1))y \cdot n$.

In Theorem \ref{main_ercp} we have shown that $|L_1| = (1+o(1))\left(1 - \frac{x}{c}\right) \cdot n$, hence we have to verify that $y = 1 - \frac{x}{c}$. Recall that $xe^{-x} = ce^{-c}$, thus:
\[
  \frac{x}{c} = e^{-c+x} = e^{-c(1 - x/c)}
  \implies
  1 - \frac{x}{c} = 1 - e^{-c(1-x/c)}.
\]
Since $y$ satisfies $y = 1 - e^{-cy}$, the identify $y = 1 - \frac{x}{c}$ follows.

\begin{remark}
When $c = 1 + \varepsilon$ for some small $\varepsilon > 0$,
one can show that $y = (1 + o_\varepsilon(1))\,2\varepsilon$. Hence, if $G \sim G\!\left(n, \frac{c}{n}\right)$ with
$c = 1 + \varepsilon$ for a small constant $\varepsilon > 0$, then whp $|L_1| = (1+o_\varepsilon(1))\,2\varepsilon n$.
\end{remark}

\begin{figure}[H]
  \centering
  \includegraphics[width=0.65\textwidth]{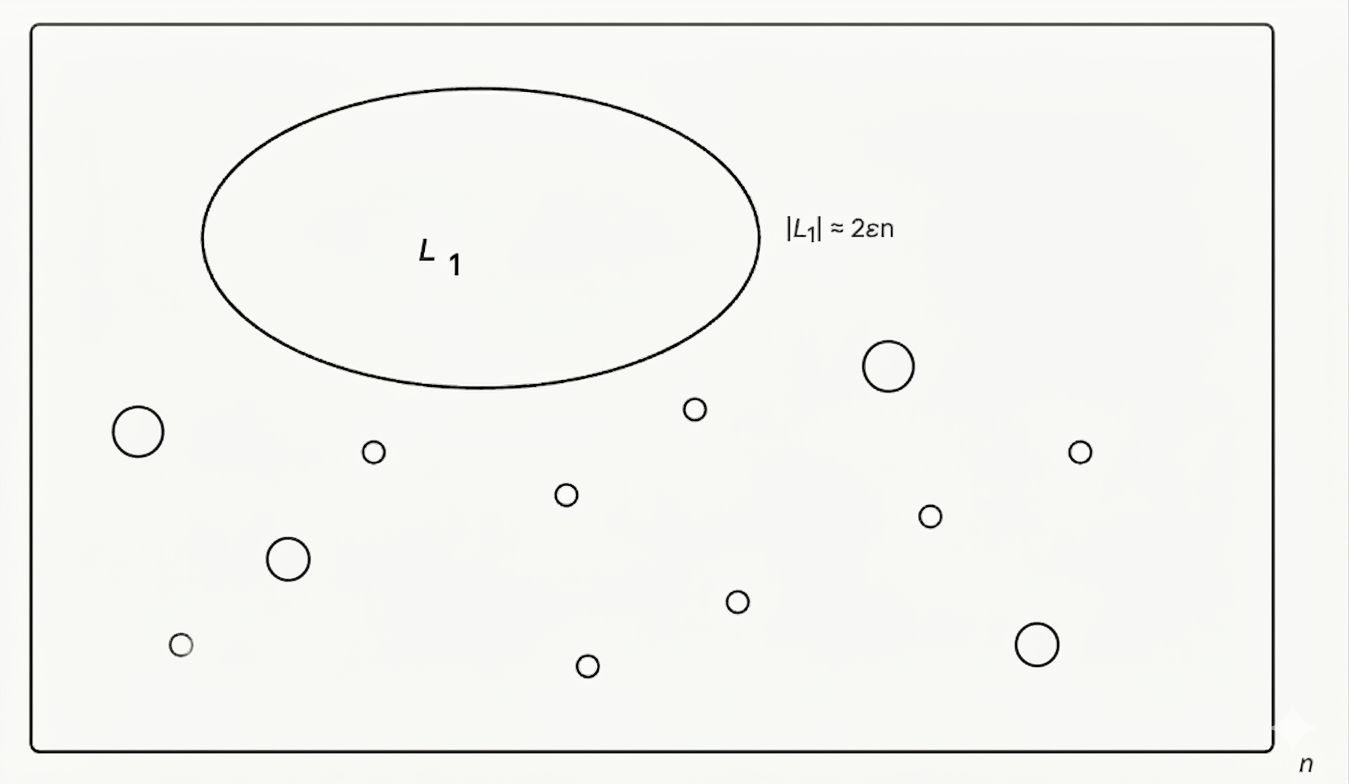}
  \caption{Linear approximation $y \approx 2\varepsilon$ near the critical point $c = 1 + \varepsilon$.}
  \label{fig:critical_approx}
\end{figure}

\section{Long Paths and Cycles in Supercritical/Sparse Random Graphs}

We will show that for $G\sim G\left(n,\frac{1+\varepsilon}{n}\right)$ (the supercritical case), $G$ contains whp a path of length $\Theta(n)$, and when $G\sim G\left(n,\frac{c}{n}\right)$ for a constant $c>1$ (sparse random graph), $G$ contains a path of length $n(1 - \delta)$, where $\delta\coloneqq \delta(c)$ is a constant satisfying $\lim_{c\to\infty}\delta(c)=0$.

First, we need some background in graph-search algorithms.

\subsection{Breadth-First Search (BFS)}

\textbf{Input:} A graph $G = (V,E)$ and a permutation $\sigma$ on $V$ determining the order of vertices.
\noindent
\textbf{Output:} The connected components of $G$.

The algorithm maintains and updates the partition $V = S \cup Q \cup T$, where:
\begin{itemize}
  \item $S$ = vertices whose processing has been completed;
  \item $Q$ = vertices currently being processed (\textbf{FIFO}\footnote{First-in-First-out.} queue);
  \item $T$ = vertices waiting to be processed.
\end{itemize}
Initialize $S = Q = \emptyset$, $T = V$. The algorithm terminates once $Q = T = \emptyset$, $S = V$.

\paragraph{Algorithm step.}
\begin{itemize}
  \item If $Q \ne \emptyset$: let $v$ be the first vertex in $Q$. Scan $T$ according to $\sigma$
        and look for a neighbor of $v$. If a neighbor $u$ is found, move $u$ from $T$ to $Q$.
        Otherwise, move $v$ from $Q$ to $S$.
  \item If $Q = \emptyset$: take the first vertex $v$ in $T$ (according to $\sigma$) and move it to $Q$.
\end{itemize}

The algorithm never closes a cycle, and it outputs a spanning forest $F$ with $V(F) = V(G)$ whose components coincide with those of $G$. The time between $Q$ becoming non-empty and becoming empty again is called an \textit{epoch}, corresponding to discovering one
connected component.

\subsection{Depth-First Search (DFS)}

\textbf{Input:} A graph $G = (V,E)$ and a permutation $\sigma$ on $V$.

\noindent
\textbf{Output:} The connected components of $G$.

The algorithm maintains the partition $V = S \cup U \cup T$, where:
\begin{itemize}
  \item $S$ = vertices we have completed processing;
  \item $U$ = vertices currently being processed (\textbf{LIFO}\footnote{Last-in-First-out.} stack);
  \item $T$ = vertices waiting to be processed.
\end{itemize}
The only difference from BFS is that when $U \ne \emptyset$, we let $v$ be the
\emph{last} vertex in $U$ (according to $\sigma$) and search for neighbors of $v$ in $T$.

\paragraph{Example of DFS.}

Consider the graph in Figure~\ref{fig:dfs_graph}.

\begin{figure}[H]
  \centering
  \includegraphics[width=0.55\textwidth]{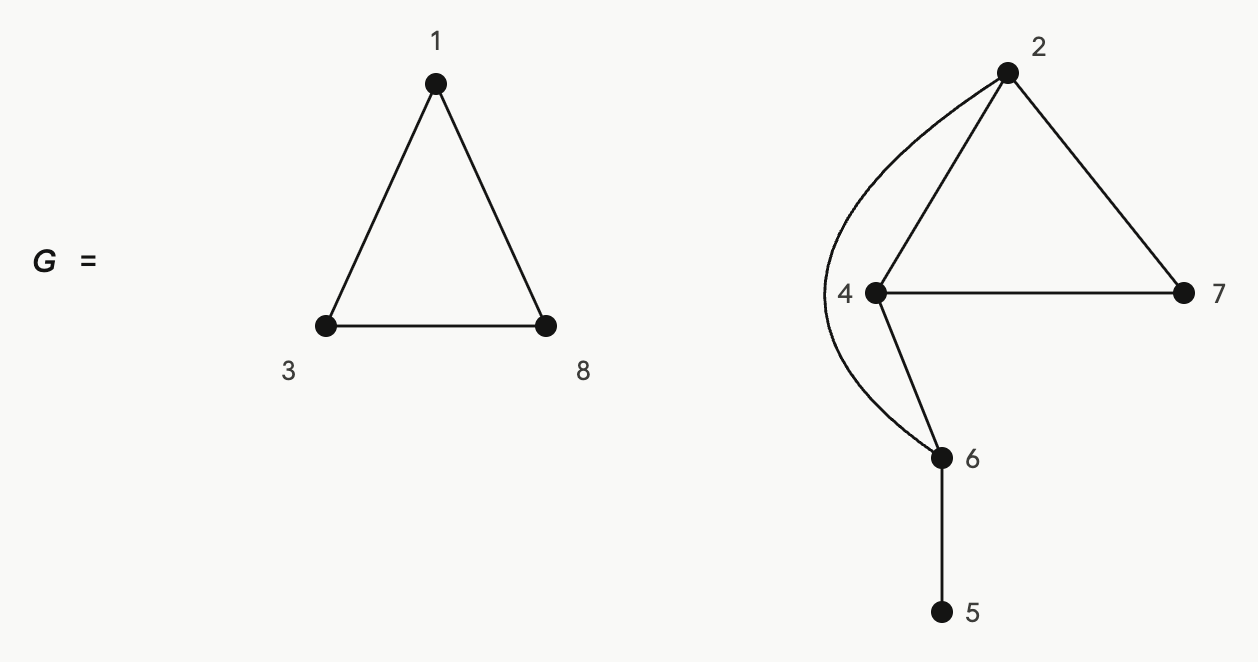}
  \caption{Example graph for the DFS walkthrough.}
  \label{fig:dfs_graph}
\end{figure}

\begin{enumerate}
  \item Start with $S = U = \emptyset$, $T = V$.
  \item Move $1$ from $T$ to $U$.
  \item Discover neighbor $3$ of $1$; move $3$ to $U$.
  \item Discover neighbor $8$ of $3$; move $8$ to $U$.
  \item Move $8$ to $S$.
  \item Move $3$ to $S$.
  \item Move $1$ to $S$.
  \item Move $2$ from $T$ to $U$.
  \item Discover neighbor $4$ of $2$; move $4$ to $U$.
  \item Discover neighbor $6$ of $4$; move $6$ to $U$.
  \item Discover neighbor $5$ of $6$; move $5$ to $U$.
  \item Move $5$ to $S$.
  \item Move $6$ to $S$.
  \item Discover neighbor $7$ of $4$; move $7$ to $U$.
  \item Move $7$ to $S$.
  \item Move $4$ to $S$.
  \item Move $2$ to $S$.
\end{enumerate}

The resulting spanning forest is shown in Figure~\ref{fig:dfs_forest}. Note that at every step, $U$ is a path in $G$.

\begin{figure}[H]
  \centering
  \includegraphics[width=0.55\textwidth]{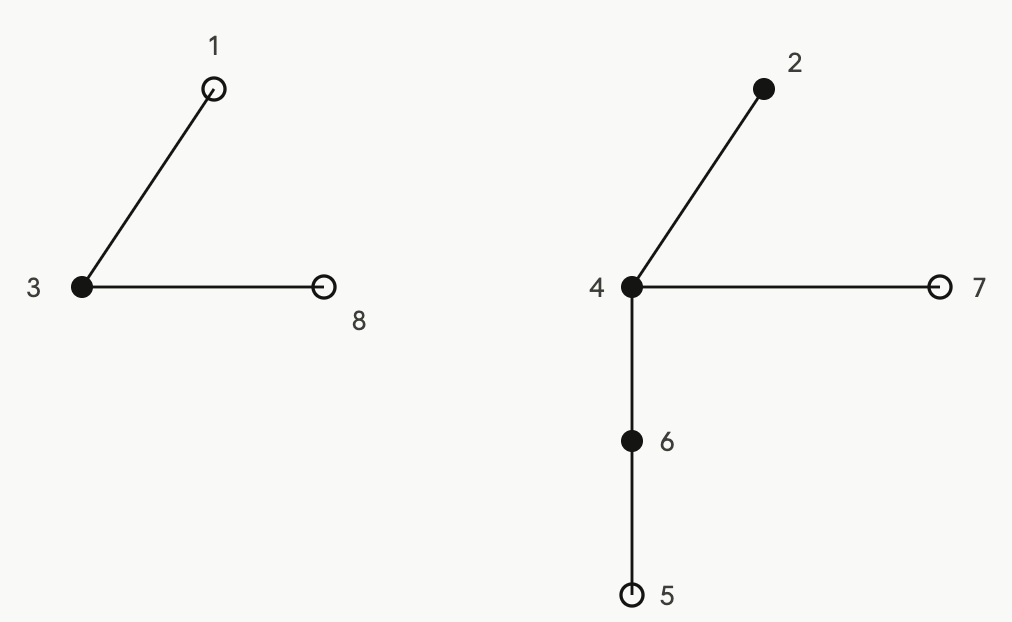}
  \caption{Spanning forest produced by DFS on the example graph.}
  \label{fig:dfs_forest}
\end{figure}

\paragraph{Basic properties of DFS.}
\begin{enumerate}
  \item At each step, exactly one vertex moves: either from $T$ to $U$, or from $U$ to $S$.\label{4: dfs 1}
  \item At each step, there are no edges between the current $S$ and the current $T$.\label{4: dfs 2}
  \item At each step, $U$ spans a \emph{path} in $G$, ordered by the order in which vertices were added to $U$ (Whenever we add $u$ to $U$, it is adjacent to the last vertex of $U$, and thus extends the current path. Removing the last vertex of $U$ shortens the path.).\label{4: dfs 3}
  \item The algorithm generates a spanning forest $F$ with $V(F) = V(G)$ and the same components as $G$. If $e = (u,v) \in E(G) \setminus E(F)$, then one of $u, v$ is an ancestor of the other in a tree of $F$ (see Figure~\ref{fig:dfs_back_edge}).\label{4: dfs 4}
\end{enumerate}

\begin{figure}[H]
  \centering
  \includegraphics[width=0.55\textwidth]{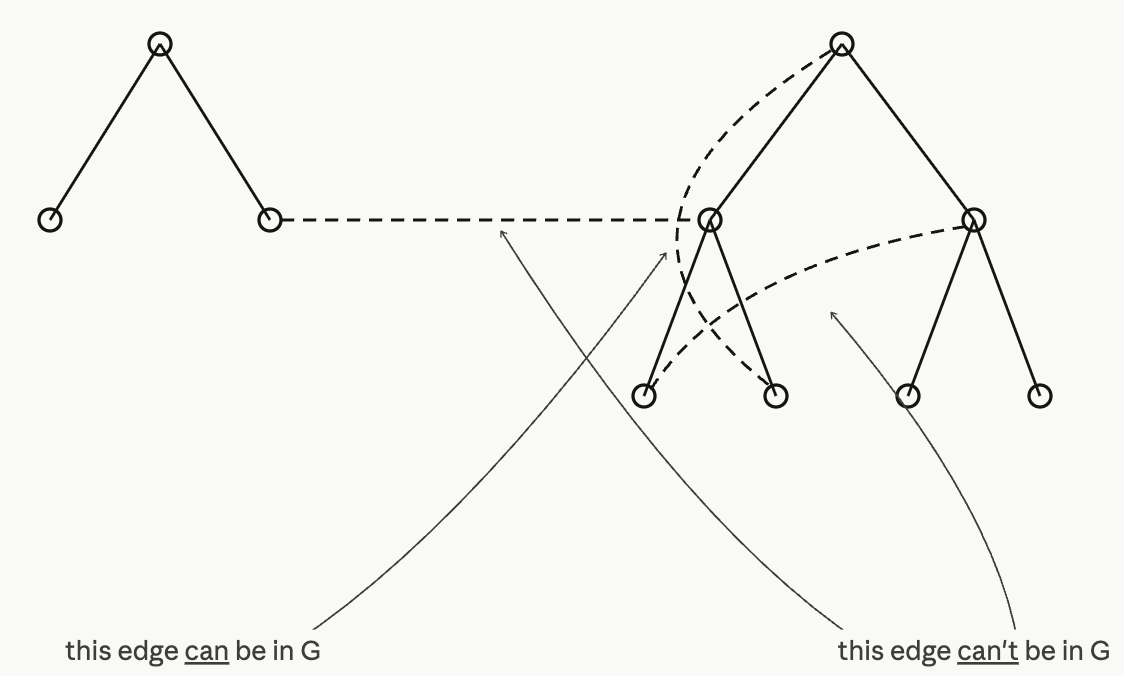}
  \caption{A back edge $(u,v) \in E(G)\setminus E(F)$: one endpoint is an ancestor of the other in $F$.}
  \label{fig:dfs_back_edge}
\end{figure}

\subsection{Using DFS to Find Long Paths and Cycles}

\begin{defn}
    For $G = (V,E)$ and $S \subseteq V$, the \emph{external neighborhood} of $S$ is \break $N(S) = \{u \in V \setminus S : u \text{ has a neighbor in } S\}$.
\end{defn}

\begin{thm}
Let $k, \ell>0$ be positive integers and $G = (V,E)$ a graph with $|V| > k$. Suppose that for every $S \subseteq V$ with $|S| = k$ we have $|N(S)| \ge \ell$. Then $G$ contains a path of length at least $\ell$.
\end{thm}

\begin{proof}
Fix an arbitrary order $\sigma$ on $V$ and run DFS on $(G, \sigma)$. Consider the first step for which $|S| = k$ (such a step exists because $|V| > k$ and vertices move one by one). By assumption $|N(S)| \ge \ell$. By Property~\ref{4: dfs 2}, there are no edges between $S$ and $T$, hence $N(S) \subseteq U$, and therefore $|U| \ge \ell$. By Property~\ref{4: dfs 3}, $U$ spans a path, so $G$ contains a path on at least $\ell$ vertices, i.e. of length at least $\ell - 1$. Moreover, the last operation was moving a vertex $v$ from $U$ to $S$; just before that move, $U$ contained $v$ plus at least $\ell$ other vertices, giving a path of length at least $\ell$.
\end{proof}

\begin{thm}[Ben-Eliezer, Krivelevich, Sudakov \cite{zbMATH06033283}]
Let $k < n$ be positive integers and $G = (V,E)$ a graph with $|V| = n$. Suppose that for every pair $A, B \subseteq V$ with $|A| = |B| = k$ and $A \cap B = \emptyset$ there is an edge in $G$ between $A$ and $B$. Then $G$ has a path of length at least $n - 2k + 1$ and a cycle of length at least $n - 4k + 4$ (provided this is positive).
\end{thm}

\begin{proof}
Fix an arbitrary order $\sigma$ on $V$ and run DFS on $(G, \sigma)$. By Property~\ref{4: dfs 1}, there exists a step at which $|S| = |T|$. By Property~\ref{4: dfs 2}, there are no edges between $S$ and $T$. The theorem's assumption then forces $|S| = |T| < k$, Hence:
\[
  |U| = n - |S| - |T| \ge n - 2(k-1) = n - 2k + 2.
\]
By Property~\ref{4: dfs 3}, $U$ spans a path $P$ in $G$ of length at least $n - 2k + 1$.

To obtain a long cycle, let $A$ be the first $k$ vertices of $P$ and $B$ the last $k$ vertices (see Figure~\ref{fig:long_cycle}). These sets are disjoint (as $n - 4k + 4>0$), and hence by the theorem's assumption, there is an edge between $A$ and $B$, which closes a cycle of length at least
\[
  |V(P)| - 2(k-1) \ge n - 4k + 4. \qedhere
\]
\end{proof}

\begin{figure}[H]
  \centering
  \includegraphics[width=0.65\textwidth]{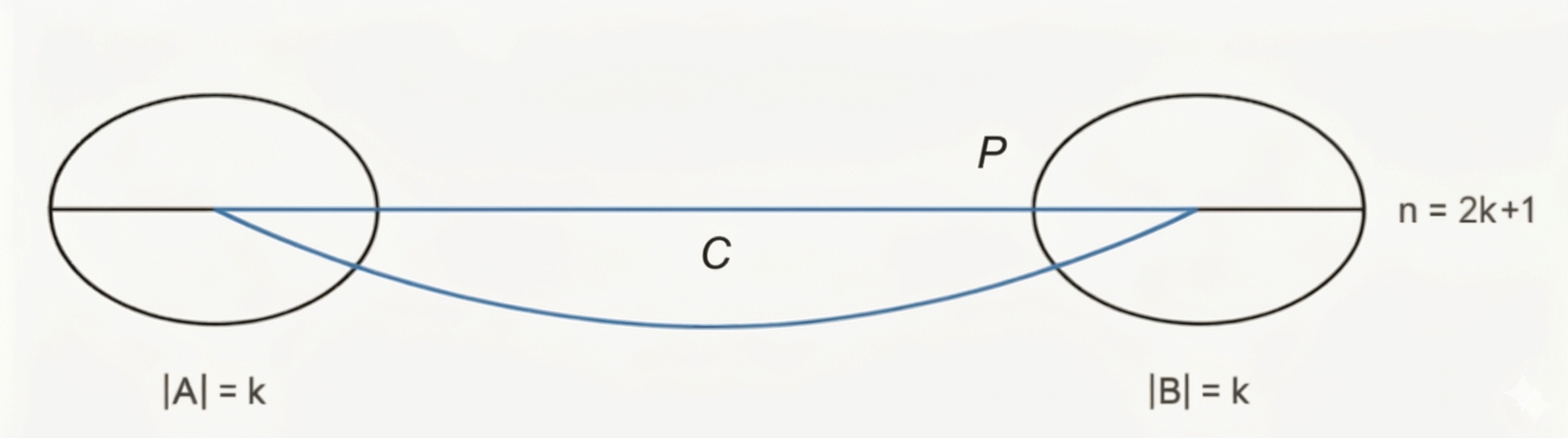}
  \caption{Closing a long cycle.}
  \label{fig:long_cycle}
\end{figure}

\begin{remark}
DFS is also applicable to directed graphs. Essentially, the same proof gives: if $G = (V,E)$ is a directed graph on $n$ vertices, $k$ is a positive integer, and for every ordered pair $(A, B)$ with $|A| = |B| = k$, $A \cap B = \emptyset$, there is an edge from $A$ to $B$ in $G$, then $G$ contains a directed path of length at least $n - 2k + 1$ and a directed cycle of length at least $n - 4k + 4$.
\end{remark}

\subsection{Using DFS for Random Graphs}

Let $G \sim G(n, p)$. We want to find long paths and cycles in $G$ whp\ by running DFS and exposing $G$ simultaneously.

Let $N = \binom{n}{2}$ and let $\bar{X} = (X_i)_{i=1}^{N}$ be an i.i.d.\ sequence of Bernoulli$(p)$ random bits. Fix an ordering on $[n]$ (e.g.\ the natural order) and run DFS on $G$ using $\bar{X}$ as follows: when DFS makes its $i$-th query ---``is $(u,v) \in E(G)$?''--- answer yes if $X_i = 1$, and no if $X_i = 0$. After DFS completes and all components are discovered, answer any remaining unqueried edges from the same sequence $\bar{X}$. The resulting graph is distributed as $G(n, p)$.

Note that a graph is (in a sense) a ``2-dimensional'' object, while the sequence $\bar{X} = (X_i)_{i=1}^{N}$ is 1-dimensional. The above approach turn DFS into a technique for ``dimension reduction" from 2 to 1.
\newpage
\phantomsection
\addcontentsline{toc}{chapter}{Lecture 5}
\lecture{5}{May 17, 2026}{Lecture 5}{Tom Guy}
\setcounter{chapter}{5}
\setcounter{section}{0}

\section{Long Paths and Cycles in $G \sim G(n,p)$}

\begin{thm}[Krivelevich, Sudakov \cite{zbMATH06214315}]
    For every sufficiently small $\varepsilon > 0$, a random graph $G \sim G\!\left(n, \tfrac{1+\varepsilon}{n}\right)$ whp contains a path of size at least $\tfrac{\varepsilon^2 n}{5}$.
\end{thm}

\begin{proof}
Run DFS on $G$, revealing the edges as the algorithm progresses, based on a vector $\bar{X} = (X_i)_{i=1}^{N}$, where $X_i \sim \mathrm{Bernoulli}\!\left(\tfrac{1+\varepsilon}{n}\right)$  i.i.d.
($N = \binom{n}{2}$).

\noindent At each step $t$: if $T \neq \emptyset$, the algorithm checks whether the current top-of-stack vertex in $U$ has a neighbor in $T$, in which case it moves a vertex from $T$ into $U$. For each $i \le t$, every value $X_i =1$ adds a vertex to $U$. This vertex might move later to $S$. Concretely, after $t$ steps, if $T \neq \emptyset$, we have that
\begin{align}\label{eq: U}
    |S \cup U| \geq \sum_{i=1}^{t} X_i, \qquad |U| \leq 1+ \sum_{i=1}^{t} X_i.
\end{align}

\noindent An explanation for the second inequality in \eqref{eq: U}: the process starts with $U$ as an empty set. The first vertex from $T$ that is added to $U$ is added for free, and then, at step $i\le t$, a new vertex moves from $T$ to $U$ due to having the corresponding random variable $X_j=1$ for every $j\le i\le t$.

Let $N_0 = \frac{\varepsilon n^2}{2}$ (think of $N_0$ as an integer number). According to Chebyshev's (or Chernoff's) inequality, we have whp that:
\[
\left|\sum_{i=1}^{N_0} X_i - \frac{\varepsilon n^2}{2} \cdot \frac{1+\varepsilon}{n}\right| \leq n^{2/3}.
\]
We aim to show that whp at time $N_0$ we have $|U| > \frac{\varepsilon^2 n}{5}$. First, we will show that whp, at $N_0$, it must be that $|S| \leq \frac{n}{3}$. Assume by contradiction that $|S|> \frac{n}{3}$. Then, look at the first time $t$ for which $|S| = \frac{n}{3}$. Notice that $t \leq N_0$, and hence
\[
|U| \leq 1+ \sum_{i=1}^{t} X_i \leq 1 + \sum_{i=1}^{N_0} X_i \leq \frac{n}{3},
\]
where the last inequality holds for sufficiently small $\varepsilon > 0$. Now,
\[
|T| = n - |S| - |U| \geq \frac{n}{3}.
\]
We know that all queries between $S$ and $T$ (at time $t$) have been asked (and returned a negative answer), thus:
\[
|S| \cdot |T| \ge \frac{n^2}{9} \quad \implies \quad
\frac{\varepsilon n^2}{2} = N_0 \geq t \geq |S| \cdot |T| \ge \frac{n^2}{9},
\]
a contradiction for sufficiently small $\varepsilon$.

\medskip
\noindent Now, $|S| \leq \frac{n}{3}$, we assume by contradiction that $|U| \leq \frac{\varepsilon^2 n}{5}$ (so in particular $T\ne \emptyset$). Since whp
\[
|S \cup U| \geq \sum_{i=1}^{N_0} X_i \geq \frac{\varepsilon(1+\varepsilon)n}{2} - n^{2/3},
\]
we have: 
\[
|S| \geq \frac{\varepsilon(1+\varepsilon)}{2} n - n^{2/3} - \frac{\varepsilon^2 n}{5}
= \frac{\varepsilon n}{2} + \frac{3\varepsilon^2 n}{10} - n^{2/3},
\]
and as before, all queries between $S$ and $T$ have been asked (and answered negatively). Thus:
\[
N_0 \geq |S| \cdot |T| \geq |S|\!\left(n - |S| - \frac{\varepsilon^2 n}{5}\right).
\]

Using $|S| \leq \frac{n}{3}$ (and the function $f(x) = x\left(n-x-\frac{\varepsilon^2 n}{5}\right)$ increases in $[0,n/3]$), we derive:
\begin{align*}
    N_0 &\geq \left(\frac{\varepsilon n}{2} + \frac{3\varepsilon^2 n}{10} - n^{2/3}\right)
    \left(n - \frac{\varepsilon n}{2} - \frac{\varepsilon^2 n}{5} + n^{2/3}\right) \\[6pt]
    &= \frac{\varepsilon n^2}{2} + \frac{3\varepsilon^2 n^2}{10} - \frac{\varepsilon^2 n}{4}
    - O(\varepsilon^3) n^2 \\[4pt]
    &= \frac{\varepsilon n^2}{2} + \frac{\varepsilon^2 n^2}{20} - O(\varepsilon^3) n^2
    > \varepsilon n^2 / 2,
\end{align*}
giving the desired contradiction. We have shown that typically at time $N_0$ we have $|U| > \frac{\varepsilon^2 n}{5}$.
Therefore, $G \sim G\!\left(n,\tfrac{1+\varepsilon}{n}\right)$ has a path of size at least $\frac{\varepsilon^2 n}{5}$.
\end{proof}

\begin{remark}\hfill \break 
\begin{enumerate}
    \item Ajtai, Komlós, Szemerédi \cite{zbMATH03769675} showed:
    In $G \sim G(n, c/n)$ for $c > 1$, whp there is a path with $\Theta(n)$ vertices.
    \item How to get a cycle from a path? For path $P$ of length $\Theta(n)$, take $o(n)$ vertices from each end of $P$ and connect them bu an edge whp using sprinkling. Thus, the questions on long paths and long cycles are the same asymptotically. 
    \item We have shown that for $\varepsilon > 0$,
    in $G \sim G\!\left(n, \tfrac{1+\varepsilon}{n}\right)$ there is typically a path/cycle of size $\frac{\varepsilon^2 n}{5}$.
    The largest a cycle could be is $|L_1|$, which is whp $ (1+o_\varepsilon(1))2\varepsilon n$. It is known that for $p = \frac{1+\varepsilon}{n}$, where $\varepsilon > 0$ small enough, whp the size of the longest cycle is $\Theta(\varepsilon^2 )n$. 
\end{enumerate} 
\end{remark}

\begin{Oquestion}
    What is the right constant before $\varepsilon^2 n$ for the typical length of a longest cycle in this regime?
\end{Oquestion}

\begin{remark}
    How does $L_1$ look whp?
\end{remark}

\begin{figure}[H]
    \centering
    \includegraphics[width=0.7\linewidth]{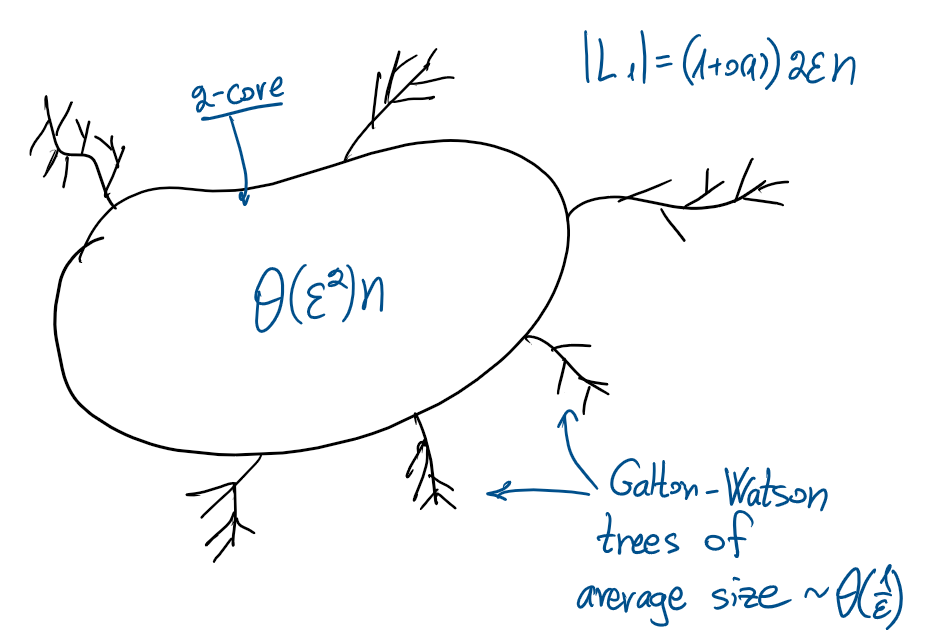}
\end{figure}

\subsection{Long Paths in Sparse Random Graphs}
We can ask ourselves which probability $p(n)$ does it take to have whp a cycle of length $n$ in $G(n,p)$, i.e. a Hamilton cycle.

\noindent\textbf{Answer:} The required $p(n)$ is $p = \frac{\ln n+\ln\ln n+\omega(n)}{n}$, where $\lim_{n \to \infty}\omega(n) = \infty$.

If we are willing to settle for a nearly spanning cycle (path), the required probability $p(n)$ is much lower, as given by the following theorem.

\begin{thm}[Ajtai, Komlós, Szemerédi \cite{zbMATH03769675}, Fernandez de la Vega \cite{zbMATH03693325}]\label{almost hamilton path}
    For every $\varepsilon > 0$, there exists $C = C(\varepsilon) > 0$ such that whp, for $G \sim G\!\left(n, \tfrac{C}{n}\right)$, a longest path has size of at least $(1-\varepsilon)n$.
\end{thm}

\begin{proof}
We have shown that for $n > 2k - 2$, and for a graph $G$ on $n$ vertices, if for every $A, B \subseteq V(G)$ with $|A| = |B| = k$, $A \cap B = \emptyset$, there is an edge between $A$ and $B$, then $G$ contains a path of size $n - 2k + 2$.

Let $k = \left\lfloor \frac{\varepsilon n}{2} \right\rfloor$. We will prove that for large enough $C = C(\varepsilon)$,
the graph $G \sim G(n, C/n)$ will contain whp at least one edge between every pair of disjoint sets of size $k$.

The probability that $G \sim G(n, C/n)$ does not satisfy the above property is at most:
\[
\underbrace{\binom{n}{k}}_{\substack{\text{choosing }A}}\underbrace{\binom{n-k}{k}}_{\substack{\text{choosing }B}}\underbrace{(1-p)^{k^2}}_{\substack{\text{no edges between }\\ A \text{ and }B}} \leq \binom{n}{k}^2 (1-p)^{k^2}
    \leq \left(\frac{en}{k}\right)^{2k} e^{-pk^2}
    \leq \left[\left(\frac{en}{k}\right)^2 e^{-pk}\right]^k,
\]
where we use $\binom{n}{k} \leq \left(\frac{en}{k}\right)^k$ and $1-p \leq e^{-p}$. Now, for our $k$, we have $(\frac{en}{k})^{2k} =  \left(O\!\left(\tfrac{1}{\varepsilon^2}\right)\right)^k$, thus by taking $C = \frac{5\ln(1/\varepsilon)}{\varepsilon}$ we get that the above probability is $o_n(1)$ (for small enough $\varepsilon$).

We thus conclude that for $G \sim G(n,p)$ with $p = \frac{5\log(2/\varepsilon)}{\varepsilon n}$, whp
$G$ contains a path of size at least $(1-\varepsilon)n$.
\end{proof}

\begin{remark}\label{almost hamilton}
    It is possible to show that for $C=C(\epsilon)$ large enough, a random graph $G\sim G\left(n,\frac{C}{n}\right)$ whp contains also a cycle of length $\ge (1-\epsilon)n$, either through sprinkling or the deterministic argument for the existence of a cycle of length at least $n-4k+4$.  
\end{remark}

\section{Hitting Time for Connectivity}

First, we define the notion of the hitting time for a non-trivial monotone property on graphs on $n$ vertices.
\begin{defn}
    Let $\tilde{G}=(G_i)_{i=0}^N$ be a graph process, and let $\mathcal{A}$ be a non-trivial monotone property on graphs on $n$ vertices. The \emph{hitting time} $\tau_{\mathcal{A}}(\tilde{G})$ of $\mathcal{A}$ with respect to $\tilde{G}$ is the minimum index $i$ such that $G_i \in \mathcal{A}$, that is,
    \[
    \tau_A(\tilde{G}) = \min\{i : G_i \in \mathcal{A}\}.
    \]
\end{defn}

\begin{remark}
    When $\tilde{G}$ is a random graph process, then $\tau_{\mathcal{A}}(\tilde{G})$ is a random variable equal to the first time $\tilde{G}$ hits $\mathcal{A}$, and it is possible to study its typical properties.
\end{remark}

\noindent We study the hitting time of a random graph process $\tilde{G}$ to become connected and denote it by $\tau_{c}(\tilde{G})$.

\medskip

\noindent\textbf{Characterizations of Graph Connectivity:}
\begin{enumerate}
    \item $G$ is connected iff $G$ contains a spanning tree.
    \item $G$ is connected iff for every partition $V(G)=A\cup B$ such that $\emptyset\ne A,B\subsetneq V(G)$, $G$ has an edge between $A$ and $B$.
\end{enumerate}

\noindent A possible reason for a graph $G$ to be disconnected is the existence of isolated vertices. We denote by $\tau_{1}(\tilde{G})$ the hitting time of the property $\delta(G)\ge 1$. Notice that $\tau_{C}(\tilde{G})\ge\tau_{1}(\tilde{G})$ deterministically, i.e. in any graph process. It turns out that for a random graph process, the hitting time for connectivity typically \textit{exactly} equals the hitting time for minimum degree $\geq 1$. In other words, typically as soon as the last isolated vertex disappears, the graph becomes connected:

\begin{thm}[Bollobás, Thomason \cite{zbMATH03943863}]
    In a random graph process $\tilde{G}$, whp $\tau_c(\tilde{G}) = \tau_1(\tilde{G})$.
\end{thm}

Before proving the statement we study the typical disappearance of isolated vertices in $G(n,p)$. Our goal is to find the threshold function $p\coloneqq p(n)$ for this (monotone) property.

For a vertex $v\in [n]$, we have $\mathbb{P}(v\text{ is isolated}) = (1-p)^{n-1}$, and hence:
\[
\mathbb{E}[\#\text{isolated vertices in } G(n,p)] = n \cdot (1-p)^{n-1} \sim ne^{-pn}.
\]
Notice that for $p = \frac{\ln n}{n}$, one has $ne^{-pn} = 1$. Hence we can guess that typically isolated vertices disappear entirely after $p = \frac{\ln n}{n}$, and we will show the following: 

\begin{thm}
     Let $p \coloneqq p(n) = \frac{\ln(n) + c(n)}{n}$. Then:
    \begin{enumerate}
        \item If $\lim_{n\to\infty} c(n) = \infty$, then whp $G \sim G(n,p)$ contains no isolated vertices.
        \item If $\lim_{n\to\infty} c(n) = -\infty$, then whp $G \sim G(n,p)$ has isolated vertices.
    \end{enumerate}
\end{thm}

\begin{remark}
    This theorem shows that $p = \frac{\ln n}{n}$ is a (very) sharp threshold function for the absence of isolated vertices.
\end{remark}

\begin{proof}
\underline{Case (1)}: Let $X$ denote the number of isolated vertices in $G \sim G(n,p)$. We have:

\[
\mathbb{E}[X] = n(1-p)^{n-1} = \frac{n}{1-p}(1-p)^n
\leq \frac{ne^{-pn}}{1-p}
= ne^{-\frac{\ln n + c(n)}{n} \cdot n}
= \frac{e^{-c(n)}}{1-p}
\xrightarrow{n\to\infty} 0.
\]
By Markov's inequality whp $X = 0$ (no isolated vertices).

\noindent \underline{Case (2)}: $\lim_{n\to\infty} c(n) = -\infty$, $p(n) = \frac{\ln n + c(n)}{n}$.

\noindent We need a lower bound for the expectation. Using $1 - x \geq e^{-\frac{x}{1-x}}$ for $0 \leq x < 1$, we get:
\begin{align*}
    \mu\coloneqq \mathbb{E}[X] &= n(1-p)^{n-1}
    \geq ne^{-(n-1)\frac{p}{1-p}}
    = n \cdot e^{-\frac{np}{1-p}} \cdot e^{\frac{p}{1-p}}
    \geq ne^{-\frac{np}{1-p}} \\[4pt]
    &= ne^{-\frac{\ln n + c(n)}{1-p}}
    = \exp\left(\ln n-\frac{\ln n}{1-p} - \frac{c(n)}{1-p}\right)\\
    &=\exp\left(-\frac{p\ln n}{1-p} - \frac{c(n)}{1-p}\right)
    \xrightarrow{n\to\infty} \infty.
\end{align*}

Notice that we can assume that $|c(n)|\le \ln\ln n$ (since the property of having no isolated vertices is monotone (increasing), hence by showing it for small $|c(n)|$ implies to larger $|c(n)|$).

We use the second moment method. For this, we compute $\mathbb{E}[X^2]$. Notice that $X = \displaystyle\sum_{i=1}^{n} X_i$, where $X_i$ is the indicator of vertex $i$ being isolated:
\[
X_i = \begin{cases} 1 & d(i) = 0; \\ 0 & d(i) > 0. \end{cases}
\]
Then:
\begin{align*}
    \mathbb{E}[X^2]
    &= \mathbb{E}\!\left[\left(\sum_{i=1}^{n} X_i\right)^{\!2}\right]
    = \sum_{i=1}^{n} \mathbb{E}[X_i^2] + \sum_{i \neq j} \mathbb{E}[X_i X_j]= n\mathbb{E}[X_1^2]+n(n-1)\mathbb{E}[X_1 X_2]\\
    &= n(1-p)^{n-1}+n(n-1)(1-p)^{2n-3}\le \mu +n^2\cdot \frac{(1-p)^{2n-2}}{1-p}=\mu+\frac{\mu^2}{1-p},
\end{align*}
where the penultimate equality holds due to obvious symmetry. Therefore:
\[
\mathrm{Var}(X) = \mathbb{E}[X^2] - (\mathbb{E}[X])^2
= \mu + \frac{\mu^2}{1-p} - \mu^2 = \frac{\mu^2 p}{1-p} + \mu = o(\mu^2).
\]
By Chebyshev's inequality:
\[
\mathbb{P}(X = 0) \leq \frac{\mathrm{Var}(X)}{\mu^2} = \frac{o(\mu^2)}{\mu^2} = o(1),
\]
hence whp $X > 0$, i.e. $G$ has some isolated vertices.
\end{proof}
\newpage
\phantomsection
\addcontentsline{toc}{chapter}{Lecture 6}
\lecture{6}{May 31, 2026}{Lecture 6}{Rom Amiaz}
\setcounter{chapter}{6}
\setcounter{section}{0}

\section{Hitting Time for Connectivity - Continued}

\begin{thm}[Bollobás, Thomason \cite{zbMATH03943863}]
    In a random graph process $\tilde{G}$, whp $\tau_c(\tilde{G}) = \tau_1(\tilde{G})$.
\end{thm}

\begin{proof}
Define 
\[
m_{1}=\frac{n-1}{2}(\ln n-\ln\ln n);\quad p_{1}=\frac{m_{1}}{N}=\frac{\ln n-\ln\ln n}{n},
\]

\[
m_{2}=\frac{n-1}{2}(\ln n+\ln\ln n);\quad p_{2}=\frac{m_{2}}{N}=\frac{\ln n+\ln\ln n}{n}.
\]

We know that for a random graph process $\tilde{G}=(G_{i})_{i=0}^{N}$
we have: $G_{i}\sim G(n,i)$, for every $0\leq i\leq N$. We will prove that for a random graph process, whp:

\begin{enumerate}
    \item In $G_{m_{1}}$: there are isolated vertices, but no more than $\ln^{2}n$ such vertices. All other vertices belong to the same connected component $L_1$;\label{prop: lec6_1}
    \item In $G_{m_{2}}$: there are no isolated vertices;\label{prop: lec6_2}
    \item If we denote the set of isolated vertices in $G_{m_{1}}$by $V_{0}$, then the process $\tilde{G}=(e_1,\ldots ,e_N)$ satisfies the following: for every $m_{1}< i\leq m_{2}$, $e_i\nsubseteq V_0$. \label{prop: lec6_3}
\end{enumerate}

First, we claim that if a graph process $\tilde{G}$ satisfies these properties, then $\tau_{c}(\tilde{G})=\tau_{1}(\tilde{G})$ (deterministically). Indeed, from property \ref{prop: lec6_1}, $\tau_{1}(G)>m_{1}$. From property \ref{prop: lec6_2}, $\tau_{1}(G)\leq m_{2}$. Hence, for every $v\in V_{0}$, there exists $m_{1}<i\leq m_{2}$ such that $v\in e_{i}$.
From property \ref{prop: lec6_1}: $L_{1}=G-V_{0}$, thus at time $m_{2}$, by property \ref{prop: lec6_3},
all vertices of $V_{0}$ have joined $L_{1}$. Therefore, if we look
at the time $\tau_{1}(G)$ (when the last isolated vertex receives
an edge), then the graph becomes connected at that exact moment, and therefore $\tau_{1}(G)=\tau_{c}(G).$ 

Now, we prove properties \ref{prop: lec6_1}-\ref{prop: lec6_3}:

\medskip
\underline{Proof of property \ref{prop: lec6_1}}: Denote by $X$ the number of isolated vertices in $G_{m_{1}}.$
We prove that whp $0<X\leq\ln^{2}n$. We already proved that $X>0$ whp in $G(n,p_{1})$. Note that:
\[
\mathbb{E}[X]=n\cdot(1-p_1)^{n-1}\leq n\exp(-p_1\cdot(n-1))\approx n\exp(-p_1n)=\ln n.
\]
Therefore, by Markov's inequality, we derive that $\mathbb{P}(X\geq\ln^{2}n)<\frac{1}{\ln n}$, hence
whp $X<\ln^{2}n$. The properties $X\leq\ln^{2}n$ and $X>0$
are monotone (increasing and decreasing respectively), therefore
we have that whp $0<X\leq\ln^{2}n$ in $G_{m_{1}}\sim G(n,m_{1})$
as well (by Proposition \ref{proposition: comparison gnp_gnm - monotone}). It is left to prove that whp $G(n,m_{1})$ contains no connected components $C$ such that $2\leq\left|C\right|=k\leq\frac{n}{2}$. For a fixed vertex set $U$ with $\left|U\right|=k$, we have that:
\[
\mathbb{P}\left(U \text{ is\:a\:connected\:component}\right)\leq k^{k-2}\cdot p_1^{k-1}(1-p_1)^{k\cdot(n-k)}.
\]
However, for large values of k, it is better instead of the above estimate to look at the
probability that $U$ is disconnected from $[n]\setminus U$:
\[
\mathbb{P}\left(U \text{ is\:a\:connected\:component}\right)\le \mathbb{P}\left(U \text{ is disconnected from } G-U\right)\leq(1-p_1)^{k\cdot(n-k)}.
\]
Overall we get:
\[
\mathbb{P}(\exists \text{\:connected\:component } C:2\leq\left|C\right|\leq\frac{n}{2})=\sum_{k=2}^{n/2}\binom{n}{k}\cdot(1-p_1)^{k\cdot(n-k)}\cdot\min\left\{1,k^{k-2}\cdot p_1^{k-1}\right\}=\sum_{k=2}^{n/2}u_{k}.
\]
Whenever $2\leq k\leq n^{0.9}$:
\begin{align*}
    u_{k}&\leq\left(\frac{en}{k}\right)^{k}k^{k-2}\cdot p_{1}^{k-1}\cdot\exp(-p_{1}k(n-k))=\frac{1}{k^{2}p}(enp_{1}\cdot \exp(-(n-k)\cdot p_{1}))^{k}\\
    &\leq n\cdot\left(3\ln n\cdot\exp\left(-0.9n\cdot\frac{\ln n}{n}\right)\right)^{k},
\end{align*}

where the last inequality holds since $k^{2}p\leq n^{0.8+o(1)}<n$
and since $n-k\geq 0.9n$. As $k\geq 2$ we get:
\[
n\cdot\left(3\ln n\cdot\exp\left(-0.9n\cdot\frac{\ln n}{n}\right)\right)^{k}\leq n\cdot(n^{0.9}\cdot3\ln n)^{k}\leq n^{-0.3k}.
\]
Now we look at the case where $n^{0.9}\leq k\leq\frac{n}{2}$:
\[
u_{k}\leq\left(\frac{en}{k}\right)^{k}\exp(-k\cdot(n-k)\cdot p_{1})\leq \left(en^{0.1}\cdot\exp\left(\frac{-np_{1}}{2}\right)\right)^{k}\leq n^{-0.3k}.
\]
Therefore, we get (from convergence of geometric series): 
\[
\sum_{k=2}^{n/2}u_{k}\leq \sum_{k=2}^{n/2}n^{-0.3k}=O(n^{-0.6}).
\]
From the relation between $G(n,m)$ and $G(n,p)$ and from the fact
that $O(\sqrt{m_1})\cdot O(n^{-0.6})=o(1)$, we have that whp there
are no connected components in $G(n,m_{1})$ of size $2\leq k\leq\frac{n}{2}$.
That concludes the proof of property \ref{prop: lec6_1}.

\medskip
\underline{Proof of property \ref{prop: lec6_2}}: we show that whp $G_{m_{2}}\sim G(n,m_{2})$ has no isolated vertices. Denote by $X$ the number of isolated vertices in $G(n,p_{2})$. Notice
that:
\[
\mathbb{E}[X]=n(1-p_{2})^{n-1}\leq n(1+o(1))\cdot\exp(-\ln n-\ln\ln n)=o(1).
\]
Therefore, by Markov's inequality whp $X=0$. From monotonicity we get the same result for $G(n,m_{2})$.

\medskip
\underline{Proof of property \ref{prop: lec6_3}}: set $V_{0}$ to be the set of isolated vertices in $G_{m_{1}}$. We have shown that
whp $\left|V_{0}\right|\leq\ln^{2}n.$ We prove that whp
no edge $e_{i}:m_{1}< i\leq m_{2}$ is contained in $V_{0}$. Indeed,
\[
\mathbb{P}(e_{i}\subseteq V_{0}|e_{1},\ldots e_{i-1})\leq\frac{\binom{\left|V_{0}\right|}{2}}{N-(i-1)}=O\left(\frac{\ln^{4}n}{n^{2}}\right).
\]
Therefore, the probability for any $e_{i}$ to be contained in $V_{0}$
for $m_{1}< i\leq m_{2}$ is at most (by the union bound):
\[
(m_{2}-m_{1})\cdot O\left(\frac{\ln^{4}n}{n^{2}}\right)=O\left(\frac{\ln^{4}n\cdot\ln\ln n}{n}\right)=o(1).
\]
\end{proof}

\begin{corollary} [Erd\H{o}s, Rényi \cite{zbMATH03150484}] Assume that $G\sim G(n,m),$$m=\frac{n-1}{2}\cdot(\ln n+c(n))$ then:
    \begin{enumerate}
        \item If $c(n)\rightarrow-\infty$, then whp G is not connected.
        \item If $c(n)\rightarrow+\infty$, then whp G is connected.
    \end{enumerate}
\end{corollary}

\begin{proof}
We showed that in a random graph process $\tilde{G}$ we have whp $\tau_{1}(\tilde{G})=\tau_{c}(\tilde{G})$, hence

\begin{enumerate}
    \item $m=\binom{n}{2}\cdot\frac{(\ln n+c(n))}{n}$, $c(n)\rightarrow-\infty$, then whp $G$ has isolated vertices, therefore not connected.
    \item $m=\binom{n}{2}\cdot\frac{(\ln n+c(n))}{n}$, $c(n)\rightarrow\infty$, then whp $G$ has no isolated vertices, therefore as $\tau_{1}(\tilde{G})=\tau_{c}(\tilde{G})$ whp we have that $G$ is connected.
\end{enumerate}
\end{proof}

\begin{corollary}
Let $G\sim G(n,p)$ with $p=\frac{(\ln n+c(n))}{n}$. Then:
\begin{enumerate}
    \item If $c(n)\rightarrow-\infty$, then whp G is not connected.
    \item If $c(n)\rightarrow+\infty$, then whp G is connected.
\end{enumerate}
\end{corollary}

\begin{proof}
From the last theorem and the equivalence between $G(n,m)$ and $G(n,p)$
(connectivity is a monotone property).
\end{proof}

\section{Perfect Matchings in Random Graphs}
\begin{defn}
For a graph $G=(V,E)$, a \textit{matching} in $G$ is a set of pairwise-disjoint edges, i.e. $M\subseteq E$, $\forall e_{1}e_{2}\in M:e_{1}\cap e_{2}=\emptyset$. A \textit{perfect matching} is a matching of size $\frac{\left|V\right|}{2}.$
(Can exist only when $\left|V\right|$ is even.)
\end{defn}

\begin{question}
If $G\sim G(n,p)$ (and n is even), what is the minimal $p(n)$ s.t. $G$ contains a perfect matching whp?
\end{question}

Recall that if $p=\frac{\ln n-\omega (n)}{n}$, where $\omega(n)\rightarrow-\infty$, then whp $G$ has isolated vertices, and therefore doesn't have a perfect matching. The following theorem shows the positive side of this question:

\begin{thm}\label{thm: perfectM - ER} [Erd\H{o}s, Rényi \cite{zbMATH03323793}] 
Suppose $p=\frac{\ln n+\omega(n)}{n}$, where $\omega(n)\rightarrow+\infty$ (and n is even). Then whp there exists a perfect matching in $G\sim G(n,p)$. 
\end{thm}

There exists a stronger theorem on the hitting time of a perfect matching. Let us denote by $\tau_{M}(\tilde{G})$ the hitting time of the property ``$G$ contains a perfect matching".

\begin{thm}[Bollobás, Thomason \cite{zbMATH03943863}] \label{PM Bollobas Thomason}
Whp, in a random graph process $\tilde{G}$, $\tau_{M}(\tilde{G})=\tau_{1}(\tilde{G})$
(if n is even).
\end{thm}

\begin{proof}[Proof of Theorem \ref{thm: perfectM - ER}]

$G\sim G(n,p)$, $p=\frac{\ln n+\omega(n)}{n}$, $\omega(n)\rightarrow+\infty$. Notice that (due to monotonicity) we can assume that $\omega(n)\leq\ln n$. Fix a vertex $i\in[n]$, and consider $X=d_{G}(i)\sim Bin(n-1,p)$. Note that
\[
\mathbb{E}[X]=p\cdot(n-1)\geq\ln n (1-o(1)).
\]
By Chernoff's inequality, we get:
\[
\mathbb{P}(d_{G}(i)\le c\ln n)\leq\exp(-\Theta(\ln n))=n^{-\Theta(1)}.
\]
Assume $c$ is small enough and define:
\[
\text{SMALL}=\{i:d_{G}(i)\leq c\ln n\},\:\text{LARGE}=[n]\setminus \text{SMALL}.
\]

\begin{lemma}\label{lem: lec6 - main}
Suppose $p=\frac{\ln n+\omega(n)}{n}$, $\omega(n)\rightarrow+\infty$ and suppose $G\sim G(n,p)$. Then:
\begin{enumerate}
    \item $\delta(G)\geq 1$;\label{prop: lec6 - 1}
    \item $\left|\text{SMALL}\right|\leq n^{0.2}$;\label{prop: lec6 - 2}
    \item $\forall u\neq v\in \text{SMALL}:dist_{G}(u,v)\geq5$ (this can be any fixed number);\label{prop: lec6 - 3}
    \item $\forall V_{0}\subseteq \text{LARGE}:\left|V_{0}\right|\leq\frac{n}{\sqrt{\ln n}}\Rightarrow\left|N(V_{0})\right|\geq\left|V_{0}\right|\cdot(\ln n)^{1/4}$;\label{prop: lec6 - 4}
    \item $\forall A,B\subseteq[n]$ disjoint sets s.t. $\left|A\right|=\left|B\right|=\frac{n}{2\sqrt{\ln n}},$ we have: $E_{G}(A,B)\neq\emptyset$.\label{prop: lec6 - 5}
\end{enumerate} 
\end{lemma}
\begin{proof}
\underline{Proof of property \ref{prop: lec6 - 1}}: We have already proven this property.

We use (multiple times) the following claim. There exists
a small enough $c>0$ s.t. for every $ k>0:$
\[
\mathbb{P}(Bin(n-k,p)\leq c\ln n)\leq n^{-0.9}.
\]
Indeed, that probability is:
\begin{align*}
    \sum_{i=1}^{c\ln n}\mathbb{P}(Bin(n-k,p)=i)&\leq(1+c\ln n)\mathbb{P}(Bin(n-k,p)=c\ln n)\\
    &=O(\ln n)\cdot\binom{n-k}{c\ln n}\cdot p^{c\ln n}\cdot(1-p)^{n-k-c\ln n}\\
    &\leq O(\ln n)\cdot\left(\frac{enp}{c\ln n}\right)^{c\ln n}\cdot\exp(-p(n-k-c\ln n))\\
    &\leq O(\ln n)\cdot\left(\frac{2e}{c}\right)^{c\ln n}e^{-0.95\ln n}=O(\ln n)\cdot\left(\left(\frac{2e}{c}\right)^{c}\cdot e^{-0.95}\right)^{\ln n}\\
    &\le n^{-0.9},
\end{align*}
where the first inequality comes from the monotonicity of the tail
in the binomial distribution and the last inequality holds for a sufficiently small $c$.

\medskip
\underline{Proof of property \ref{prop: lec6 - 2}}: from symmetry,
\[
\mathbb{E}[\left|\text{SMALL}\right|]=n\cdot\mathbb{P}(d(1)\leq c\ln n)\leq n\cdot n^{-0.9}=n^{0.1}.
\]
Therefore, by Markov's inequality, $\mathbb{P}(\left|\text{SMALL}\right|\geq n^{0.2})\leq n^{-0.1}$,
so whp $\left|\text{SMALL}\right|\leq n^{0.2}$.

\medskip
\underline{Proof of property \ref{prop: lec6 - 3}}: we start with distance 1:
\begin{align*}
    \mathbb{P}(\exists u,v\in \text{SMALL}:dist_{G}(u,v)=1)&=\mathbb{P}(\exists u,v\in \text{SMALL}:(u,v)\in E(G))\\
    &\leq\underbrace{\binom{n}{2}}_{\text{choosing } u,v}\cdot \underbrace{p}_{(u,v)\in E(G)}\cdot \underbrace{\left(\mathbb{P}(Bin(n-2,p)\leq c\ln n)\right)^{2}}_{\text{pay for } u,v\in \text{SMALL}}\\
    &=O(n\ln n)\cdot \left(n^{-0.9}\right)^2=o(1).
\end{align*}

For distance 2:
\[
\mathbb{P}(\exists u,v\in \text{SMALL}:dist_{G}(u,v)=2)\leq\binom{n}{3}\cdot p^{2}\cdot \left(\mathbb{P}(Bin(n-3,p)\leq c\ln n)\right)^{2}=O(n\ln^{2}n)\cdot n^{-1.8}=o(1).
\]
Similarly, we can show that whp there are no $u,v\in \text{SMALL}$ at
distance 3 or 4 from each other. \renewcommand{\qedsymbol}{}
\end{proof}
\renewcommand{\qedsymbol}{}
\end{proof}
\newpage
\phantomsection
\addcontentsline{toc}{chapter}{Lecture 7}
\lecture{7}{June 7, 2026}{Lecture 7}{Aner Mash}
\setcounter{chapter}{7}
\setcounter{section}{0}

\section{Perfect Matchings in Random Graphs - Continued} 

We continue the proof of Theorem \ref{thm: perfectM - ER}.

\underline{Proof of property \ref{prop: lec6 - 4}}:
Assume that $k\coloneqq |V_0| \leq \frac{n}{\ln^{1/2}n}$. $V_0 \subseteq LARGE$, hence, for every $v\in V_0$, $d(v) \geq c\cdot \ln n$. Thus, $V_0$ is incident to at least $|V_0|\cdot \frac{c\cdot \ln n}{2}=\frac{k\cdot \ln n}{2}$ edges. These edges lie inside $V_0$ or between $V_0$ and $U\coloneqq N_G(V_0)$ --- the external neighborhood of $V_0$ in $G$.


\begin{figure}[h]
    \centering
    \includegraphics[width=0.5\textwidth]{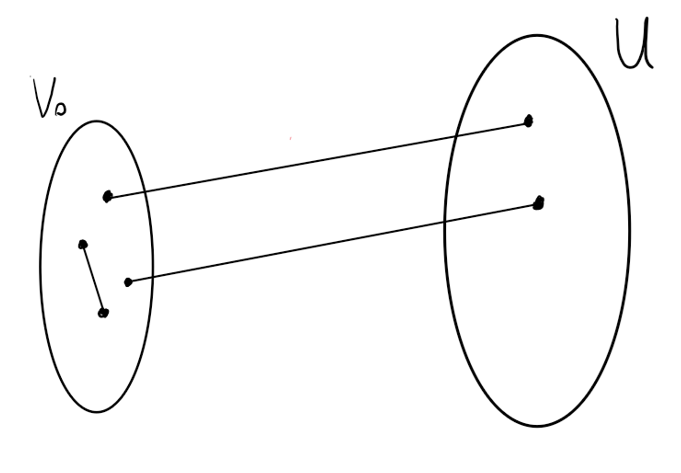}
\end{figure}

\begin{align*}
    \mathbb{P} \left(\exists V_0 \subseteq LARGE,\; |V_0|=k,\;  1\leq k \leq \frac{n}{\ln^{1/2} n},\;  |N(V_0)| \leq k\cdot \ln^{1/4}n\right)\\
    \leq \sum_{k=1}^{n/\ln^{1/2}n} \underbrace{\binom{n}{k}}_{\text{choosing } V_0} \underbrace{\binom{n-k}{k\cdot \ln^{1/4}n}}_{\text{choosing } N_G(V_0) } \underbrace{\binom{\binom{k}{2} +k\cdot k\ln^{1/4}n}{\frac{kc\cdot \ln n}{2}}}_{\text{choosing the edges touching } V_0}  \cdot p^\frac{k\cdot c\ln n}{2}\\
    \leq \sum_{k=1}^{n/\ln^{1/2}n} \left( \frac{3n}{k\cdot \ln^{1/4}n} \right)^{k\cdot \ln^{1/4}n} \left( \frac{3k^2\cdot \ln^{1/4}n}{\frac{k\cdot c\ln n}{2}} \cdot p \right)^\frac{k\cdot c\ln n}{2}\\
    \leq \sum_{k=1}^{n/\ln^{1/2}n} \left[\left(\frac{3n}{k\cdot \ln^{1/4}} \right)^{k\cdot \ln^{1/4}n} \cdot \left( \frac{6k}{c\ln^{3/4}n} \cdot \frac{2\ln n}{n} \right)^\frac{c\ln n}{2}\right]^{k}\\
    = \sum_{k=1}^{n/\ln^{1/2}n} \left[ \left( \frac{3n}{k\cdot \ln^{1/4}n} \right)^{\ln^{1/4}n} \cdot \left(\frac{12}{c} \cdot \frac{k\ln^{1/4}n}{n}\right)^{\frac{c\ln n}{2}} \right] ^k\\
    \leq \sum_{k=1}^{n/\ln^{1/2}n} \left[ \left( \frac{12}{c} \cdot \frac{k\ln^{1/4}n}{n} \right)^ {\frac{c\ln n}{4}} \right]^k = o(1).
\end{align*}

\underline{Proof of property \ref{prop: lec6 - 5}}: It suffices to show that whp, $G\sim G(n,p)$ contains an edge between any two sets $A$ and $B$ satisfying:
\begin{equation*}
    A,B \subseteq[n],\quad A\cap B = \emptyset,\quad |A|=|B|= \frac{n}{2 \ln^{1/2}n}.
\end{equation*}

The probability that this does not happen is at most:
\begin{align*}
    \underbrace{\binom{n}{\frac{n}{2\ln^{1/2}n}}^2}_{\text{choosing } A,B} \cdot \underbrace{\left(1-p\right)^{\left(\frac{n}{2\ln^{1/2}n}\right)^2}}_{E_G(A,B)=\emptyset}&\leq \left( 2e\cdot \ln^{1/2}n\right)^{\frac{n}{\ln^{1/2}n}} \cdot e^{-{\frac{\ln n}{n}}\cdot \frac{n^2}{4\ln n}}=\left(2e\cdot \ln^{1/2}n \right)^{\frac{n}{\ln^{1/2}n}} \cdot e^{-\frac{n}{4}} \\
    &\leq \left(\ln n\right) ^\frac{n}{\ln^{1/2}n} \cdot e^{-\frac{n}{4}} =\exp \left\{ \frac{n\cdot \ln\ln n}{\ln^{1/2}n}-\frac{n}{4} \right\} = o(1).
\end{align*}
This completes the proof of Lemma \ref{lem: lec6 - main}.\hfill \qedsymbol
\medskip

We now prove deterministically that if $G$ is a graph on $n$ vertices ($n$ is even and sufficiently large) satisfying properties \ref{prop: lec6 - 1}-\ref{prop: lec6 - 5}, then $G$ contains a perfect matching, proving Theorem \ref{thm: perfectM - ER}. To that end, we use Tutte's Theorem. Before stating Tutte's Theorem, we recall that for a graph $G=(V,E)$, we denote by $o(G)$ the number of odd connected components (i.e., of odd size) in $G$.

\begin{thm}[Tutte \cite{zbMATH03045906}]
    A graph $G=(V,E)$ contains a perfect matching if and only if for every $S\subseteq V$, we have  $o(G-S) \leq |S|$.
\end{thm}

We use this theorem to argue about existence of the desired matching in $G$. First, for every $v \in \text{SMALL}$, we arbitrarily choose an edge $e_v \in E(G)$ such that $v \in e_v$
(this is possible since property \ref{prop: lec6 - 1} asserts that $\delta(G) \geq 1$).

Note that the chosen edges $e_v$ are disjoint, because if they are not disjoint, then there exist vertices in $\text{SMALL}$ at distance of at most 2 from each other, which is a contradiction to property \ref{prop: lec6 - 3}.
\begin{figure}[H]
    \centering
    \includegraphics[width=0.3\textwidth]{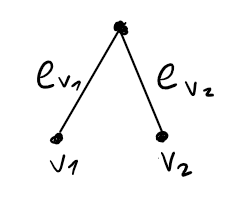}
\end{figure}
Therefore, the set of chosen edges 
$$M_0 = \{e_v\mid v\in SMALL\} \subseteq E(G)$$
is a matching of size $|M_0| = |SMALL| \leq n^{0.2}$ (By property \ref{prop: lec6 - 2}). 

\begin{figure}[H]
    \centering
    \includegraphics[width=0.6\textwidth]{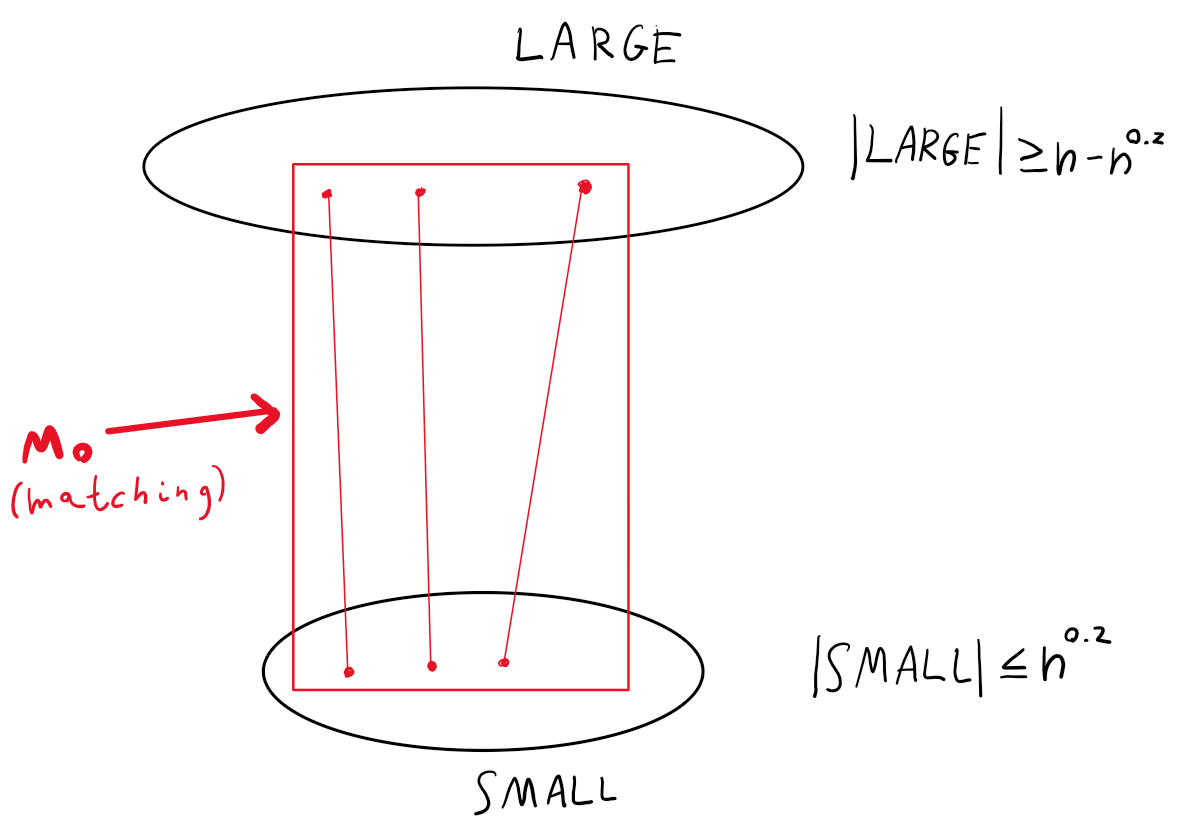}
\end{figure}

We construct a perfect matching $M$ in $G$ such that $M_0 \subseteq M$.
Let $W \subseteq V$ denote the union of the edges of $M_0$.
$$|W|=2\cdot |M_0|= 2\cdot |SMALL| \leq 2\cdot n^{0.2}.$$

For every $v \in V\setminus W$, observe that $v \in LARGE$, hence $d_G(v) \geq c\cdot \ln n$.
Moreover, $v$ has at most one neighbor in $W$, since otherwise there would exist two distinct vertices in $\text{SMALL}$ at distance of at most 4 from each other, contradicting property \ref{prop: lec6 - 3}).

\begin{figure}[H]
    \centering
    \includegraphics[width=0.3\textwidth]{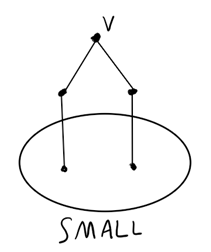}
\end{figure}

Denote $U=V\setminus W$. $U$ satisfies $|U|=n-o(n)$.
The properties of $G[U]$ are:
\begin{enumerate}
    \item [(a)] $\forall V_0 \subseteq U,\ |V_0| \leq \frac{n}{\ln^{1/2}n}:\ |N(V_0,U)| \geq |N(V_0)| - |V_0| \geq |V_0| \cdot \ln^{1/8}n$ (by property \ref{prop: lec6 - 4}).\label{prop: a}
    \item [(b)] For every $A,B\subseteq U$ such that $ A\cup B = \emptyset$ and $|A|,|B| \geq \frac{n}{2\ln^{1/2}n}$, there is an edge between $A$ and $B$ (by property \ref{prop: lec6 - 5}).\label{prop: b}
\end{enumerate}

We show the existence of a perfect matching in $G_1 := G[U]$ by verifying Tutte's condition for $G_1$. Let $S \subseteq U$ with $ k\coloneqq |S| \geq 0$. We need to verify that $o(G_1[U\setminus S]) \leq |S|.$

If $k=0$, we have to show that $G_1$ is connected. Indeed, if $C$ is a connected component of $G_1$, then $|C|>\frac{n}{\ln^{1/2}n}$ (every small set expands itself). Otherwise, $G_1$ has two connected components $C_1 \neq C_2,\ |C_1|,|C_2| \geq \frac{n}{\ln^{1/2}n}$, contradicting property \ref{prop: lec6 - 2} (there is an edge between $C_1$ and $C_2$).

So we can assume that $k>0$. Let $C_1,...,C_\ell$ denote the connected components (of some parity) in $G_1 - S$. It suffices to show that $\ell \leq k$. WLOG, assume that $|C_1| \geq |C_2| \geq \ldots \geq |C_\ell|$.

First, note that $\alpha(G) \leq \frac{n}{\ln^{1/2}n}$. Otherwise, there exists an independent set $I$ in $G$ with $|I| \geq \frac{n}{\ln^{1/2}n}$, and we can find $A,B \subseteq I, \ A\cap B = \emptyset, \ |A|=|B| \geq \frac{n}{2\ln^{1/2}n}$ such that there is no edge between $A$ and $B$, in contradiction to property \ref{prop: lec6 - 5}.

If $G_1-S$ has $\ell$ connected components, then $\alpha (G_1-S)\geq \ell$ (we choose one representative of each connected component).
Hence, $\ell \leq \frac{n}{\ln^{1/2}n}$, so we can assume that $k \leq \frac{n}{\ln^{1/2}n}$.\\

\textbf{Case 1: $\sum_{i=2}^{\ell}|C_i| \leq \frac{n}{\ln^{1/2}n}$}.\\
Let $V_0= \bigcup_{i=2}^{\ell}C_i$. Notice that $\ell-1 \leq |V_0| \leq \frac{n}{\ln^{1/2}n}$ and $N_{G_1}(V_0)\subseteq S$.
From the expansion property \ref{prop: a}, it follows that: 
\[
k=|S| \geq |N_G(V_0)| \geq |V_0| \cdot \ln^{1/8}n \geq (\ell-1) \cdot \ln^{1/8}n,
\]
and in particular $k \geq \ell$, as desired.

\textbf{Case 2: $|V_0| \geq \frac{n}{\ln^{1/2}n}, \ |C_1| \geq \frac{n}{2\ln^{1/2}n}$}.\\
Under these conditions, $G$ contains an edge between $C_1$ and $V_0$ --- a contradiction. \\

\textbf{Case 3: $|\bigcup_{i=2}^{\ell}C_i| = |V_0| \leq \frac{n}{\ln^{1/2}n}, \ |C_1| \leq \frac{n}{2\ln^{1/2}n}$}. \\
Under these conditions, there exists a collection $I \subseteq [\ell]$ of connected components $\{C_i\}_{i\in I}$, such that:
$$\frac{n}{2\cdot \ln^{1/2}n} \leq |\cup_{i\in I}C_i| \leq \frac{n}{\ln^{1/2}n}$$
(We add components to $I$ one by one until the desired total volume is reached).
Hence, the set $\bigcup_{i \in I}C_i$ expands itself:
$$\left|N_{G_i}\left(\cup_{i\in I}C_i\right)\right| \geq |\cup_{i \in I}C_i| \cdot \ln^{1/8}n = \Theta\left(\frac{n}{\ln^{3/8}n}\right).$$
As before:
$$N_{G_1}(\cup_{i \in I}C_i) \subseteq S,$$
and we assumed that $k=|S| \leq \frac{n}{\ln^{1/2}n}$ --- a contradiction.

Overall, We proved that there exists a perfect matching $M_1$ in $G_1=G[U]=G[V\setminus W]$, therefore, the union $M:= M_0 \cup M_1$ is a perfect matching in the entire graph $G$, completing the proof of Theorem \ref{thm: perfectM - ER}.
\hfill $\square$

\section{A generalization of Cayley's formula for graphs with a given maximum degree}

Similar to Cayley's theorem, which determines the number of spanning trees in the complete graph on $n$ vertices, we would like to estimate the number of $k$-vertex subgraphs of a given graph $G$ that are trees. To that end, we use the following definition:

\begin{defn}
    Given a graph $G$, a vertex $v \in V(G)$, and  $k \in \mathbb{N}$, let $\mathcal{T}(v,k)$ denote the collection of trees in $G$ of size $k$ rooted at $v$. Furthermore, we denote $|\mathcal{T}(v,k)| = t(v,k)$.
\end{defn}

\begin{thm}[Cayley's formula]
    Let $2\le n\in \mathbb{N}$. In $K_n$, $t(v,n)=n^{n-2}$, for every $v\in V(K_n)$.
\end{thm}

\begin{example}
    For $k=1$: $t(v,1)=1$. For $k=2$: $t(v,2)=d_G(v)$.
\end{example}

\begin{thm} [Beveridge, Frieze, McDiarmid \cite{zbMATH01288356}]
    Let $G$ be a graph with maximum degree $\Delta$, and let $v \in V(G)$. Then for all $k \ge 2$:
    $$t(v,k) \leq \frac{k^{k-2} \cdot  \Delta^{k-1}}{(k-1)!} \leq (e\Delta)^{k-1}.$$
\end{thm}

\begin{proof}
    For a tree $T \in \mathcal{T}(v,k)$, consider a labeling $\varphi$ of the vertices of $T$ satisfying:
    $$ \varphi(v)=k,\ \varphi:V(T) \to [k],\ \varphi \text{ is a bijection}.$$ 
    For every $T \in \mathcal{T}(v,k)$ there exist $(k-1)!$ such labelings. Therefore, the total number of pairs $(T, \varphi)$, where $T \in \mathcal{T}(v,k)$ and $\varphi$ is a valid labeling, is:
    $$|\mathcal{T}(v,k)| \cdot (k-1)! = t(v,k)\cdot (k-1)!\ .$$
    Now, given a pair $(T, \varphi)$, we construct a spanning tree $T'$ in $K_k$ as follows:
    $$(i,j) \in E(T') \iff \varphi(x)=i,\ \varphi(y)=j\ \text{for}\ x\neq y\in V(T)\ \text{and}\ (x,y)\in E(T).$$
    
\begin{figure}[H]
    \centering
    \includegraphics[width=0.5\textwidth]{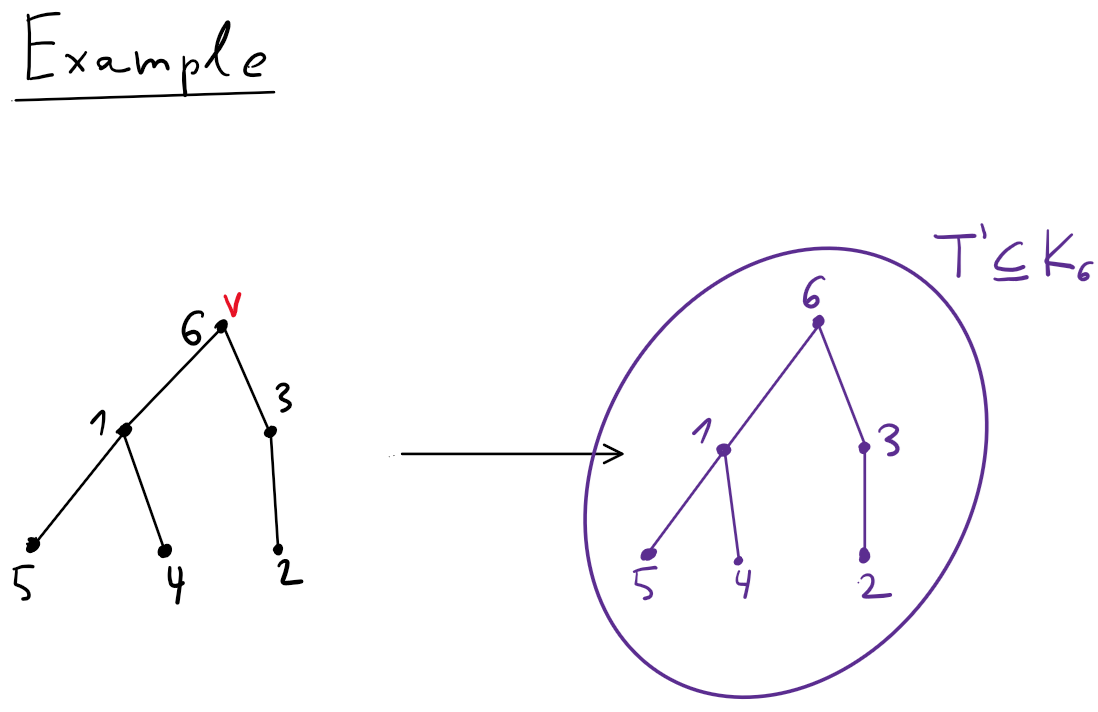}
\end{figure}

    We now bound from above the number of pairs $(T, \varphi)$ that yield a tree $T'$ in $K_k$.
    We run BFS on the tree $T'$ starting from vertex $k$.
    When we reach (in the BFS) a vertex $\ell \in [k-1]$, we have to define $\varphi^{-1}(\ell)$: we find the parent $i \in [k]$ of $\ell$ in $T'$, and choose $\varphi^{-1}(\ell)$ from among the (available) neighbors of $\varphi^{-1}(i)$ --- there are at most $\Delta$ options.
    Hence, there are at most $\Delta^{k-1}$ ways to choose $\varphi^{-1}$. We obtain:
    $$t(v,k) \cdot (k-1)! \leq k^{k-2} \cdot \Delta^{k-1} \implies t(v,k) \leq \frac{k^{k-2} \cdot \Delta^{k-1}}{(k-1)!}.$$
    Finally, we note that for all $k \ge 2$, one can verify that
    $$ \frac{k^{k-2}}{(k-1)!} \leq \frac{(k-1)^{k-1}}{(k-1)!} \leq e^{k-1} \implies t(v,k) \leq (e\Delta)^{k-1}.$$
\end{proof}

\begin{remark}
    We have proved:
    $$\Delta(G) \leq \Delta \implies t(v,k) \leq \frac{k^{k-2}\cdot \Delta^{k-1}}{(k-1)!}.$$
    It is easy to see that in fact the essentially same proof yields the following:
    Let $G$ be a graph with $\delta(G) = \delta \ge k$, then for every $v \in V$ and for every integer $k \ge 2$:
    $$t(v,k) \geq \frac{(\delta -k)^{k-1} \cdot k^{k-2}}{(k-1)!}.$$
    In particular, if $G$ is a $d$-regular graph and $d \ge k$, then:
    $$\frac{(d -k)^{k-1} \cdot k^{k-2}}{(k-1)!} \leq t(v,k) \leq \frac{d^{k-1}\cdot k^{k-2}}{(k-1)!}. $$
\end{remark} 

\section{The Binary (Hyper)cube}

\begin{defn}
    For every integer $d \ge 1$, we define the \textit{binary (hyper)cube} $Q^d$ as follows:  
    \begin{enumerate}
        \item $V(Q^d)=  \left\{0,1\right\}^d$;
        \item $\forall \bar{x}, \bar{y} \in V(Q^d):  (\bar{x},\bar{y})\in E(Q^d) \iff d_H(x, y) = 1$, i.e., $\bar{x}$ and $\bar{y}$ differ in exactly one coordinate.
    \end{enumerate}
\end{defn}

\begin{remark}
    $Q^d$ can be constructed from $Q^{d-1}$ by taking two disjoint copies of $Q^{d-1}$ and connecting the corresponding vertices in both copies by a matching.
\end{remark}

\begin{defn}
    For $\bar{x}\in V(Q^d)$, we define its \textit{weight}: 
    $$|\bar{x}|\coloneqq | \left\{1\leq i \leq d\mid x_i=1\right\} |.$$
\end{defn}

\begin{figure}[H]
    \centering
    \includegraphics[width=0.5\textwidth]{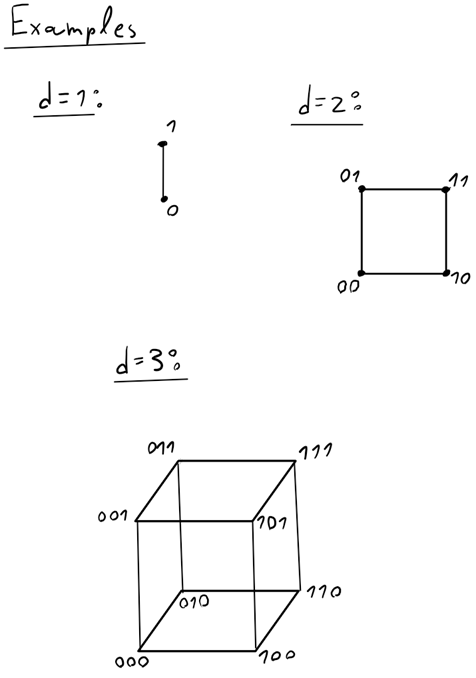}
\end{figure}

\subsection{Basic properties of $Q^d$}
\begin{enumerate}
    \item $n=|V(Q^d)|=2^d$.
    \item $Q^d$ is a $d$-regular graph.
    \item $Q^d$ is a bipartite graph, with sides: 
    $$O=\left\{\bar{x}\in V(Q^d): |\bar{x}|\ \text{is odd}\right\},\quad E=\left\{\bar{x}\in V(Q^d): |\bar{x}|\ \text{is even}\right\}.$$
    \item $Q^d$ is $d$-connected.
    \item In $Q^d$, there exist a perfect matching and a Hamilton cycle.
\end{enumerate}

\subsection{Isoperimetric Inequalities in $Q^d$}

\begin{defn}
    Let $G=(V,E)$ be a graph and let $\emptyset \neq S \subsetneq V$.
    \begin{enumerate}
        \item The \textit{vertex boundary of S} is $N(S)= \left\{v\in V\setminus S\mid v \ \text{has a neighbor in }S\right\}$.
        \item The \textit{edge boundary of S} is $\partial S= \left\{e \in E\mid |e\cap S| =\allowbreak |e \cap \bar{S}| = 1\right\}$.
    \end{enumerate}
\end{defn}

\begin{figure}[H]
    \centering
    \includegraphics[width=0.5\textwidth]{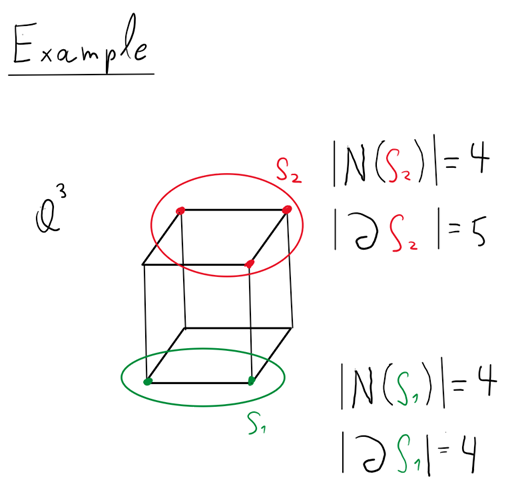}
\end{figure}

The isoperimetric properties of a graph look as follows:\\
\textbf{Vertex Isoperimetry:}
$$\forall S\subseteq V, 0<|S|=s<|V|:\ |N(S)| \geq f(s).$$
\textbf{Edge Isoperimetry:} 
$$S \subseteq V, 0<|S|=s<|V(G))|:\ |\partial S| \geq g(s).$$

\subsubsection{Edge Isoperimetric properties of $Q^d$}

We begin with the following observation:
\begin{obs}
    $Q^d$ is a $d$-regular graph, hence for every $S\subseteq V(Q^d)$:
    $$d\cdot |S| = \sum_{v\in S}d(v)=2\cdot e(S) + |\partial S| \implies |\partial S| = d\cdot |S| - 2\cdot e(S).$$
    That is, it suffices to bound $e(S)$ from above.
\end{obs}

\begin{thm}[Harper \cite{zbMATH03353132}]\label{thm: Harper}
    For every $d \ge 1$ and every $S \subseteq V(Q^d)$:
    $$e_{Q^d}(S) \leq \frac{1}{2}|S|\cdot \log_2|S|.$$
    Thus, $|\partial S| \geq |S|\cdot(d- \log_2|S|)$.
\end{thm}

\begin{remark}
    The above bound is tight for a $k$-dimensional subcube ($k \le d$), which is defined as follows:
    $$S=\left\{\bar{x} \in V(Q^d)\mid x_1=x_2=...=x_{d-k}=0\right\}$$
    (in general, fixing $d-k$ coordinates to arbitrary values).
    
    Indeed, $|S|=2^k$ and every $v\in S$ has exactly $k$ neighbors in $S$. Thus,
    \[
    e(S)=\frac{1}{2}|S|\cdot k=\frac{1}{2}\cdot 2^k\cdot k=\frac{1}{2}|S|\cdot \log_2|S|.
    \]
\end{remark}
\newpage
\phantomsection
\addcontentsline{toc}{chapter}{Lecture 8}
\lecture{8}{June 14, 2026}{Lecture 8}{Itay Markbreit}
\setcounter{chapter}{8}
\setcounter{section}{0}

\section{Edge Isoperimetric Inequalities for the Hypercube - Continued}

We show a (slightly) weaker version of Harper's result.

\begin{thm}\label{thm: weak Harper}
    Suppose that $S \subseteq V(Q^d)$. Then
    \[
    e(S)\le |S|\log_2 |S|.
    \]

\end{thm}
In order to prove this theorem, we show the following two lemmas:

\begin{lemma}
    Let $S\subseteq V(Q^d)$. If the minimum degree of the induced subgraph $Q^d[S]$ is $\delta$, then
    \[
    |S|\ge 2^\delta .
    \]
\end{lemma}

\begin{proof}

Fix a vertex $\bar{v}\in S$. For every $0\le i\le \delta$, define
\[
A_i=\{\bar{u}\in S\mid\text{dist}_{Q^d}(\bar{u},\bar{v})=i\}.
\]
Notice that (for example)
\[
A_0=\{\bar{v}\},\quad A_1=N(\bar{v})\cap S.
\]
Observe that the sets $A_i$ are pairwise disjoint, hence
\[
|S|\ge \sum_{i=0}^{\delta}|A_i|.
\]
For every $i\ge 0$, all vertices in $A_i$ have the same parity, hence $A_i$ is an independent set in $Q^d$. Therefore, the neighbors of each vertex $\bar{u}\in A_i$ lie in $A_{i-1}\cup A_{i+1}$. Moreover, $\bar{u}$ sends at most $i$ edges into $A_{i-1}$, thus, at least $\delta-i$ edges into $A_{i+1}$ (since the degree of $Q^d[S]$ is at least $\delta$). Similarly, every $w\in A_{i+1}$ incident to at most $i+1$ edges from $A_i$. Therefore,
\[
|A_i|(\delta-i)
   \le |E(A_i,A_{i+1})|\le |A_{i+1}|(i+1),
\]
implying
\[
|A_{i+1}|
   \ge
   \frac{\delta-i}{i+1}|A_i|.
\]
Observe that for every $0\le i\le \delta$,
\[
\frac{\binom{\delta}{i+1}}{\binom{\delta}{i}}=\frac{\delta-i}{i+1},
\]
hence, it is easy to verify (by induction on $i$) that for every $0\le i\le \delta$,
\[
|A_i|
   \ge
   \binom{\delta}{i}.
\]
Consequently,
\[
|S|
   \ge
   \sum_{i=0}^{\delta}|A_i|
   \ge
   \sum_{i=0}^{\delta}\binom{\delta}{i}
   =
   2^\delta.\qedhere
\]
\end{proof}

\begin{lemma}
    If $G$ is a graph with average degree $d$, then $G$ contains a subgraph $G'$ whose minimum degree satisfies
    \[
    \delta(G') \ge \frac d2.
    \]
\end{lemma}

\bigskip

\begin{proof}

We construct a sequence
\[
G=G_0,G_1,\ldots
\]
as follows.

As long as $G_i\neq\varnothing$, choose a vertex $v$ whose degree in $G_i$ is smaller
than $\frac d2$ (if exists), and define
\[
G_{i+1}=G_i-\{v\}.
\]

If the process terminates with a non-empty graph $G_t$, then every vertex of
$G_t$ has degree at least $\frac d2$, and therefore we are done, by taking $G'=G_t$.

Thus, it remains to show that the process cannot terminate with the empty graph.

Suppose, for contradiction, that eventually $G_t=\varnothing$.
Then every deleted vertex had degree smaller than $\frac d2$ at the moment it
was removed.

Therefore the total number of removed edges is less than
\[
|V(G)|\cdot \frac d2.
\]

On the other hand, the total number of removed edges is exactly $e(G)$. Hence
\[
e(G)<|V(G)|\frac d2,
\]
which contradicts the fact that the average degree of $G$ is $d$.
\end{proof}

\begin{proof}[Proof of Theorem \ref{thm: weak Harper}]
Let $S\subseteq V(Q^d)$. Since the average degree of the induced subgraph $Q^d[S]$ is
\[
\frac{2e(S)}{|S|},
\]
the previous lemma implies that there exists $S_0\subseteq S$ such that the induced subgraph $G[S_0]$ has minimum degree at least
\[
\frac{e(S)}{|S|}.
\]
By the first lemma, $|S_0|\ge 2^{e(S)/|S|}$.
Since $|S_0|\le |S|$, we obtain $|S| \ge 2^{\,e(S)/|S|}$.
Taking logarithms,
\[
\log_2 |S|
\ge
\frac{e(S)}{|S|},
\]
and therefore
\[
e(S)\le |S|\log_2 |S|.\qedhere
\]
\end{proof}

\begin{remark}
    If $S$ is a $(d-1)$-dimensional subcube, then
    \[
    e(S)
    =
    2^{d-1}\log_2\!\left(2^{d-1}\right)
    =
    (d-1)2^{d-1}.
    \]
    In particular, Theorem \ref{thm: weak Harper} gives a weak/meaningless bound for large sets $S\subseteq V(Q^d)$.
\end{remark}

The following theorem gives another isoperimetric bound that treats large subsets of vertices in $Q^d$.

\begin{thm}\label{isop_large sets_Q^d}
    For every $S\subseteq V(Q^d)$ satisfying $0<|S|\le 2^{d-1}$, we have
    \[
    |\partial S|\ge |S|.
    \]
\end{thm}

\begin{remark}
    The theorem is tight when $S$ is a $(d-1)$-dimensional subcube.
\end{remark}

\begin{proof}

We prove the statement by induction on $d$. The case $d=1$ is trivial. Let $d>1$ and let
\[
S\subseteq V(Q^d),\quad 0<|S|\le 2^{d-1}.
\]
Define
\[
S_0=\{\bar{x}\in S\mid x_1=0\},
\quad
S_1=\{\bar{x}\in S\mid x_1=1\}.
\]
Then
\[
S=S_0\sqcup S_1.
\]
Without loss of generality, assume $|S_0|\le |S_1|$. We further define
\[
V_0=\{\bar{x}\in V\mid x_1=0\},
\quad
V_1=\{\bar{x}\in V\mid x_1=1\}.
\]
Notice that $S_0\subseteq V_0, S_1\subseteq V_1$ and $|V_0|=|V_1|=2^{d-1}$. By the induction hypothesis, since
\[
|S_0|\le \frac12 |S|,
\quad \text{and}\quad
|S_0|\le \frac12 |V_0|,
\]
we have
\[
|E(S_0,V_0\setminus S_0)|\ge |S_0|,
\quad
|E(S_1,V_1\setminus S_1)|
\ge
\min\{|S_1|,\;|V_1\setminus S_1|\}
\ge |S_0|.
\]

Furthermore, $Q^d$ contains a perfect matching between $V_0$ and $V_1$. In this matching, exactly $|S_1|$ edges are incident to vertices of $S_1$, and at least $|S_1|-|S_0|$ of them have their other endpoint in $V_0 \setminus S_0$. Thus, at least $|S_1|-|S_0|$ of such edges belong to $\partial S$.

\begin{figure}[H]
    \centering
    \includegraphics[width=0.5\textwidth]{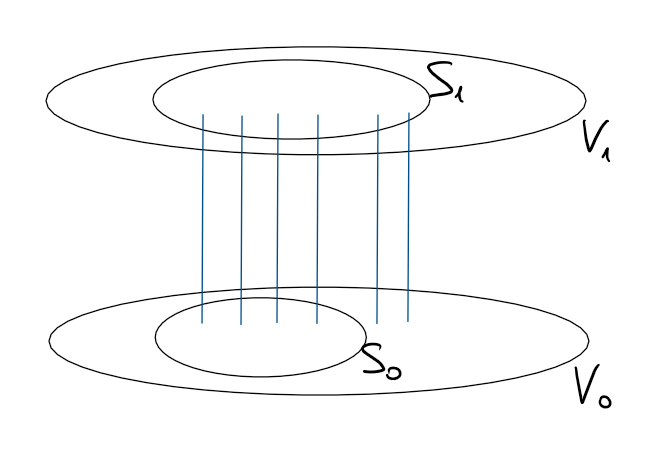}
\end{figure}

To conclude,
\[
|\partial S|
\ge
|E(S_0,V_0\setminus S_0)|
+
|E(S_1,V_1\setminus S_1)|
+
(|S_1|-|S_0|)
\ge
|S_0|+|S_0|+(|S_1|-|S_0|)
=
|S_0|+|S_1|
=
|S|.\qedhere
\]
\end{proof}

\section{Vertex Isoperimetric Inequalities for the Hypercube}

For general interest, we also present a vertex isoperimetric inequality. We begin by exhibiting a subset of vertices in the hypercube which serves as an example attaining equality in the isoperimetric inequality that we will present later.

\begin{example}
    For $0\le k\le d$, define
    
    \[
    B_k=
    \{\bar x\in V(Q^d)\mid |\bar x|\le k\},
    \]
    a ball of radius $k$ around $\bar{0}$. The number of vertices of weight at most $k$ is
    \[
    |B_k|
    =
    \sum_{i=0}^{k}
    \binom{d}{i}.
    \]
    Moreover,    
    \[
    N(B_k)
    =
    \{\bar x\in V(Q^d)\mid |\bar x|=k+1\},
    \]
    and therefore
    \[
    |N(B_k)|
    =
    \binom{d}{k+1}.
    \]
\end{example}

\begin{defn}
    We define the \textit{lexicographic order} on the vertices of $Q^d$ as follows:

    For $\bar x\neq \bar y\in V(Q^d)$,
    
    \[
    \bar x <_{\mathrm{lex}} \bar y
    \iff
    x_{i_0}=0,\ y_{i_0}=1, \quad \text{where} \quad i_0=
    \min\{1\le i\le d\mid x_i\neq y_i\,\}.
    \]
\end{defn}

\begin{defn}
    We define the \textit{simplicial order} on the vertices of $Q^d$ as follows:

    \[
    \bar x <_{\mathrm{sim}} \bar y
    \iff
    |\bar x|<|\bar y|,\quad \text{or} \quad (|\bar x|=|\bar y|
    \quad\text{and}\quad
    \bar x <_{\mathrm{lex}} \bar y).
    \]
\end{defn}

\begin{example}
    For $d=3$, the simplicial order on $V(Q^3)$ is

    \[
    \sigma=
    (000,001,010,100,011,101,110,111).
    \]
\end{example}

\begin{thm}[Harper \cite{zbMATH03254683}]
    Let $A \subseteq V(Q^d)$.
    Let $B$ be the set of the first $|A|$ vertices in the simplicial order on $V(Q^d)$. Then
    \[
    |N(A)| \ge |N(B)|.
    \]
    In particular, if    
    \[
    |A|=\sum_{i=0}^{k} \binom{d}{i},
    \quad
    0 \le k < d,
    \]
    then
    \[
    |N(A)| \ge \binom{d}{k+1}.
    \]
\end{thm}

\section{The Model $G_p$}

\begin{defn}
    Let $G=(V,E)$ be a finite graph, and let $0 \le p \le 1$. We define a probability distribution on the subgraphs $H \subseteq G$ as follows: for every edge $e\in E(G)$,
    \[
    \mathbb{P}(e\in E(H)) = p.
    \]
    That is, each edge is chosen independently with probability $p$.

    Equivalently, for every $H\subseteq G$, we can define:     
    \[
    \mathbb{P}(G_p = H)
    =
    p^{|E(H)|}
    (1-p)^{\,|E(G)|-|E(H)|}.
    \]
\end{defn}

\begin{remark}
    The value of $p$ may depend on the parameters of the graph $G$.
\end{remark}

\begin{remark}
    If $G=K_n$ (the complete graph on $n$ vertices), then $G_p$ is exactly $G(n,p)$.
\end{remark}

\subsection{Typical Questions about $G_p$}

\begin{enumerate}
\item What is the typical size of the connected components of $G_p$? For which values of $p$ is the existence of a giant component guaranteed whp?

\item What are the typical combinatorial properties of $G_p$? For example, perfect matching, connectivity, Hamilton cycle, etc.

\item What is the critical probability for which a fixed graph $H$
appears in $G_p$ whp?
\end{enumerate}

We focus on the following question:
\begin{question}
    We consider a $d$-regular graph $G$, where $d$ can be a constant satisfying $d\ge 3$ or $d\to \infty$. What is the critical probability $p$ of the existence of the giant component in $G_p$?
\end{question}

\underline{Guess/Explanation}: Choose a vertex $v$ of $G$ and expose the connected component of $G_p$ containing $v$, denoted by $C_v$. To that end, we perform a search process from $v$ (BFS or DFS), while revealing the edges of $G_p$. If we get (during the process) to a vertex $u$ in $C_v$ (due to an edge entering $u$), we expect to have $d-1$ edges leaving $u$ and going outside the currently explored connected component $C_v$. Intuitively, by the GW process, if the expected number of edges leaving $u$ is $(d-1)p\eqqcolon c<1$, then one may expect that no large component will emerge from $v$. On the other hand, if $(d-1)p = c > 1$, then the situation changes dramatically, and one expects that $C_v$ will be large, with probability that bounded away from 0.

We therefore might guess that the critical probability for the emergence of a giant component in $G_p$ is
\[
p^*=\frac{1}{d-1}.
\]

\begin{remark}
    Notice that for large $d$ we have
    \[
    \frac{1}{d-1}
    \approx
    \frac{1}{d},
    \]
    whereas for constant values of $d$, there is a substential difference between $\frac{1}{d-1}$ and $\frac{1}{d}$.
\end{remark}

\subsection{Sub-Critical Percolation}

\begin{thm}
    Let $G$ be a $d$-regular graph on $n$ vertices, and let $\epsilon>0$ be a fixed constant. Set $p=\frac{1-\epsilon}{d-1}$. Then whp (as $n\to\infty$), every connected component
    $L_i$ of $G_p$ satisfies
    \[
    |L_i|
    \le
    \frac{9\ln n}{\epsilon^2}.
    \]
\end{thm}

\begin{proof}
    We run a search process on $G$ (BFS or DFS) and ``feed" it with random bits. Let $m = |E(G)|$ and notice that $m=\frac{nd}{2}\le n^2$. Let $\bar{X}=(X_i)_{i=1}^{m}$ be a sequence of independent random variables satisfying
    \[
    \mathbb{P}(X_i=1)=p,
    \quad
    \mathbb{P}(X_i=0)=1-p.
    \]
    
    When the algorithm asks whether the $i$-th edge (the order is determined by the execution of the algorithm) of $G$ falls into $G_p$, if $X_i=1$, then $e_i\in E(G_p)$, and if $X_i=0$, then $e_i\in E(G_p)$. At the end of the algorithm, we obtain a random subgraph $G_p$, together with its connected components.
    
    Fix $k=k(n,\epsilon)$, whose value will be chosen later. Suppose there exists a connected component $K$ with $|K|>k$. Let us consider the precise moment during the execution of the algorithm at which the $(k+1)$-st vertex of $K$ was discovered. Since, in each connected component, the first vertex comes "for free", by that moment we have obtained exactly $k$ positive answers. In addition, note that
    \begin{enumerate}
        \item we queried only edges incident to the first $k$ vertices (according to the order in which the algorithm discovered them) of $K$. 
        
        \item $K$ contains a tree that spans the first $k$ vertices in $K$.
    \end{enumerate}

    Hence, we have queried at most $kd-(k-1)$ edges (where $k-1$ is a lower bound on the number of edges inside these $k$ vertices).

    Therefore, $\bar{X}$ must contain a segment of length $kd-(k-1)$, with at least $k$ positive answers, i.e. there are $k$ indices in this segment for which $X_i=1$.

    \begin{figure}[H]
        \centering
        \includegraphics[width=0.5\textwidth]{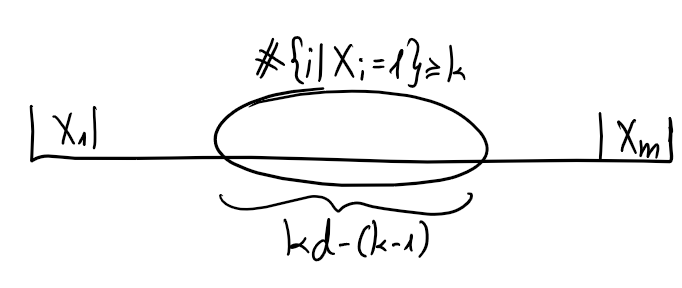}
    \end{figure}

    Since $p=\frac{1-\epsilon}{d-1}$, the number of $X_i=1$ in this segment is distributed as
    \[
    Y \sim Bin\bigl(kd-(k-1),\,p\bigr).
    \]
    Hence
    \[
    \mathbb E[Y]
    =
    \bigl(kd-k+1\bigr)\cdot \frac{1-\epsilon}{d-1}
    =
    (k(d-1)+1)\cdot \frac{1-\epsilon}{d-1}
    =
    (1-\epsilon)k+\frac{1-\epsilon}{d-1}.
    \]
    By Chernoff's inequality, we have
    \[
    \mathbb{P}(Y \ge k)
    \le
    e^{-\frac{\epsilon^{2}k}{4}}.
    \]
    Now, the probability that there exists an interval of length
    $kd-(k-1)$ with at least $k$ successes in $[m]$ is at most
    
    \[
    m\cdot e^{-\frac{\epsilon^{2}k}{4}}
    <
    n^{2}\cdot e^{-\frac{\epsilon^{2}k}{4}}.
    \]
    By choosing
    \[
    k=
    \left\lfloor
    \frac{9\ln n}{\epsilon^{2}}
    \right\rfloor,
    \]
    this probability is $o(1)$.
    
    Therefore, whp, $G_p$ has no component of size $>k$, namely for all $i\ge 1$,
    
    \[
    |L_i|
    \le
    \frac{9\ln n}{\epsilon^{2}}.
    \]
\end{proof}

\begin{remark}
    A few comments are in place.
    \begin{enumerate}
        \item The theorem remains true even if the assumption that $G$ is $d$-regular is replaced by the assumption that $\Delta(G)\le d$, with the same proof.
        \item Applying the theorem to the case where $G=K_n$, we obtain that in $G(n,p)$, for $p=\frac{1-\epsilon}{n}$, whp $|L_i|\le \frac{9\ln n}{\epsilon^2}$, for every $i\ge 1$.
        \item In fact, in $G(n,p)$ with $p=\frac{1-\epsilon}{n}$ ($\epsilon>0$ is a small constant), whp $|L_1|=\Theta\left(\frac{\ln n}{\epsilon^2}\right)$.
    \end{enumerate}
\end{remark}

Later, we look at $Q^d_p$ with $p=\frac{c}{d}$, where $c>1$ is a constant, and show that whp, there exists a giant component $L_1$ in $Q^d_p$ such that $|L_1|=(1+o(1))yn$ where $y\in (0,1)$ satisfies $y=1-e^{-cy}$. Furthermore, for every $i\ge 2$, we have $|L_i|=O(d)$.
\newpage
\phantomsection
\addcontentsline{toc}{chapter}{Lecture 9}
\lecture{9}{June 21, 2026}{Lecture 9}{Gali Maman}
\setcounter{chapter}{9}
\setcounter{section}{0}

\section{The Super-Critical Case --- $Q^d_p$}

\begin{thm}[Ajtai, Komlós, Szemerédi \cite{zbMATH03769676}; Bollob\'{a}s, Kohayakawa, \L uczak \cite{zbMATH00035138}]\label{giant_in_Q^d_p}
Assume $c > 1$ is a constant, and consider $Q^d_p$ with $p = \frac{c}{d}$. Then whp:
\begin{enumerate}
    \item $|L_1| = (1+o(1))yn$, where $n \coloneqq 2^d$ and $y\coloneqq y(c)$ is the (unique) solution of
    \begin{equation}\label{eq: y}
        \tag{$\ast$} 1-y = e^{-cy}
    \end{equation}
    in $(0,1)$.
    \item Every other component $L_i$, $i\ge 2$, satisfies
    \[
    |L_i| \le \frac{d}{c-1-\ln c}.
    \]
\end{enumerate}
\end{thm}

\begin{remark}~
\begin{enumerate}
    \item The order of $L_1$ in $Q^d_p$ is asymptotically equal to the order of the giant component of $G(n,p)$ for $p = \frac{c}{n}$.
    \item If $c = 1+\epsilon$ for small $\epsilon > 0$, then it is easy to see that
    \[
    c - 1 - \ln c>0\;\text{ and }\; c - 1 - \ln c = \Theta(\epsilon^2),
    \]
    and therefore whp
    \[
    |L_i| \le O\left(\frac{d}{\epsilon^2}\right) = O\left(\frac{\ln n}{\epsilon^2}\right)
    \]
    (this can be shown to be tight up in $n$ and in $\epsilon$).
\end{enumerate}
\end{remark}

\begin{proof}

The proof is based on a paper by Michael Krivelevich \cite{arXiv:2311.07210} and consists of a sequence of lemmas. We will not prove the following two lemmas here; proofs may be found in a variety of standard sources.

\begin{lemma}\label{lec8: lem1}
Consider a Galton--Watson (GW) process with offspring distribution $\mathrm{Bin}\!\left(d, \frac{c}{d}\right)$. When $c>1$, as $d \to \infty$ the process survives (i.e. does not become extinct at any finite stage) with probability $(1+o(1))y$, where $y \in (0,1)$ is defined as in \eqref{eq: y}.
\end{lemma}

\begin{lemma}[Measure concentration / edge-exposure martingale / McDiarmid's inequality]\label{lec8: lem2}
Let $G=(V,E)$ be a finite graph with $|E|=m$ edges, and let $f: 2^G \to \mathbb{R}$ be a function on the subgraphs of $G$ (i.e.\ $f(G')$ is a function of the subgraph $G' \subseteq G$). Suppose $f$ is Lipschitz with parameter $C$: for every $G', G'' \subseteq G$ differing in exactly one edge,
\[
|E(G') \triangle E(G'')| = 1 \quad \Longrightarrow \quad |f(G') - f(G'')| \le C.
\]
Let $X = f(G')$ for $G' \sim G_p$ (the random subgraph of $G$ obtained by including each edge independently with probability $p$) with $p\in [0,1]$. Then for all $\beta > 0$,
\[
\mathbb{P}\big[\,|X - \mathbb{E}[X]| \ge \beta\,\big] \le 2 \exp\left\{ -\frac{\beta^2}{C^2 m} \right\}.
\]
\end{lemma}

The following two lemmas were proved in previous lectures.

\begin{lemma}[Counting trees in $Q^d$ given a root]\label{lec8: lem3}
Let $v \in V(Q^d)$ and let $k\ge 1$ be an integer. The number of trees on $k$ vertices in $Q^d$ that contain $v$ is at most
\[
(ed)^{k-1}.
\]
\end{lemma}

\begin{lemma}[Harper's isoperimetric inequality, weak version]~\label{lec8: lem4}
\begin{enumerate}
    \item For every $S \subseteq V(Q^d)$,
    \[
    e(S) \le |S| \cdot \log_2 |S|,
    \]
    and consequently
    \[
    |\partial S| \ge |S| \,(d - 2 \log_2 |S|).
    \]
    \item For every $S \subseteq V(Q^d)$ with $|S| \le \frac{n}{2}$,
    \[
    |\partial S| \ge |S|.
    \]
\end{enumerate}
\end{lemma}

From now on, assume $p = \frac{c}{d}$ for a constant $c>1$, and we study $Q^d_p$. We also fix an integer $t>0$ (to be chosen later). Lastly, we consider $y$ as defined in \eqref{eq: y}.

\begin{lemma}\label{lec8: lem5}
    In $Q^d_p$, whp there is no connected component of size in the range $\left[\frac{d}{c-1-\ln c}, \, d^t\right]$.
\end{lemma}

\begin{proof}
Let $v \in V(Q^d)$ and denote by $C_v$ the connected component of $v$ in $Q^d_p$. By the union bound, it suffices to show
\[
\mathbb{P}\left(|C_v| \in \left[\frac{d}{c-1-\ln c}, \, d^t\right]\right) = o\!\left(\frac{1}{n}\right).
\]
Suppose $|C_v| = k$ for some $k \in \left[\frac{d}{c-1-\ln c}, \, d^t\right]$. Then:
\begin{enumerate}
    \item There is a tree $T \subseteq Q^d_p$ containing $v$, with $|V(T)| = k$.
    \item There is no edge of $Q^d_p$ between $V(T)$ and $V(Q^d) \setminus V(T)$.
\end{enumerate}
Therefore,
\[
\mathbb{P}(|C_v| = k) \le \underbrace{(ed)^{k-1}}_{\text{choice of tree } T} \cdot \underbrace{p^{k-1}}_{\text{pay for } E(T)} \cdot \underbrace{(1-p)^{k(d - 2\log_2 k)}}_{\text{no edges between } V(T) \text{ and } V(Q^d)\setminus V(T)}.
\]

Summing over $k$,
\begin{align*}
\mathbb{P}\left(|C_v| \in \left[\frac{d}{c-1-\ln c}, \, d^t\right]\right)
&\le \sum_{k= d/(c-1-\ln c)}^{d^t} (ed)^{k-1} p^{k-1} \cdot (1-p)^{k(d - 2\log_2 k)} \\
&\le \sum_{k= d/(c-1-\ln c)}^{d^t} (ec)^k \cdot e^{-\frac{c}{d}\, k(d - 2\log_2 k)} \\
&= \sum_{k= d/(c-1-\ln c)}^{d^t} \left[ ec \cdot e^{-c + \frac{2c\log_2 k}{d}} \right]^k \\
&= \sum_{k= d/(c-1-\ln c)}^{d^t} \left[ e^{1 + \ln c - c + o(1)} \right]^k \\
&= (1+o(1))\, e^{(1 + \ln c - c + o(1))\, \frac{d}{c-1-\ln c}} \\
&= (1+o(1))\, e^{-d+o(d)} = o\!\left(\frac{1}{n}\right) \qquad (\text{since } n = 2^d).
\end{align*}
\end{proof}

\begin{lemma}\label{lec8: lem6}
For every $v \in V(Q^d)$, we have
\[
\mathbb{P}\left(|C_v| \ge d^t\right) = (1+o(1))y.
\]
\end{lemma}

\begin{proof}
First, we estimate $\mathbb{P}\left(|C_v| \ge d^{1/2}\right)$. We run a search algorithm (BFS or DFS) on $Q^d_p$ to grow $C_v$. As long as $|C_v| \le d^{1/2}$, for every vertex $u \in C_v$ whose neighbors we are about to expose, at most $d^{1/2}$ of its edges have already been exposed. Hence we can lower-bound (couple) the growth of $C_v$ by a GW process with offspring distribution $\mathrm{Bin}(d - d^{1/2}, p)$, and therefore $|C_v| \ge d^{1/2}$ with probability $y(c')$, which equals to $(1-o(1))y$, since
\[
c' = (d-d^{1/2})\,p = (1-o(1))c.
\]

Next, we estimate $\mathbb{P}\left(d^{1/2} \le |C_v| \le d^t\right)$. This event is handled very similarly to before:
\[
\mathbb{P}\left(d^{1/2} \le |C_v| \le d^t\right) \le \sum_{k=d^{1/2}}^{d^t} (ed)^{k-1} p^{k-1} (1-p)^{k(d - 2\log_2 k)}
\le \sum_{k=d^{1/2}}^{d^t} \left(ec\cdot  e^{-c + \frac{2c\log_2 k}{d}}\right)^{k} = o(1),
\]
because $ec\cdot e^{-c} < 1$ for all $c > 1$. Therefore
\[
\mathbb{P}\left(|C_v| \ge d^t\right) = 1 - \mathbb{P}\left(|C_v| \le d^{1/2}\right) - \mathbb{P}\left(|C_v| \in \left(d^{1/2}, d^t\right)\right) = y + o(1) = (1+o(1))y.
\]
\end{proof}

\begin{lemma}\label{lec8: lem7}
Denote
\[
W = \{v \in V(Q^d) \mid |C_v| \ge d^t\}
\]
(the collection of ``largish'' components).
Then, whp
\[
|W| = (1+o(1))yn.
\]
\end{lemma}

\begin{proof}
Notice that we can write
\[
|W| = \sum_{v \in V(Q^d)} \mathbbm{1}_v, \qquad
\mathbbm{1}_v =
\begin{cases}
1 & |C_v| \ge d^t, \\
0 & \text{otherwise.}
\end{cases}
\]
By linearity of expectation,
\begin{align*}
\mathbb{E}[|W|] &= \sum_{v \in V(Q^d)} \mathbb{E}[\mathbbm{1}_v] = \sum_{v \in V(Q^d)} \mathbb{P}\left(|C_v| \ge d^t\right) \\
&= (1+o(1))yn \qquad (\text{by Lemma \ref{lec8: lem6}}).
\end{align*}
We now apply the concentration result (Lemma \ref{lec8: lem2}). For $G' \subseteq Q^d$, define $f(G') = |W|$, and set a random variable $X = f(G')$.
Then:
\begin{itemize}
    \item[(1)] $\mathbb{E}[X] = (1+o(1))yn$.
    \item[(2)] Toggling the state of a single edge of $E(Q^d)$ (present in $Q^d_p$ or not) changes the value of $X$ by at most $2d^t$ (this is maximal precisely when the toggled edge is the unique connection between two components, each of size $< d^t$, that would merge into a component of size $\ge d^t$).
\end{itemize}

\begin{center}
\includegraphics[width=0.55\textwidth]{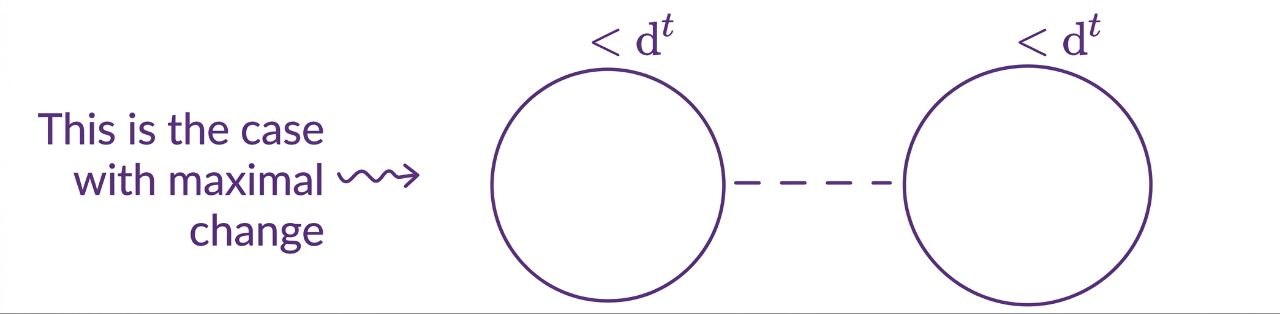}
\end{center}
By McDiarmid's inequality (Lemma \ref{lec8: lem2}),
\[
\mathbb{P}\big[|X - \mathbb{E}[X]| \ge n^{2/3}\big] \le 2\exp\left\{-\frac{n^{4/3}}{\frac{nd}{2}\cdot (2d^t)^2}\right\} = o(1).
\]
Hence, whp
\[
X = (1+o(1))\,\mathbb{E}[X] = (1+o(1))yn.
\]
\end{proof}

Informally, in the next lemma we show that $W$ is spread ``nicely'' throughout $V(Q^d)$.

\begin{lemma}\label{lec8: lem8}
Let $W$ be as in Lemma \ref{lec8: lem7}. Whp, in $Q^d_p$, every vertex $v \in V(Q^d)$ is at distance at most two from $W$.
\end{lemma}

\begin{proof}
By symmetry, it suffices to show that
\[
\mathbb{P}\Big[\,\bar{0} \text{ is not within distance } 2 \text{ of } W\,\Big] = o\!\left(\frac{1}{n}\right).
\]
Define $\epsilon = \frac{c-1}{c} > 0$ and $I = \left[\left\lfloor \frac{\epsilon d}{2} \right\rfloor\right]$. For every $i \ne j \in I$, define
\[
H_{ij} = \big\{ \bar{x} \in Q^d \mid x_i = 1,\ x_j = 1,\ \forall k \in I \setminus \{i,j\},\ x_k = 0 \big\},
\]
and let $u_{ij} \in H_{ij}$ be the point with $(u_{ij})_i = (u_{ij})_j = 1$ and $(u_{ij})_k = 0$ for all $k \ne i,j$.

\bigskip
\begin{center}
\begin{tikzpicture}[scale=1]
\draw[-] (0,0) -- (8,0);
\draw[thick] (0,0) -- (3,0);
\draw (0,0) node[below]{$1$} -- (8,0) node[below]{$d$};
\draw (2.5,0.1) -- (2.5,-0.1);
\draw (0,0.1) -- (0,-0.1);
\node[above] at (1.3,0.1) {$I$, length $\lfloor \epsilon d/2 \rfloor$};
\filldraw (0.6,0) circle (1.5pt) node[below]{$i$};
\filldraw (1.3,0) circle (1.5pt) node[below]{$j$};
\end{tikzpicture}
\end{center}
\bigskip

Then:
\begin{enumerate}
    \item $H_{ij}$ is a subcube of $Q^d$ of dimension $d - \left\lfloor \frac{\epsilon d}{2} \right\rfloor$.
    \item $u_{ij} \in H_{ij}$, and $u_{ij}$ is at distance $2$ from $\bar{0}$.
    \item For $(i,j) \ne (i',j')$, observe that $V(H_{ij}) \cap V(H_{i'j'}) = \emptyset$.
\end{enumerate}
Since percolation restricted to each $H_{ij}$ is distributed as $Q^{d-\lfloor \epsilon d/2\rfloor}_p$, we note
\[
\left(d - \left\lfloor \frac{\epsilon d}{2} \right\rfloor\right) p \ge \left(d - \frac{\epsilon d}{2}\right)\frac{c}{d} = \left(1 - \frac{\epsilon}{2}\right) c = \left(1 - \frac{c-1}{2c}\right) c = \frac{c+1}{2} > 1.
\]
That is, percolation on $H_{ij}$ is supercritical. Hence, for every $i \ne j \in I$,
\[
\mathbb{P}\big[u_{ij} \text{ belongs to a component of size } \ge d^t \text{ of } H_{ij}\big] \ge \delta
\]
for some $\delta \coloneqq \delta(c) > 0$. Therefore,
\begin{align*}
    \mathbb{P}(\text{for all } i \ne j \in I,\ u_{ij} &\text{ does not belong to a component of size } \ge d^t \text{ of } H_{ij})\\
    &\le (1-\delta)^{\binom{|I|}{2}} \le e^{-\delta \binom{|I|}{2}}= e^{-\Theta(d^2)} = o\!\left(\frac{1}{n}\right),
\end{align*}
as required.

Therefore, with probability $1 - o(1/n)$, there exist $i \ne j \in I$ such that the component of $u_{ij}$ in $H_{ij}$ has at least $d^t$ vertices. Since the component of $u_{ij}$ in $Q^d_p$ contains the one in $H_{ij}$, it follows that whp, every vertex $v \in V(Q^d)$ is at distance at most $2$ from $W$.
\end{proof}

Let
\[
    t = 31,
\]
and define $p_1, p_2$ by
\[
p_2 = \frac{1}{d^5}, \qquad 1-p = (1-p_1)(1-p_2),
\]
so in particular $p_2 \ge p - p_1$. Let
\[
G_1 \sim Q^d_{p_1}, \qquad G_2 \sim Q^d_{p_2}
\]
be independent, and set $G = G_1 \cup G_2$. Notice that $G \sim Q^d_p$. Denote
\[
W_1 = \{v \in V(Q^d) : |C_{G_1,v}| \ge d^t\}.
\]

\begin{lemma}\label{lec8: lem9}
    Whp, all components of $W_1$ merge into a single component after adding the edges of $G_2$.
\end{lemma}

\begin{proof}
By Lemma \ref{lec8: lem8} (applied to $G_1$), every vertex $v \in V(Q^d)$ is at distance $\le 2$ from $W_1$. Suppose toward contradiction that exposing the edges of $G_2$ does not unite the components of $W_1$. Then, there is a partition
\[
W_1 = A \cup B, \qquad A, B \ne \emptyset,
\]
respecting the connected components of $W_1$, such that there is no path in $G_2$ between $A$ and $B$.

Suppose $W_1$ has $s \le n$ components, and (WLOG) one of $A, B$ has $\ell \le s/2$ of them. Define
\[
A' := \{v \in V(Q^d) \setminus B : \mathrm{dist}_{Q^d}(v, A) \le 2\}, \qquad B' := V(Q^d) \setminus A'.
\]

\begin{center}
\includegraphics[width=0.55\textwidth]{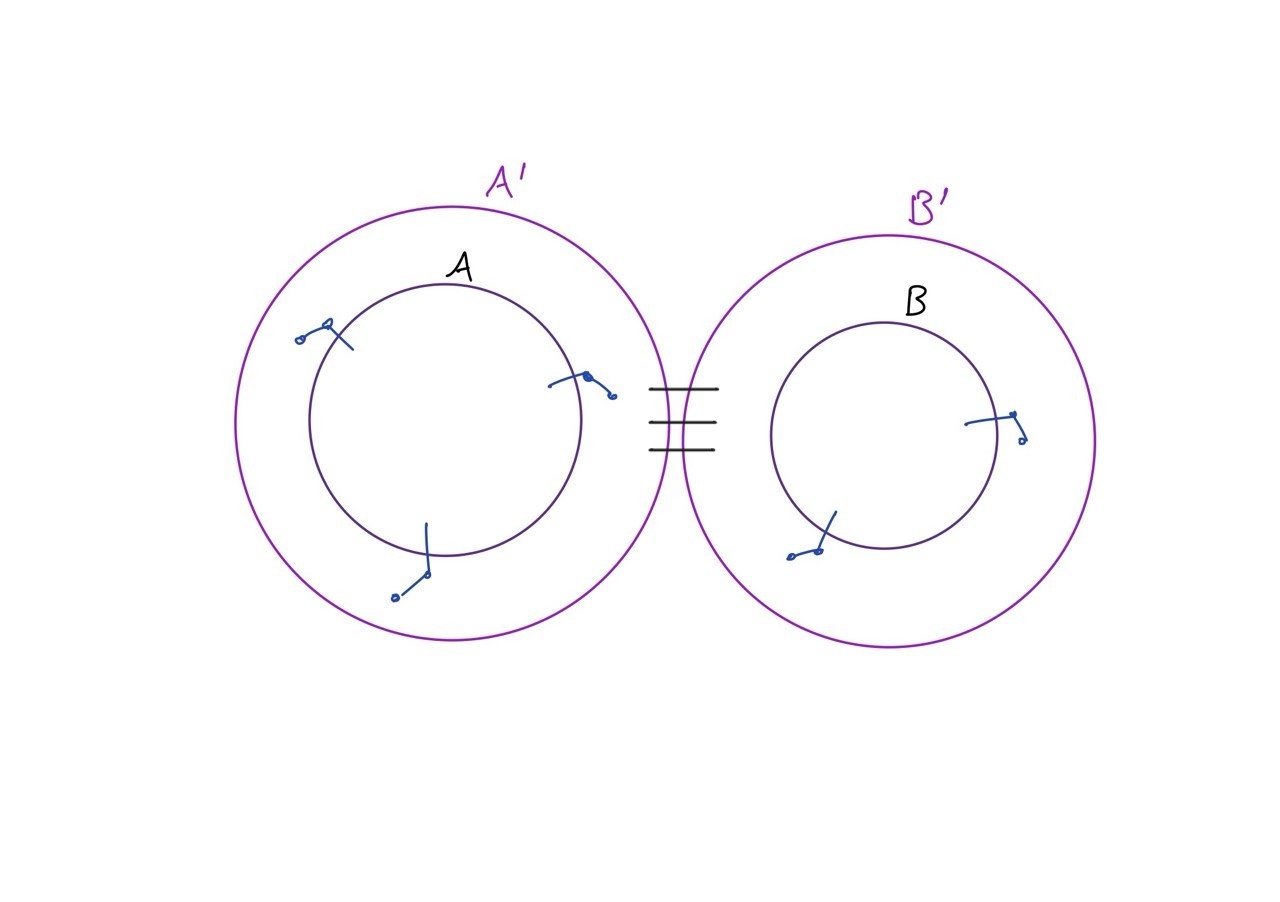}
\end{center}
Then $A \subseteq A'$, $B \subseteq B'$, and $A' \cup B' = V(Q^d)$. By the isoperimetric inequality in $Q^d$, there are at least $\min\{|A'|,|B'|\}$ edges between $A'$ and $B'$. Since
\[
|A'| \ge |A| \ge \ell\, d^{31}, \qquad |B'| \ge |B| \ge \ell\, d^{31},
\]
there are at least $\ell\, d^{31}$ edges of $Q^d$ between $A'$ and $B'$. By construction of $A, A', B, B'$, each such edge lies on a path of length $\le 5$ between $A$ and $B$ (at most two vertices inside each of $A$ and $B$, plus the crossing edge).

Hence there are in total at least $\ell\, d^{31}$ (not necessarily disjoint) paths of length $\le 5$ between $A$ and $B$ in $Q^d$. Since every edge of $Q^d$ lies on at most $5\,d^4$ such paths, a greedy argument yields at least
\[
\frac{\ell\, d^{31}}{5 \cdot 5\, d^4 + 1} \ge \frac{\ell\, d^{27}}{30}
\]
pairwise edge-disjoint paths of length $\le 5$ between $A$ and $B$. Each such path is entirely contained in $G_2$ with probability $\ge p_2^5$, independently across the paths. Hence
\[
\mathbb{P}\big[A, B \text{ not connected in } G_2\big] \le \left(1-p_2^5\right)^{\ell d^{27}/30} \le e^{-p_2^5 \cdot \ell d^{27}/30} = e^{-\Theta(\ell d^2)}.
\]
Therefore
\[
\mathbb{P}\big[\exists \text{ components of } W_1 \text{ disconnected in } G_2\big] \le \sum_{\ell \le s/2} \binom{s}{\ell} e^{-\Theta(\ell d^2)} \le \sum_{\ell \le s/2} n^{\ell}\, e^{-\Theta(\ell d^2)} = o(1),
\]
using $s \le n$ and $n=2^d$.
\end{proof}

\paragraph{Proof of the theorem.}
Set $p_1, p_2, G_1, G_2, W_1$ as above. Since $p = (1+o(1))p_1$, Lemma \ref{lec8: lem7} gives that whp
\[
|W_1| = (1+o(1))yn.
\]
Moreover, every connected component of $G_1$ outside of $W_1$ has size $\le \frac{d}{c-1-\ln c}$. By Lemma \ref{lec8: lem9}, all components of $W_1$ merge into a single component whp, and therefore
\[
|L_1| \ge |W_1| = (1+o(1))yn.
\]
On the other hand, recalling
\[
W = \{v \in V(Q^d) : |C_{G,v}| \ge d^{t}\},
\]
Lemma \ref{lec8: lem7} gives $|L_1| \le |W|$, and whp
\[
|W| = (1+o(1))yn,
\]
so that
\[
|L_1| \le (1+o(1))yn.
\]
Combining the two bounds,
\[
|L_1| = (1+o(1))yn.
\]

\begin{center}
\includegraphics[width=0.7\textwidth]{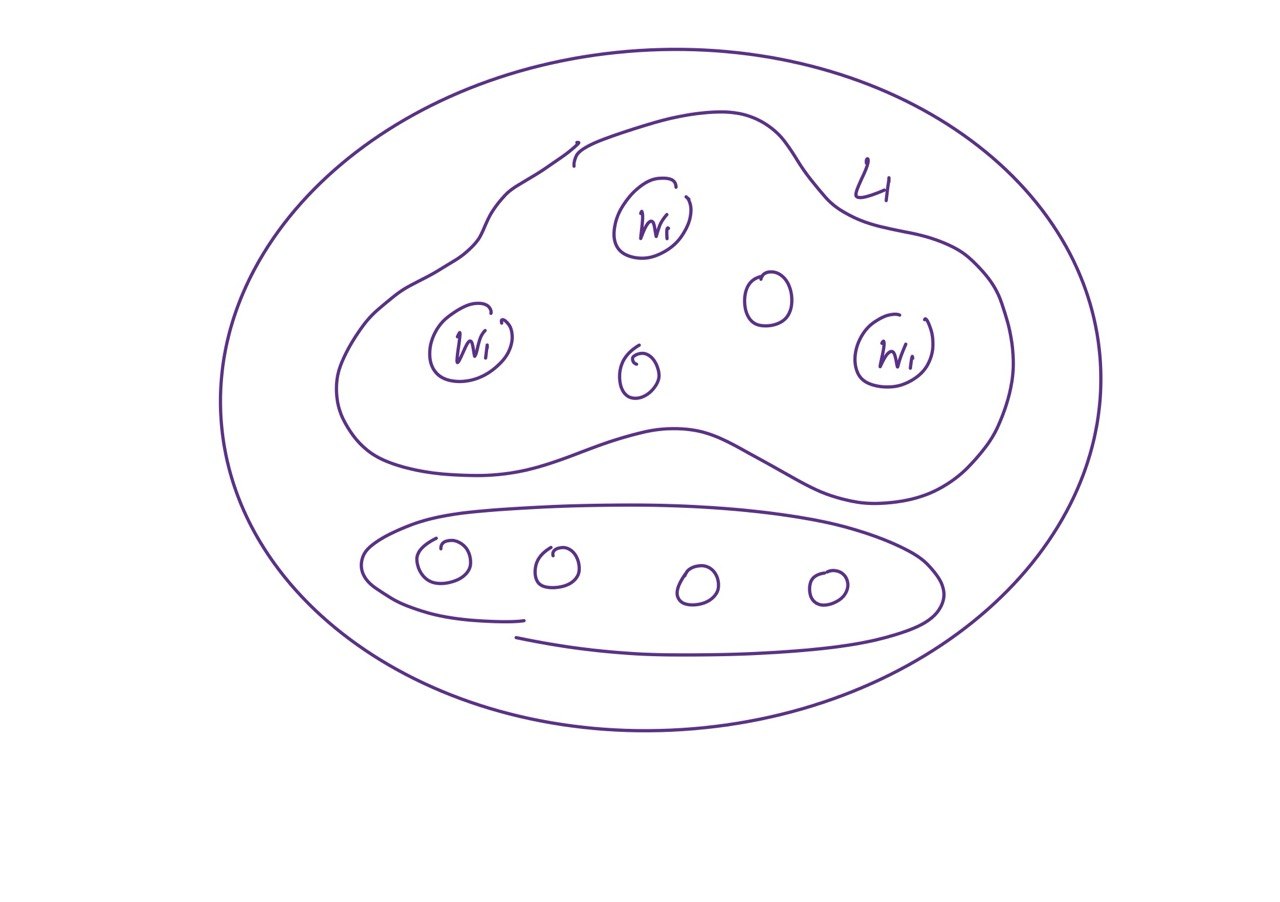}
\end{center}
It remains to show that all components of $G = Q^d_p$ outside $W_1$ have size whp at most $\frac{d}{c-1-\ln c}$ in $G$.

Define an auxiliary random graph $\Gamma$ whose vertices are the connected components of $G_1$ outside $W_1$, where two such components are joined by an edge in $\Gamma$ iff there is a $G_2$-edge between them. Each vertex of $\Gamma$ corresponds to a component of size $O(d)$ with maximum degree $d$, hence connected to $O(d^2)$ components; between any two components there are at most $d^2$ possible connecting edges. Therefore, any two vertices of $\Gamma$ are joined by an edge with probability
\[
O(d^2)\cdot p_2 = O(d^{-3}).
\]
Thus in $\Gamma$ the degrees are $O(d^2)$ and the edge probability is $O(d^{-3})$ —-- so we are in the sub-critical regime. By the sub-critical analysis, whp every connected component of $\Gamma$ has size
\[
O(\log_2 n) = O(d).
\]
Expanding each vertex of $\Gamma$ back into its underlying component of $G_1$ (size $O(d)$), a component of $G$ outside $W_1$ has size whp
\[
O(d) \cdot O(d) = O(d^2).
\]
But by Lemma \ref{lec8: lem5}, $G = Q^d_p$ has no component of size in $\left[\frac{d}{c-1-\ln c}, d^2\right]$, so this bound forces the actual size to fall below the lower end of that interval. Hence, for all $i\ge 2$, we have
\[
|L_i| \le \frac{d}{c-1-\ln c},
\]
which completes the proof.
\end{proof}

\newcommand{\lecturername}{Dr. Sahar Diskin} 

\renewcommand{\lecture}[4]{\heading{#1}{Percolation on Finite Graphs, Tel Aviv Univ., Spring 2026}{#2}{\lecturername}{Scribe: #4}{#3}}

\newpage
\phantomsection
\addcontentsline{toc}{chapter}{Lecture 10}
\lecture{10}{June 28, 2026}{Lecture 10}{Aner Mash}
\setcounter{chapter}{10}
\setcounter{section}{0}

\section{A comparison between $G(n,p)$ and $Q^d_p$}

Consider the random graph $G(n,p)$:
\begin{figure}[H]
    \centering
    \includegraphics[width=0.5\textwidth]{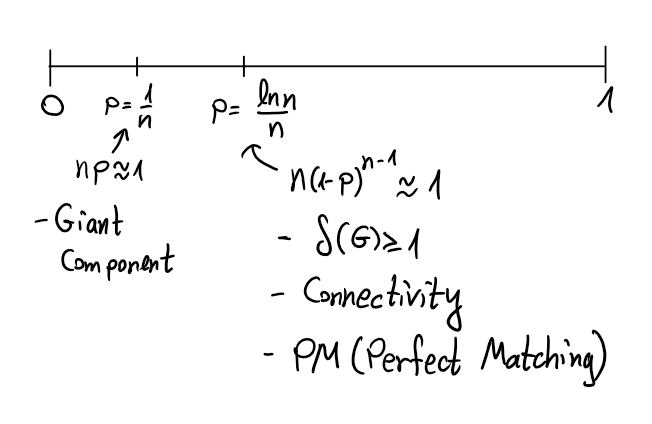}
\end{figure}
\noindent Consider the random hypercube $Q_p^d$:
\begin{figure}[H]
    \centering
    \includegraphics[width=0.5\textwidth]{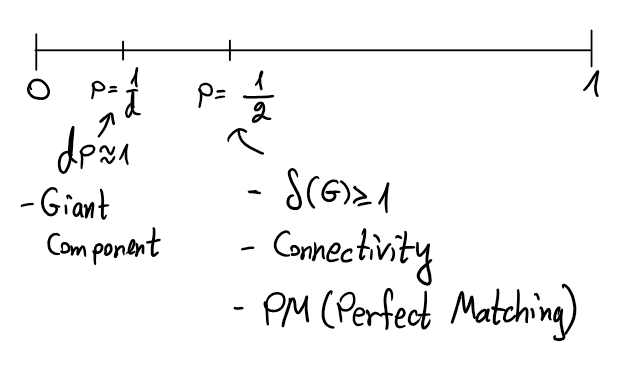}
\end{figure}
Intuitively, $p=\frac{1}{2}$ comes from the calculation of the first moment:
\[
\mathbb{E}[\# \text{ isolated vertices}] = 2^d \cdot (1-p)^d=1\iff p=\frac{1}{2}.
\]

\section{Perfect Matching in $Q^d_p$}

\begin{thm}[Bollob\'as \cite{zbMATH04162929}]\label{thm: PM}
    \[
    \lim_{d\to\infty} \mathbb{P}(Q_p^d \text{ contains a perfect matching}) = \begin{cases} 0 & p \le \frac{1}{2} - \frac{\omega(d)}{d} \\ 1 & p \ge \frac{1}{2} + \frac{\omega(d)}{d} \end{cases}
    \]
\end{thm}
\noindent The proof is due to Sahar Diskin and Anna Geisler \cite{arXiv:2404.14020}.

\subsection{Preparation}

\begin{lemma}\label{finger_print}
    Let $\frac{2^d}{d^{\ln^2 d}}\le m\le 2^d$ be an integer. Then:
    \[
    \# \{ S \subseteq V(Q^d) : |S|=m, \, |\partial S| < m \cdot \ln^4 d \} \leq \exp\left\{\frac{2m}{\ln^2 d}\right\}.
    \]
\end{lemma}

\begin{remark}
    For such $m$,
    \[
    \#\{\text{subsets of size } m\} = \binom{2^d}{m} \ge \left(\frac{2^d}{m}\right)^m.
    \]
    This lemma shows that relatively few subsets of size $m$ have weak expansion in $Q^d$. 
\end{remark}

\begin{proof}
    Fix $\frac{2^d}{d^{\ln^2 d}}\le m\le 2^d$. Define the family:
    \[
    \mathcal{F} \coloneqq \{ S \subseteq V(Q^d) \mid |S|=m, \, |\partial S| < m \cdot \ln^4 d \}.
    \]
    
    Given $i \in [d]$ and $S \subseteq V$, let $E_i(S, S^c)$ denote the set of edges in $E(S, S^c) = \partial S$ that are oriented along the $i$-th dimension.
    \begin{figure}[H]
        \centering
        \includegraphics[width=0.5\textwidth]{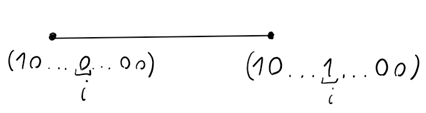}
    \end{figure}
    
    Denote $e_i(S, S^c) \coloneqq |E_i(S, S^c)|$. Given a subset  $I \subseteq [d]$, we denote:
    \[
    e_I(S, S^c) \coloneqq \sum_{i \in I} e_i(S, S^c).
    \]
    
    We say that $S$ is a \textit{``bad'' set w.r.t. $I$} if
    \[
    e_I(S, S^c) < m \cdot \ln^4 d \cdot \frac{|I|}{d}.
    \]
    
    Averaging for every $1\le k\le d$, if $S \in \mathcal{F}$, then there exists a subset $I$ with $|I|=k$, such that $S$ is ``bad'' with respect to $I$.
    Let us denote this family by $\mathcal{F}_I$.
    
    We wish to bound $|\mathcal{F}|$. Observe that 
    \[
        \mathcal{F}\subseteq \bigcup_{\substack{I\subseteq [d] \\ |I|=k}} \mathcal{F}_I,
    \]
    hence
    \[
    \forall k \quad |\mathcal{F}| \le \binom{d}{k} \cdot \max_{\substack{I \subseteq [d] \\ |I|=k}} |\mathcal{F}_I|.
    \]
    Set
    \[
        k = \log_2(\ln^5d).
    \]
    Our goal is to estimate $|{\cal F}_I|$ from above, for $I\subseteq [d]$ with $|I|=k$.
    Let $\mathcal{Q}_I$ be the set of all subcubes of $Q^d$ with all fixed coordinates outside $I$.
    Observe that every $Q\in \mathcal{Q}_I$ is a $k$-dimensional subcube of $Q^d$, and in particular $|V(Q)| = 2^k$. Moreover, notice that $|\mathcal{Q}_I|=\frac{2^d}{2^k}$.

    \begin{claim} Given a $k$-dimensional cube $Q$ and a subset $A\subseteq V(Q)$ such that $A \neq \emptyset$ and $A \neq V(Q)$, we have $|\partial A| \ge k$. (Exercise - use Harper) \end{claim}
    
    Let $S \in \mathcal{F}_I$. The number of distinct subcubes from ${\cal Q}_I$ such that the intersection with $S$ is neither $\emptyset$ nor $S$ is at most $\frac{|S| \cdot \ln^4 d}{d}$. Otherwise, the number of edges in the directions of $I$ would be at least $k \cdot \frac{|S| \cdot \ln^4 d}{d}$, which implies $S \notin \mathcal{F}_I$ --- a contradiction.
    
    Therefore, $S$ contains at least $\frac{|S|}{2^k} - \frac{|S| \cdot \ln^4 d}{d}$ subcubes from ${\cal Q}_I$ and at most $\frac{|S| \cdot \ln^4 d}{d} \cdot 2^k$ additional vertices. We obtain:
    \[
    |\mathcal{F}_I| \le \binom{\frac{2^d}{2^k}}{\frac{m}{2^k}} \cdot \binom{2^d}{\frac{m \cdot \ln^4 d}{d}\cdot 2^k}\le  \binom{\frac{2^d}{\ln^5 d}}{\frac{m}{\ln^5 d}} \cdot \binom{2^d}{\frac{m \cdot \ln^9 d}{d}} \stackrel{\text{exercise}}{\le} \exp\left(\frac{m}{\ln^2 d}\right),
    \]
    where we used $m \ge \frac{2^d}{d^{\ln^2 d}}$. Thus,
    \[
    |\mathcal{F}| \le \binom{d}{\log_2(\ln^5 d)} \cdot \exp\left(\frac{m}{\ln^2 d}\right) < \exp\left(\frac{2m}{\ln^2 d}\right).
    \]
\end{proof}

Lastly, recall the well-known Hall's condition for having a perfect matching in a bipartite graph:
\begin{claim}\label{hall}
    Let $G=(A\cup B,E)$ be a bipartite graph such that $|A|=|B|$. The graph $G$ contains a perfect matching if and only if for every $X\subseteq A$, we have $|N_G(X)|\ge |X|$.  
\end{claim}

It easily follows from Hall’s theorem that if a bipartite graph $G=(A\cup B,E)$ has no perfect matching, then there are sets $A_0\subseteq A$ and $B_0\subseteq B$ such that $G[A_0\cup B_0]$ is connected and
\begin{enumerate}
    \item $|A_0|=|B_0|+1$ and $N_G(A_0)=B_0$ or
    \item $|B_0|=|A_0|+1$ and $N_G(B_0)=A_0$.
\end{enumerate}
In both cases $|A_0\cup B_0|$ is odd.

\subsection{Proof of Theorem \ref{thm: PM}}
    
Let $p \ge \frac{1}{2} + \frac{\omega(d)}{d}$. Whp, $Q^d_p$ contains no isolated vertices (the expected number of isolated vertices in $Q^d_p$ tends to 0, hence the assertion is implied by Markov's inequality). Denote by $\mathcal{O}$ and $\mathcal{E}$ the sides of the hypercube (which is a bipartite graph). For every $W\subseteq V(Q^d)$, denote by $W_\mathcal{O}=W\cap \mathcal{O}$ and $W_\mathcal{E}=W\cap \mathcal{E}$.

\begin{figure}[H]
    \centering
    \includegraphics[width=0.5\textwidth]{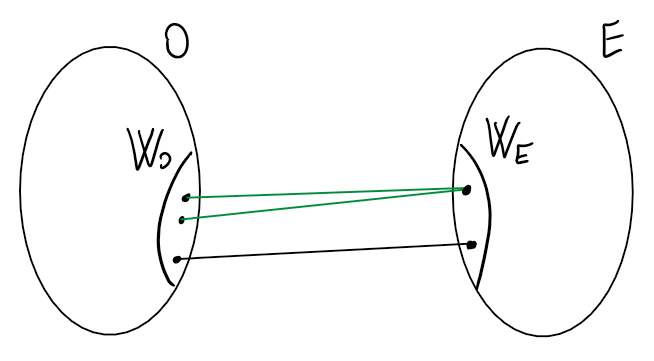}
    \caption*{Hall's condition fails because of the green edges}
\end{figure}

Let $m\in \mathbb{N}_{odd}$. A set $W\subseteq V(Q^d)$ is called an \textit{$m$-obstacle} if $|W|=m$, $W$ is connected in $Q^d$, $|W_\mathcal{O}|< |W_\mathcal{E}|$ and $N_{Q^d_p}(W_\mathcal{O}) = W_\mathcal{E}$. One can define an analogous notion of an $m$-obstacle by exchanging the roles of $\mathcal{O}$ and $\mathcal{E}$. Hence, by the remark after Claim \ref{hall}, it suffices to show that whp, there are no $m$-obstacles on either side for any odd $1\le m\le 2^{d-1}$. We show it for our definition of an $m$-obstacle, where the proof for the analog definition is very similar. First observe that $m=1$ is an easy to check case. Hence, assume that $m\ge 3$. We split the proof into cases:
\medskip

\textbf{Case 1:} $3 \le m \le 2^{\frac{d}{100}}$. \\
By Theorem \ref{thm: Harper}, we have $e(W) \le \frac{m \cdot \log_2 m}{2}$. Also, $|W_\mathcal{O}| \ge \frac{m}{2}$.
Therefore,
\[
e_{Q^d}(W_\mathcal{O}, V \setminus W_\mathcal{E}) \ge \frac{m}{2}\cdot d-e(W)\ge \frac{m}{2}\cdot (d-\log_2 m)\ge \frac{49}{100}md.
\]

Now,
\begin{align*}
\mathbb{E}[\# m\text{-obstacles}] &\le 2^d \cdot (ed)^m \cdot \left(1-p\right)^{\frac{49}{100}md}\le 2^d \cdot (ed)^m \cdot \left(\frac{1}{2}\right)^{\frac{49}{100}md} \\
&\le 2^{-\frac{49}{100}md + d + 2m \cdot \log_2 d} \le 2^{-\frac{45}{100}md + d}.
\end{align*}

Since $m\ge 3$, we have 
\[
    \sum_{m=3}^{2^{\frac{d}{100}}} \mathbb{E}[\# m\text{-obstacles}] \to 0,
\]
and finish by Markov's inequality.
\medskip

\textbf{Case 2:} $2^{\frac{d}{100}} \le m \le \frac{2^d}{d^{\ln^2 d}}$. \\
\[
\mathbb{E}[\# m\text{-obstacles}] \le 2^d \cdot (ed)^m \cdot \left(\frac{1}{2}\right)^{\frac{m}{2} \cdot (d - \log_2 m)} \le 2^{d + 2m\cdot \log_2 d  - \frac{m \cdot d}{2} + \frac{m \cdot \log_2 m}{2}} \le 2^{-\frac{m\ln^3 d}{3}}.
\]

Again, summing over all values of $m$, we finish by Markov's inequality.
\medskip

\textbf{Case 3:} $\frac{2^d}{d^{\ln^2 d}}\le m\le 2^{d-1}$. \\
Let $W$ be an $m$-obstacle. We split this case into two subcases according to the size of $\partial W$.
\begin{enumerate}
    \item [(1)] Suppose $|\partial W| \ge m \cdot \ln^4 d$.\\
        Note that $dm=d|W|=2e(W)+|\partial W|$. Hence, $e(W)\le \frac{m}{2}(d-\ln^4 d)$. We have $|W_\mathcal{O}|\ge \frac{m}{2}$ and thus $e(W_\mathcal{O}, V \setminus W_\mathcal{E}) \ge \frac{m}{2}\cdot d-e(W)\ge  \frac{m\ln^4 d}{2}$. Therefore, the probability that such $m$-obstacle $W$ exists in $Q^d_p$ is at most:
        \begin{align*}
        \binom{2^d}{m} \cdot \left(\frac{1}{2}\right)^{\frac{m\ln^4 d}{2}} &\le \exp\left(m\left(2 \ln⁡\left(\frac{2^d}{m}\right)-\frac{\ln^4 d}{2}\right)\right)\le \exp\left(m\left(2\ln^3 d-\frac{\ln^4 d}{2}\right)\right)\\
        &\le \exp\left(-\frac{m\ln^4 d}{3}\right).
        \end{align*}
        Summing over $\le 2^d$ values of $m$ and using Markov's inequality completes this case.
        \item [(2)] Suppose $|\partial W| < m \cdot \ln^4 d$.\\
        In this case, by Lemma \ref{finger_print} we have that the probability we have such an $m$-obstacle in $Q^d_p$ is at most
        \[
        \exp\left( \frac{2m}{\ln^2 d}\right) \cdot \left(\frac{1}{2}\right)^{\frac{m}{2} \cdot (d - \log_2 m)}\le \exp\left( \frac{2m}{\ln^2 d}\right) \cdot \left(\frac{1}{2}\right)^{\frac{m}{2}}.
        \]
        Lastly, as before, we finish by Markov's inequality.
\end{enumerate}

\renewcommand{\lecturername}{Prof. Michael Krivelevich} 

\newpage
\phantomsection
\addcontentsline{toc}{chapter}{Lecture 11}
\lecture{11}{July 5, 2026}{Lecture 11}{Nati Pupko}
\setcounter{chapter}{11}
\setcounter{section}{0}

\section{Preliminaries}


\begin{defn}
    Given a graph $G(V,E)$, for any two vertices $v,u\in V(G)$ we define the \textit{distance} between $u$ and $v$ to be the (edge) length of a shortest $u,v$ path in $G$, and denote it by $\text{dist}_G(u,v)$. In case no path exists between $u$ and $v$, we define $\text{dist}_G(u,v)=\infty$.
\end{defn}

\begin{defn}
    The \textit{diameter} of a graph $G(V,E)$ is defined to be the largest distance between a pair of vertices in $G$. We denote it by $\text{diam}(G)$. Note that $G$ is connected iff $\text{diam}(G)<\infty$.
\end{defn}

\begin{defn}
    Given two vertices in the hypercube, $u,v \in Q^d$, we define the \textit{Hamming distance} between $u$ and $v$ to be the number of coordinates in which $u$ and $v$ differ. Formally, define
    \[
    \mathcal{D}(u,v) = \{ i \in [d] \mid u_i \neq v_i \},
    \]
    and the Hamming distance is defined as
    \[
    d(u,v) = |\mathcal{D}(u,v)|.
    \]
\end{defn}

We now prove that for the hypercube, these two notions of distance are equivalent.
\begin{claim}
    Given two vertices in the hypercube, $u,v \in Q^d$, we have:
    $$d(u,v)=\text{dist}_{Q^d}(u,v).$$
\end{claim}

\begin{proof}
    In order to establish equality, we first prove that $\text{dist}_{Q^d}(u,v)\ge d(u,v)$ by proving that any $u-v$ path in $Q^d$ must include at least $d(u,v)$ edges. After proving this, we show that $\text{dist}_{Q^d}(u,v)\le d(u,v)$ by constructing an explicit $u,v$ path with $d(u,v)$ edges, proving equality.

    Indeed, every edge of $Q^d$ connects between two vertices of Hamming distance 1, therefore, a path of length $k$ can change at most $k$ coordinates. In particular, in order to connect two vertices with $d(u,v)$ different coordinates we need at least $d(u,v)$ edges.

    For the other direction, we can define $P=(v=v_0,...,v_{d(u,v)}=u)$ by defining $v_j$ to be the vector obtained by altering the first $j$ coordinates of $v$ to match $u$.
\end{proof}

\begin{corollary}
    $\mathrm{diam}(Q^d)=d$.
\end{corollary}

\section{Diameter of the giant in $Q^d_p$}

We now turn to analyzing the typical diameter of the giant component in $Q^d_p$, where $p=\frac{1+\epsilon}{d}$ and $\epsilon>0$ is a small constant. We can imagine theoretical cases where the diameter is asymptotically greater than $d$ (where many edges are deleted, and the shortest path between two vertices $u,v$ is highly convoluted). However, we wish to understand the typical behavior.

The question of what is the right order of $\text{diam}(Q^d_p)$ has been open for some time, until recently.

In 2023, Erde, Kang and Krivelevich \cite{zbMATH07628800} proved that whp $\text{diam}(Q^d_p)=O(d^3)$. This was the first result to establish that the correct order is polynomial in $d$. Later, in 2026, Anastos, Diskin, Lichev and Zhukovskii \cite{arXiv:2510.13348} established that whp $\text{diam}(Q^d_p)=O(d)$. This is also the best possible bound, in the sense that whp $\text{diam}(Q^p_d)=\Theta(d),$ (but not exactly $d$).

We prove a weaker statement, that the diameter is typically polynomial in $d$.
\begin{thm}\label{diam_Q^d_p}
Let $\epsilon>0$ be a sufficiently small constant. Let $L_1$ be the largest component of $Q^d_p$, where $p=\frac{1+\epsilon}{d}$. Then, whp 
\[
    \text{diam}(L_1)=O(d^{12}).
\]
\end{thm}

\subsection{Proof of Theorem \ref{diam_Q^d_p}}
Throughout the proof we use the following notation: $n=2^d$, $p=\frac{1+\epsilon}{d}$, where $\epsilon>0$ is a sufficiently small constant. We also denote $G\sim Q^d_p$, and let $L_1$ be the unique giant component of size $|L_1|=\Theta_\epsilon(n)$ (which existence is asserted by Theorem \ref{giant_in_Q^d_p}).

The proof relies on a series of lemmas:

\begin{lemma}\label{lem:lemma 1}
    For any $\epsilon>0$ small enough, there exist $C_1\coloneqq C_1(\epsilon)>0$ large enough, and $\alpha\coloneqq \alpha(\epsilon)>0$ small enough, such that the following holds whp. Let $S\subseteq V(G)$ such that $G[S]$ connected, and of size $C_1d\le |S|\le n^\alpha$. Then $|N_G(S)|\ge \alpha|S|$.
\end{lemma}

\begin{proof}
    Let $C_1 d \le k \le n^{\alpha}$ (where $C_1$ and $\alpha$ are constants that will be chosen later). We define
    \[
    \mathcal{A}_k = \{ S \subseteq V \mid |S| = k \text{ and } G[S] \text{ is connected}, \, |N(S)| \le \alpha k \}.
    \]
    It is enough to show that $\mathbb{P}(\mathcal{A}_k) = o(\frac{1}{n})$, then apply the union bound to obtain:

    $$\mathbb{P}\bigg(\bigcup_{k=C_1d}^{n^\alpha}\mathcal{A}_k\bigg)\le \sum_{k=C_1d}^{n^\alpha}\mathbb{P}(\mathcal{A}_k)\le \sum_{k=C_1d}^{n^\alpha} o\bigg(\frac{1}{n}\bigg)=o(1).$$

    We now proceed to proving $\mathbb{P}(\mathcal{A}_k)=o(\frac{1}{n})$.
    If $\mathcal{A}_k$ occurs (i.e. the set is not empty), then there exists a tree $T$ in $Q^d$ on $k$ vertices such that:
    \begin{enumerate}[label=(\arabic*)]
        \item All edges of $T$ are present in $G$.\label{first}
        \item Denote $S = V(T)$. The $S$ has at most $\alpha k$ neighbors in $G$ outside of $S$.\label{second}
    \end{enumerate}

    Given a tree $T$, conditions \ref{first} and \ref{second} address disjoint sets of edges, and therefore are independent events. We will bound the probability of each of them separately, than multiply, and apply a union bound over all possible trees $T$.

    \underline{Bounding the probability of \ref{first}:} A tree of size $k$ has $k-1$ edges, therefore the probability of it being present in $G$ is $p^{k-1}$.

    \underline{Bounding the probability of \ref{second}:} 
    Condition \ref{second} implies that in the graph $G[S, V\setminus S]$, the maximal size of a matching is $i$, for some $0\le i\le \alpha k$. Let $M$ be a maximal such matching. Then:
    \begin{enumerate}
        \item All $i$ edges of $M$ in $G$.
        \item All edges of $Q^d$ between $S$ and $V\setminus S$ (i.e., $\partial S$) that are not incident to vertices of $M$ are not present in $G$ (as otherwise $M$ could be made larger).
    \end{enumerate}
    Hence,
    \[
        \mathbb{P}(\ref{second})\le \sum_{i=0}^{\alpha k} \binom{|\partial S|}{i}p^i(1-p)^{|\partial S|-2id}\le \sum_{i=0}^{\alpha k} \binom{kd}{i}p^i(1-p)^{k(d-2\log_2 k)-2id},
    \]
    where the last inequality holds by (weak) Harper's inequality. We note that $k\le n^\alpha$, and that for $\alpha$ small enough we have:
    $$k(d-2\log_2k)\ge (1-\epsilon^3)dk.$$
    Thus, overall, by the union bound:
    \begin{align*}
    \mathbb{P}(\mathcal{A}_k) &\le n \cdot (ed)^{k-1} p^{k-1} \sum_{i=0}^{\alpha k} \binom{kd}{i} p^i (1-p)^{(1-\epsilon^3)kd - 2id} \\
    &\le n \left(e(1+\epsilon)\right)^k (1-p)^{(1-\epsilon^3)kd} \sum_{i=0}^{\alpha k} \binom{kd}{i} p^i (1-p)^{-2id}.
    \end{align*}
    Note that:
    \[
    \sum_{i=0}^{\alpha k} \binom{kd}{i} p^i (1-p)^{-2id} \le \sum_{i=0}^{\alpha k} \left(\frac{ekd}{i}\right)^i \left(\frac{1+\epsilon}{d}\right)^i \cdot 10^i
    = \sum_{i=0}^{\alpha k} \left( \frac{ek(1+\epsilon)}{i} \cdot 10 \right)^i.
    \]
    The terms in this sum grow exponentially, and the maximum is achieved at the boundary $i = \alpha k$:
    \[
    \sum_{i=0}^{\alpha k} \left( \frac{ek(1+\epsilon)}{i} \cdot 10 \right)^i \le O(k) \cdot \left( \frac{30}{\alpha} \right)^{\alpha k} \le e^{\epsilon^3 k},
    \]
    for sufficiently small $\alpha\coloneqq \alpha(\epsilon)>0$, where we used the fact that $\lim_{\alpha \to 0^+} \left( \frac{30}{\alpha} \right)^{\alpha} = 1$.

    Plugging this estimate back into $\mathbb{P}(\mathcal{A}_k)$, we obtain:
    \begin{align*}
    \mathbb{P}(\mathcal{A}_k) &\le n \left( e(1+\epsilon) \right)^k (1-p)^{(1-\epsilon^3)kd} e^{\epsilon^3 k} \\
    &\le n \left( e(1+\epsilon) e^{-(1+\epsilon)(1-\epsilon^3)d \cdot \frac{1}{d}} e^{\epsilon^3} \right)^k \\
    &\le n \left( e(1+\epsilon) e^{-1-\epsilon+3\epsilon^3} \right)^k = n \left( (1+\epsilon) e^{-\epsilon + 3\epsilon^3} \right)^k.
    \end{align*}
    Using $1+\epsilon \le e^{\epsilon - \frac{\epsilon^2}{3}}$ for small $\epsilon > 0$, we obtain:
    \[
    \mathbb{P}(\mathcal{A}_k) \le n \left( e^{-\frac{\epsilon^2}{3} + 3\epsilon^3} \right)^k \le n e^{-\frac{\epsilon^2}{4} k}.
    \]
    Choosing $C_1\coloneqq C_1(\epsilon)$ sufficiently large in terms of $\epsilon$, we obtain:
    \[
    \mathbb{P}(\mathcal{A}_k) = o\left(\frac{1}{n}\right).
    \]
    This completes the proof of Lemma~\ref{lem:lemma 1}.
\end{proof}

We will apply Lemma \ref{lem:lemma 1} on balls centered at $v$, i.e., $B(v,k)=\{u \in V: \text{dist}_G(u,v)\le k\}$. This is clearly a connected set, therefore, by Lemma \ref{lem:lemma 1} if we have $C_1d\le |B(v,k)|\le n^\alpha$, then observing that $B(v,k+1)=B(v,k)\cup N_G\left(B(v,k)\right)$, we obtain:
$$|B(v,k+1)|\ge (1+\alpha)|B(v,k)|.$$

In order to ensure that $|B(v,k)| \ge C_1d$, we let $k_0=C_1d$, thus, $|B(v,k_0)|\ge k_0=C_1 d$. Now, by induction, for every $i\ge 0$,

\begin{align}\label{growing balls}
    |B(v,k_0+i)|\ge \text{min}\big\{ (1+\alpha)^i|B(v,k_0)|,\ n^\alpha \big\}.
\end{align}

We use this observation for the next lemma:

\begin{lemma}\label{lem: partition}
    For every sufficiently small constant $\epsilon>0$, there exist $\alpha=\alpha(\epsilon)>0$ and $C_2=C_2(\epsilon)>0$ such that whp there exists a partition $V(L_1)=W_1\cup ... \cup W_t$ satisfying:
    \begin{enumerate}
        \item $W_i \cap W_j= \emptyset$ for all $1\le i\neq j\le t$;
        \item $n^\alpha\le |W_i|\le 2n^\alpha$;
        \item for all $u,v \in W_j$, $\text{dist}_{L_1}(u,v)\le 4C_2d$.
    \end{enumerate}
\end{lemma}

\begin{proof}
By \eqref{growing balls}, we conclude that there exists a constant $C_2=C_2(\epsilon)>0$, such that for every $v\in V(L_1)$ we have $|B(v,C_2d)|\ge n^\alpha$.
Choose $X \subseteq V(L_1)$ to be a maximal set of vertices such that for any $u \neq v \in X$:
\[
B_{L_1}(u, C_2 d) \cap B_{L_1}(v, C_2 d) = \emptyset.
\]
Write $X = \{v_1, \dots, v_{t_0}\}$, for some $t_0\in \mathbb{N}$. By maximality, for every $u \in V(L_1)$, there exists some $v_i \in X$ such that:
\[
B(u, C_2 d) \cap B(v_i, C_2 d) \neq \emptyset \implies \text{dist}_{L_1}(u, v_i) \le 2 C_2 d.
\]
We define a family of disjoint sets $U_1, \dots, U_{t_0} \subseteq V(L_1)$ as follows:
\begin{enumerate}
    \item Initialize $U_i = B_{L_1}(v_i, C_2 d)$.
    \item For any vertex $u \notin \bigcup_{i=1}^t U_i$, we assign $u$ to exactly one $U_i$ for which $\text{dist}(u,v_i) \le 2 C_2 d$.
\end{enumerate}
This yields a partition of $V(L_1)$ to sets $U_1, \dots, U_{t_0}$ satisfying:
\begin{itemize}
    \item $U_i \supseteq B(v_i, C_2 d) \implies |U_i| \ge n^{\alpha}$.
    \item For any $u \in U_i$, we have $\text{dist}_{L_1}(u, v_i) \le 2 C_2 d$.
\end{itemize}
\begin{remark}
The distance between any $u,v \in U_i$ is at most $4 C_2 d$ inside $L_1$.
\end{remark}

Finally, we divide each $U_i$ into smaller disjoint connected subsets $W_j\subseteq U_i$ of sizes $n^{\alpha} \le |W_j| \le 2n^{\alpha}$. This gives a partition:
\[
V(L_1) = W_1 \cup \dots \cup W_t,
\]
where for each $j \in [t]$:
\begin{enumerate}
    \item $W_i \cap W_j = \emptyset$ for $i \neq j$;
    \item $n^{\alpha} \le |W_i| \le 2n^{\alpha}$;
    \item For all $u,v \in W_j$, $\text{dist}_{L_1}(u,v) \le 4 C_2 d$.
\end{enumerate}
\end{proof}

Our strategy now becomes proving that the sets from Lemma \ref{lem: partition} are close to one another, which would allow us to construct a short path between any pair of vertices of $G$ in the following way: If $u\in W_i$ and $v\in W_j$, if we prove that $W_i$ and $W_j$ are close, we can take representatives $u'\in W_i$ and $v'\in W_j$ so that $u'$ and $v'$ are close, which allows us to construct a short $u-v$ path in the following fashion:
$$u\to u'\to v'\to v,$$
where Lemma \ref{lem: partition} tells us that $u$ and $u'$, as well as $v$ and $v'$ are at distance at most $4C_2 d$ from one another.

In order to show something of this sort, we apply sprinkling.
Given our $p = \frac{1+\epsilon}{d}$, we choose $p_2 = \frac{\delta}{d}$, for some $\delta\coloneqq \delta(\epsilon)$ small enough, and $p_1$ such that $1-p = (1-p_1)(1-p_2)$, which gives $p_1 \ge \frac{1+\epsilon-\delta}{d}\ge \frac{1+\epsilon/2}{d}$.

Let $G_1 \sim Q^d_{p_1}$ and $G_2 \sim Q^d_{p_2}$, so $G = G_1 \cup G_2 \sim Q^d_p$. Let $L_1'\subseteq L_1$ be the giant component in $G_1$ (which exists, whp, by Theorem \ref{giant_in_Q^d_p}).

\begin{lemma}\label{lem: paths}
    Whp, any partition of $L_1'=A\cup B$ respecting the sets $W_1,..,W_t$ as in Lemma \ref{lem: partition} (applied in $G_1$) so that $A$ contains $\ell$ of these sets, $1\le \ell\le \frac{3t}{4}$, in $G_2$ there are $\Omega\big(\frac{\ell\cdot n^\alpha}{d^9}\big)$ edge-disjoint paths of length at most 5 between $A$ and $B$. Therefore, whp there exist at least $\Omega\big(\frac{\ell\cdot n^\alpha}{d^{10}}\big)$ different endpoints of paths of length at least 5 between $A$ and $B$.
\end{lemma}
\begin{proof}
    Let:
    \[
    A' = \{ v \in V(Q^d)\setminus B \mid \text{dist}(v, A) \le 2 \}, \quad B' = V(Q^d) \setminus A'.
    \]
    Notice that $A\subseteq A'$ and $B\subseteq B'$, hence $\min\{|A'|,|B'|\}\ge \min\{|A|,|B|\}$. We use this inequality later.
    Observe that since $A$ contains $\ell$ of these sets, with $1\le \ell\le \frac{3t}{4}$, then $|A|\ge \ell \cdot n^\alpha$, and also $|B|\ge (t-\ell)n^\alpha \ge \frac{\ell}{3}\cdot n^\alpha$.
    Together, by Theorem \ref{isop_large sets_Q^d},
    \[
    |E_{Q^d}(A', B')| \ge \min\{|A'|,|B'|\}\ge \min\{|A|,|B|\}\ge  \frac{\ell \cdot n^{\alpha}}{3}.
    \]
    
    \begin{figure}[h]
    \centering
    \begin{tikzpicture}[scale=0.8]
    \draw (0,0) circle (1.2);
    \draw (0,0) circle (0.7);
    \node at (0,0) {$A$};
    \node at (0,1.0) {$A'$};
    
    \draw (5,0) circle (1.2);
    \draw (5,0) circle (0.7);
    \node at (5,0) {$B$};
    \node at (5,1.0) {$B'$};
    
    \draw (1.2,0.2) -- (3.8,0.2);
    \draw (1.2,0) -- (3.8,0);
    \draw (1.2,-0.2) -- (3.8,-0.2);
    \node at (2.5,0.5) {$\ge \frac{\ell \cdot n^{\alpha}}{3}$};
    \node at (2.5,-1) {$Q^d$};
    \end{tikzpicture}
    \end{figure}
    
    Also, as we have proven that whp in $G_1$ every vertex of $Q^d$ is at distance at most 2 from $L_1'$ (see Theorem \ref{lec8: lem8}), we thus have that every $v\in B'$ is at distance at most 2 from $B$, hence, every edge $e \in E_{Q^d}(A', B')$ is part of a path of length at most 5 in $Q^d$ between $A$ and $B$.
    Hence, the number of edge-disjoint paths in $Q^d$ between $A$ and $B$ is at least:
    \[
   \frac{|E(A',B')|}{5\cdot 5d^4+1}\ge \frac{\ell \cdot n^{\alpha}}{80 d^4}.
    \]
    Each such path is present in $G_1\cup G_2$ with probability at least $p_2^5 = \left(\frac{\delta}{d}\right)^5$. Thus, the expected number of open paths connecting $A$ and $B$ in $G=G_1\cup G_2$ is:
    \[
    \frac{\ell \cdot n^{\alpha}}{80 d^4} \cdot \left(\frac{\delta}{d}\right)^5 = \frac{\delta^5\cdot \ell \cdot n^{\alpha}}{80 d^9}.
    \]

    Using the standard concentration of the binomial distribution and the union bound over all partitions $V(L_1')=A\cup B$, we yield whp at least $\Omega\big(\frac{\ell\cdot n^\alpha}{d^9}\big)$ edge-disjoint paths of length at most 5 between $A$ and $B$ in $G$.
    For the second part of the lemma, we note that for any $v\in B$ has degree $d$ in $Q^d$, hence it is an endpoint of at most $d$ such paths.
\end{proof}

We are now close to completing the proof. To finish, we first prove that the vertices of $L_1'$ are close to one another (in $L_1$):
\begin{lemma}
    Whp, the distance in $L_1$ between any pair of vertices $u,v\in L_1'$ is $O(d^{12})$.
\end{lemma}
\begin{proof}
    For each partition of $L_1'$ as in Lemma \ref{lem: paths}, whp there are at least $\Omega\left(\frac{\ell \cdot n^{\alpha}}{d^{10}}\right)$ distinct endpoints in $B$ that correspond to at least $\Omega\left(\frac{\ell \cdot n^{\alpha}}{d^{10}}\right)$ edge-disjoint paths between $A$ and $B$. Now, since $|W_j|\le 2n^\alpha$, for every $1\le j\le t$, these endpoints must belong to at least:
    \[
    \frac{\Omega\left(\frac{\ell \cdot n^{\alpha}}{d^{10}}\right)}{2n^{\alpha}} = \Omega\left(\frac{\ell}{d^{10}}\right)
    \]
    distinct components $\{W_j\}$ in $B$. Look at $W_j\subseteq B$. Recall that the distance (in $L_1$) between any $u,v\in W_j$ is at most $4C_2 d$. Hence, there are
    \[
        \Omega\left(\frac{\ell}{d^{10}}\right)\cdot n^\alpha =\Omega\left(\frac{\ell\cdot n^\alpha}{d^{10}}\right)
    \]
    distinct vertices at distance at most $5+4C_2 d$ from $A$. Therefore.
    \[
        |B(A,5+4C_2 d)|\ge |A|\left(1+\Omega\left(\frac{1}{d^{10}}\right)\right).
    \]

    Now, let $v\in V(L_1')$. There is $1\le j\le t$ such that $v\in W_j$. Denote $A_0=W_j$ and $B_0=V(L_1')\setminus A_0$. Apply the above argument on the partition $V(L_1')=A_0\cup B_0$. Define $A_1\subseteq V(L_1')$ as the union of all $\{W_j\}$ at distance at most $5+4C_2 d$ from $A_0$. By the above,
    \[
        A_1\supseteq B(A,5+4C_2 d) \implies |A_1|\ge |A_0|\left(1+\Omega\left(\frac{1}{d^{10}}\right)\right).
    \]
    Iteratively define $A_i$ in such a way that $|A_{i+1}|\ge |A_i|\left(1+\Omega\left(\frac{1}{d^{10}}\right)\right)$ for every $i\in \mathbb{N}$. 
    
    After $O(d^{10}\cdot d)=O(d^{11})$ steps of the procedure we double the volume of $A_0$, and thus, since $|L_1'|=\Theta(n)$, we have to move $O(d^{11}\cdot \log_2n)=O(d^{12})$ steps to cover more than half of $L_1'$. Therefore, the distance (in $L_1$) between any two vertices from $L_1'$ is whp $O(d^{12})$.
\end{proof}

Now, all that remains is to show that typically we do not have a vertex in $L_1\setminus L_1'$ that is too far from $L_1'$. If we prove this, we are done, because for any pair of vertices $u,v\in L_1$ we can take vertices $u',v' \in L_1'$, such that $u$ is close to $u'$, and $v$ is close to $v'$, providing a short path from $u$ to $v$ by:
$$u\to u' \to v' \to v.$$

\begin{defn}
    Let $v\in V(L_1')$. We denote by $C_v$ the set of connected components of $L_1\setminus L_1'$ that connect to $v$ in $G_2$ (the set of components $C_i$ such that there exists a $v_i\in C_i$ with $(v_i,v)\in E(G_2)$).
\end{defn}

We have already proven that in $G_1$, whp, every connected component outside of $L_1'$ is of size at most $K_1 d$ for some $K_1\coloneqq K_1(\epsilon)>0$. We now prove the following lemma:

\begin{lemma}\label{small in L_1L_1'}
    There exists a constant $K_2\coloneqq K_2(\epsilon)>0$, such that whp, for every $v\in V(L_1')$ the set $C_v$ is of size at most $|C_v|\le K_2d$.
\end{lemma}

Observe that $$L_1=L_1'\cup \bigcup_{v \in V(L_1)} C_v.$$
Note that the sets $C_v$ are not necessarily disjoint. 
Assuming Lemma \ref{small in L_1L_1'}, we have that whp for every $u\in V(L_1)$, there exists $v\in V(L_1')$ such that $u\in C_i\in C_v$ for some component $C_i$ in $L_1\setminus L_1'$ that is connected to $v$ by an edge from $G_2$. Thus, the distance between $u$ and $L_1'$ is at most $|C_i|\le |C_v|$ steps, which the lemma argues is $O(d)$.
To conclude, the distance (in $L_1$) between any two vertices $u,v\in V(L_1)$ is $O(d)+O(d^{12})+O(d)=O(d^{12})$, completing the proof of Theorem \ref{diam_Q^d_p}.

We now prove Lemma \ref{small in L_1L_1'}.

\begin{proof}
Assume in contradiction that there exists a $v\in V(L_1)$ such that $|C_v|\ge K_2d$. Since all connected components in $L_1\setminus L_1'$ are of size at most $K_1d$, we can find a subset $\tilde{C}=\cup C_{i_j}$, of components from $C_v$ ($\tilde{C}$ is a set of vertices that form the union of some subset of components of $C_v$), such that:

$$K_2d \le |\tilde{C}|\le (K_1+K_2)d.$$

Denote the size of $\tilde{C}$ by $k$. We note that $\tilde{C}$ satisfies the following properties:
\begin{enumerate}
    \item $\tilde{C}\cup \{v\}$ is a connected graph in $G$ with $k+1$ vertices;
    \item All edges between $\tilde{C}$ and $V(Q^d)\setminus \tilde{C}$ are not present in $G_1$.
\end{enumerate}

We aim to prove that whp no such set exists. We note that these two conditions concern disjoint sets of edges and are thus independent. We bound the probability by:
\begin{enumerate}
    \item Applying the union bound over all possible sizes $k$ of $\tilde{C}$: $K_2d\le k\le (K_1+K_2)d$.
    \item Choosing a tree $T$ on $k+1$ vertices ($\le (k+1)(ed)^{k}$ possible spanning trees) and requiring its edges to be presented in $G$ ($p^k$, $k$ edges in $G$).
    \item Choosing a vertex $v$ in $T$ (in $k+1$ ways) and requiring all edges at the boundary of $T\setminus \{v\}$ to be closed in $G_1$.
\end{enumerate}
We obtain that the probability that such a set $\tilde{C}$ exists is bounded from above by:
\begin{align*}
    \sum_{k=K_2d}^{(K_1+K_2)d}n(k+1)(ed)^kp^k(1-p_1)^{k(d-2\log_2 k)}&\le n^2\sum_{k=K_2d}^{(K_1+K_2)d}\big(e(1+\epsilon)\big)^k e^{-\frac{1+\epsilon-\delta}{d}k(d-2\log_2 k)} \\
    &\le n^2\sum_{k=K_2d}^{(K_1+K_2)d}\exp\big(1+\epsilon-\frac{\epsilon^2}{3}-1-\epsilon+\delta+o(1)\big)^k \\
    &\le n^2\sum_{k=K_2d}^{(K_1+K_2)d}\exp\big(-\frac{\epsilon^2}{3}+\delta+o(1)\big)^k.
\end{align*}
    
Taking $K_2$ to be sufficiently large, and $\delta=\delta(\epsilon)>0$ sufficiently small, this probability tends to $0$, completing the proof.
\end{proof}

\newpage
\phantomsection
\addcontentsline{toc}{chapter}{Lecture 12}
\lecture{12}{July 12, 2026}{Lecture 12}{Binyamin kobzantsev}
\setcounter{chapter}{12}
\setcounter{section}{0}

\section{Hitting Time for Connectivity in a Random Cube}

\subsection{Background / Definitions}

\noindent \textbf{Notation:} $N = |E(Q^d)| = \frac{nd}{2}$, where $n = 2^d$.

\begin{defn}
Let $\sigma \in S_N$ be a permutation of $E(Q^d)$. We define a process $\tilde{Q} = \{Q_i\}_{i=0}^N$ on $Q^d$ by:
\begin{align*}
    V(Q_i) &= V(Q^d) = \{0,1\}^d; \\
    E(Q_i) &= \{e_{\sigma(1)}, \dots, e_{\sigma(i)}\}.
\end{align*}
\end{defn}

\begin{remark}
Note that:
\[
\overline{K_{2^d}}=Q_0 \subsetneq Q_1 \subsetneq \dots \subsetneq Q_N = Q^d
\]
are graphs on $V(Q^d)$, and $|E(Q_i)| = i$ for all $0 \le i \le N$.
\end{remark}

\begin{defn}
If, when defining a graph process on $Q^d$, the permutation $\sigma \in S_N$ is chosen uniformly at random, we say that the process is a \textit{random graph process} on $Q^d$.
\end{defn}

\begin{defn}
We say that $\mathcal{A}$ is an \textit{increasing (non-trivial) monotone property} of subgraphs of $Q^d$ if:
\begin{enumerate}
    \item $\mathcal{A} \subseteq \{G \subseteq Q^d\}$
    \item If $G_1 \in \mathcal{A}$ and $G_1 \subseteq G_2 \subseteq Q^d$, then $G_2 \in \mathcal{A}$.
    \item $Q_0 \notin \mathcal{A}$ and $Q_N = Q^d \in \mathcal{A}$.
\end{enumerate}
\end{defn}

\begin{defn}
Given a graph process $\tilde{Q} = \{Q_i\}_{i=0}^N$ and a monotone non-trivial property $\mathcal{A}$, we define the \textit{hitting time} of $\mathcal{A}$ with respect to $\tilde{Q}$ by:
\[
\tau_{\mathcal{A}}(\tilde{Q}) = \min \{ 0 \le i \le N \mid Q_i \in \mathcal{A} \}.
\]
Thus, $\tau_{\mathcal{A}}(\tilde{Q})$ is the first moment in the process where $\mathcal{A}$ holds.
\end{defn}

\begin{remark}
If $\sigma \in S_N$ is chosen randomly, then for any property $\mathcal{A}$, the hitting time $\tau_{\mathcal{A}}(\tilde{Q})$ is a random variable. One can therefore study its distribution, find its expectation, variance, etc.
\end{remark}

\begin{thm}[Bollobás \cite{zbMATH04162929}]
In a random graph process $\tilde{Q}$, whp:
\[
\tau_C(\tilde{Q}) = \tau_1(\tilde{Q}),
\]
where $\tau_C(\tilde{Q})$ is the hitting time for connectivity, and $\tau_1(\tilde{Q})$ is the hitting time for the disappearance of isolated vertices.
\end{thm}

\begin{remark}
Meaning in words: In a random graph process of the cube, whp, the graph becomes connected exactly at the time when the last isolated vertex disappears.
\end{remark}

\begin{remark}
For every graph process on $Q^d$, deterministically:
\[
\tau_C(\tilde{Q}) \ge \tau_1(\tilde{Q}).
\]
\end{remark}

\begin{proof}[Proof (Diskin, Krivelevich \cite{arXiv:2404.09289})]
It is enough to prove that there exists $t \coloneqq t(d)$ such that whp in a random graph process $\tilde{Q}$, the graph $Q_t$ satisfies:
\begin{enumerate}
    \item $\delta(Q_t) = 0$.
    \item The connected components of $Q_t$ are either isolated vertices or a giant component $L_1$ of size $|L_1| = 2^d(1 - o(1))$.
    \item For any two isolated vertices $u \neq v$ in $Q_t$, we have that $u$ and $v$ are not connected by an edge of $Q^d$.
\end{enumerate}

Indeed, at time $t$, there are isolated vertices and therefore $\tau_C(\tilde{Q}) \ge \tau_1(\tilde{Q}) > t$. Also, every edge touching an isolated vertex has its other endpoint in $L_1$. Therefore, adding any edge touching an isolated vertex $v$ in $Q_t$ connects $v$ to $L_1$. Hence, exactly at the moment when the last isolated vertex of $Q_t$ disappears, then this vertex and all other isolated vertices are connected to $L_1$, thus we get a connected graph. 

Due to monotonicity (similar to $G(n,p)$ and a random graph process of $K_n$), it is enough to prove the following claim:

\begin{claim}
Assume $p = \frac{1}{2} - \varepsilon$, where $\varepsilon > 0$ is a sufficiently small constant. Consider a random cube $Q_p^d$. Then, whp:
\begin{enumerate}[label=(\arabic*)]
    \item There are isolated vertices in $Q_p^d$.\label{prop1_12}
    \item $|L_1(Q_p^d)| = 2^d(1 - o(1))$, and any other component besides $L_1$ is an isolated vertex.\label{prop2_12}
    \item No two isolated vertices in $Q_p^d$ are adjacent in $Q^d$.\label{prop3_12}
\end{enumerate}
\end{claim}

\begin{proof}[Proof of \ref{prop1_12}]
Recall that $Q^d$ is a bipartite graph with parts $O, E$ of size $|O| = |E| = 2^{d-1}$. The random variable $X$ counts the number of isolated vertices in part $O$ of $Q_p^d$ and is distributed as $\operatorname{Bin}(2^{d-1}, (1-p)^d)$. Hence,
\[
\mathbb{E}[X] = 2^{d-1} \cdot (1-p)^d = 2^{d-1} \left(\frac{1}{2} + \varepsilon\right)^d = 2^{d-1} \frac{(1+2\varepsilon)^d}{2^d} = \frac{(1+2\varepsilon)^d}{2} \xrightarrow[d \to \infty]{} \infty.
\]
Therefore, whp $X > 0$, which implies that whp there are isolated vertices in $Q_p^d$.
\end{proof}
 
\begin{proof} [Proof of \ref{prop3_12}]
We want to show that whp, for every edge $(u,v) \in E(Q^d)$, it is not the case that both $u$ and $v$ are isolated in $Q_p^d$, i.e., $d_{Q_p^d}(u) = d_{Q_p^d}(v) = 0$ is false.

For any edge $e = (u,v) \in E(Q^d)$:
\[
\mathbb{P}\left(d_{Q_p^d}(u) = d_{Q_p^d}(v) = 0\right) = (1-p)^{2d-1} = \left(\frac{1}{2} + \varepsilon\right)^{2d-1} = \frac{(1+2\varepsilon)^{2d-1}}{2^{2d-1}}.
\]
Thus, the expected number of such pairs of adjacent (in $Q^d$) isolated vertices (in $Q^d_p$) is
\[
|E(Q^d)| \cdot \frac{(1+2\varepsilon)^{2d-1}}{2^{2d-1}} = \frac{2^d \cdot d}{2} \cdot \frac{(1+2\varepsilon)^{2d-1}}{2^{2d-1}} \xrightarrow[d \to \infty]{} 0,
\]
for a sufficiently small constant $\varepsilon > 0$. By Markov's inequality, whp there are no such pairs.
\end{proof}

\begin{proof}[Proof of \ref{prop2_12}]
We use two-round exposure (sprinkling). Let us define $p_2 = \varepsilon$, and define $p_1$ such that:
\[
1-p = (1-p_1)(1-p_2).
\]
This implies $p_1 \ge \frac{1}{2} - 2\varepsilon$.
Recall that $G = G_1 \cup G_2 \sim Q_p^d$ where $G_i \sim Q_{p_i}^d$ for $i=1,2$. Let $L'_1$ denote the giant component of $G_1$, and let $L_1$ denote the giant component of $G$.
We will show that whp in $G_1$ there is no component of size:
\[
k \in \left[2, n^{\frac{1}{3}}\right].
\]
Indeed, if there exists a component in $G_1$ of size $k$, then there exists a tree $T$ of size $k$ in $Q^d$ whose edges are contained in $G_1$, and there are no edges in $G_1$ between $V(T)$ and $V(Q^d) \setminus V(T)$.

Therefore, using (weak) Harper's isoperimetric inequality, the probability that there exists a component of size $k \in [2, n^{1/3}]$ is at most:
\begin{align*}
    \sum_{k=2}^{n^{1/3}} n(ed)^{k-1} (1-p_1)^{k(d - 2\log_2 k)} &\le \sum_{k=2}^{n^{1/3}} n(ed)^{k-1} \left(\frac{1}{2} + 2\varepsilon\right)^{k(d - 2\log_2 k)}\\
    &\le \sum_{k=2}^d n \left( ed \left(\frac{1}{2} + 2\varepsilon\right)^{d(1-o(1))} \right)^k + \sum_{k=d+1}^{n^{1/3}} n \left( ed \left(\frac{1}{2}+2\varepsilon\right)^{\frac{d}{3}} \right)^k\\
    &\le d \cdot n^{-\frac{1}{3}} + n^{\frac{1}{3}} \cdot n^{-\frac{d}{4}} = o(1),
\end{align*}
where the last inequality holds since $\epsilon>0$ is sufficiently small. Therefore, whp there are no components of size $k \in [2, n^{1/3}]$ in $G_1$. Furthermore,
\[
\mathbb{P}(d_{G_1}(v) = 0) = (1-p_1)^d \le \left(\frac{1}{2} + 2\varepsilon\right)^d \le n^{-1 + 9\varepsilon},
\]
which implies:
\[
\mathbb{E}[\text{number of isolated vertices in } G_1] \le n^{9\varepsilon}.
\]
By Markov's inequality, whp there are at most $n^{10\varepsilon}$ isolated vertices in $G_1$. Now, we will use the edges of $G_2$ to connect the large components of $G_1$. Let us define:
\[
W = \{ v \in V(Q^d) \mid v \text{ belongs to a component of size } \ge n^{\frac{1}{3}} \text{ in } G_1 \}.
\]
Then, whp:
\[
|W| \ge n - n^{10\varepsilon}.
\]

\begin{figure}[htbp]
\centering
\begin{tikzpicture}[scale=0.9]
    \draw[thick] (0,0) ellipse (2.5cm and 1.8cm);
    \node[above] at (0,1.8) {$W$};
    
    \draw[thick] (-1,0.5) circle (0.4cm) node {O};
    \draw[thick] (1,0.7) circle (0.4cm) node {O};
    \draw[thick] (0,-0.8) circle (0.5cm) node {O};
    
    \draw[thick] (4,1) circle (0.2cm) node {\footnotesize v};
    \draw[thick] (4,-0.5) circle (0.2cm);
    \draw[thick] (3.5,-1.5) circle (0.2cm);
    
    \node[right] at (4.3, 0) {isolated vertices in $G_1$};
    
    \draw[dashed, ->, >=stealth, thick] (1,0.7) -- (3.8,1) node[midway, above] {$G_2$ edges};
    \draw[dashed, ->, >=stealth, thick] (0,-0.8) -- (3.8,-0.5);
    \draw[dashed, ->, >=stealth, thick] (0,-0.8) -- (3.3,-1.4);
    
\end{tikzpicture}
\caption{The set $W$ contains components of size $\ge n^{1/3}$ in $G_1$. Some edges from $G_2$ connect these components to the isolated vertices outside $W$.}
\end{figure}
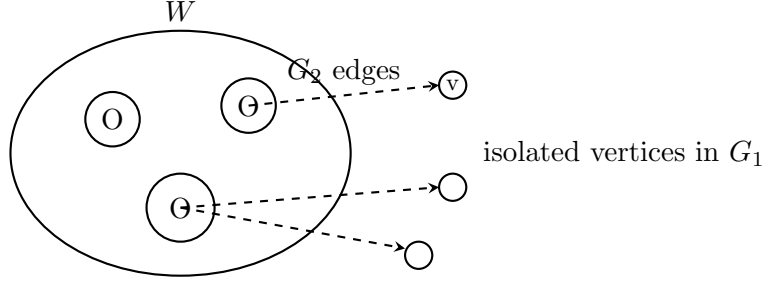

If $W$ does not merge into one component in $G$, then there exists a partition $W = A \cup B$ where $A, B \neq \emptyset$, with no connecting edge between $A$ and $B$ in $G_2$. We can assume $a = |A| \le |B|$. Notice that
\[
|E_{Q^d}(A, A^c)| \ge |A|=a.
\]
Some of these edges leaving $A$ in $Q^d$ can lead to $A^c \setminus B$. Since there are at most $n^{10\varepsilon}$ vertices outside of $W$, at most $n^{10\varepsilon}d$ such edges do not touch $B$. Hence, since $a \ge n^{1/3}$, the number of edges in $Q^d$ between $A$ and $B$ is at least:
\[
|E_{Q^d}(A, B)| \ge |A| - n^{10\varepsilon} d \ge \frac{a}{2}.
\]
Therefore:
\[
\mathbb{P}(E_{G_2}(A, B) = \emptyset) \le (1-p_2)^{\frac{a}{2}} = (1-\varepsilon)^{\frac{a}{2}}.
\]
Denote by $s$ the number of components of size at least $n^{1/3}$. Obviously, $s\le n$. Thus, the probability that $W$ does not connect into a single component in $G=G_1\cup G_2$ is whp at most

\[
\sum_{a=n^{1/3}}^{n/2} \left( \sum_{i=1}^{a/n^{1/3}} \binom{s}{i} \right) \cdot (1-\varepsilon)^{\frac{a}{2}} \le \sum_{a=n^{1/3}}^{n/2} s^{\frac{a}{n^{1/3}}} \cdot e^{-\frac{\varepsilon a}{2}}
\le \sum_{a=n^{1/3}}^{n/2} \left( n^{\frac{1}{n^{1/3}}}\cdot  e^{-\frac{\varepsilon }{2}} \right)^a = o(1).
\]
Therefore, whp in $G = G_1 \cup G_2$, the vertices of $W$ merge into a single component. Recall that whp, for every edge in $Q^d$ incident to a vertex isolated in $G_1$, its other endpoint is in $W$ (apply \ref{prop3_12} to $G_1$). Thus, after exposing $G_2$, we obtain whp a single giant component $L_1$ in $G$, and, outside of it, only isolated vertices.
\end{proof}
\end{proof}

\bibliographystyle{abbrv}
\bibliography{Bib}

\end{document}